\documentclass{amsart}
\pdfoutput=1

\usepackage[english]{babel}
\usepackage[utf8]{inputenc}
\usepackage[T1]{fontenc}
\usepackage{amsmath, latexsym, amssymb, amsfonts, amsthm}
\usepackage{stmaryrd}

\usepackage{enumitem}
\usepackage{url}
\usepackage[round, semicolon, authoryear, sort&compress, longnamesfirst]{natbib}
\usepackage[colorlinks=true, allcolors=blue]{hyperref}
\usepackage{xspace}

\usepackage{tikz}
\usetikzlibrary{arrows.meta}
\usetikzlibrary{calc}
\pgfmathsetmacro{\ec}{0.14} %edge constant
\pgfmathsetmacro{\cc}{1.5*\ec} %corner constant
\pgfmathsetmacro{\wc}{0.5*1.732*\ec} %width constant

\usepackage{caption}
\usepackage{subcaption}
\newtheorem{thm}{Theorem}[section]
\newtheorem{lem}[thm]{Lemma}
\newtheorem{cor}[thm]{Corollary}
\newtheorem{prop}[thm]{Proposition}

\theoremstyle{definition}
\newtheorem{defn}[thm]{Definition}
\newtheorem{ex}[thm]{Example}

\theoremstyle{remark}

\newcommand{\R}{{\mathbb R}}
\newcommand{\Z}{{\mathbb Z}}
\newcommand{\N}{{\mathbb N}}
\renewcommand{\P}{{\mathbb P}} % For poset indexing a Morse decomposition
\renewcommand{\S}{{\mathbb S}} % For poset indexing a Morse decomposition

\newcommand{\cB}{{\mathcal B}}
\newcommand{\cC}{{\mathcal C}}

\newcommand{\cM}{{\mathcal M}}

\newcommand{\cR}{{\mathcal R}}

\newcommand{\cV}{{\mathcal V}}

\DeclareMathOperator{\cl}{cl}
\DeclareMathOperator{\Cl}{Cl}
\DeclareMathOperator{\exit}{ex}   % for exit set
\DeclareMathOperator{\Ex}{Ex}   % for exit set
\DeclareMathOperator{\Fix}{Fix}   % for fixed points
\DeclareMathOperator{\Sol}{Sol}   % for sets of solutions (see Mrozek 2017)
\DeclareMathOperator{\Con}{Con}

\DeclareMathOperator{\id}{id}
\DeclareMathOperator{\image}{im} % image
\DeclareMathOperator{\card}{card}
\DeclareMathOperator{\dom}{dom}
\DeclareMathOperator{\sbdy}{bd} % boundary

\newcommand{\mdm}{\texttt{mdm}\xspace}

\newcommand{\maxdim}{k} % dimension of range of mdm functions
\newcommand{\rightconnects}[1]{\rightarrowtriangle_{#1}}
\newcommand{\leftconnects}[1]{\leftarrowtriangle_{#1}}
\newcommand{\leftrightconnects}[1]{\leftrightarrowtriangle_{#1}} %
\newcommand{\nrightconnects}[1]{\Arrownot\rightarrowtriangle_{#1}}
\newcommand{\nleftconnects}[1]{\Arrownot\leftarrowtriangle_{#1}}
\newcommand{\ccomponents}{\mathcal{C}/{\sim}}

\setlist[enumerate,1]{label=$\mathrm{(\arabic{enumi})}$}

\begin{document}

	\title{Multiparameter Discrete {Morse} Theory}

	\author{Guillaume Brouillette}
	\author{Madjid Allili}
	\author{Tomasz Kaczynski}
	\address[G.~Brouillette, T.~Kaczynski]{D{\'e}partement de math{\'e}matiques, Universit{\'e} de Sherbrooke, boul. de l'Universit{\'e}, Sherbrooke, J1K 2R1, QC, Canada}
	\address[M.~Allili]{Departments of Computer Science and Mathematics, Bishop's University, College Street, Sherbrooke, J1M 1Z7, QC, Canada}
	\email[G.~Brouillette]{Guillaume.Brouillette2@USherbrooke.ca}
	\email[M.~Allili]{mallili@ubishops.ca}
	\email[T.~Kaczynski]{Tomasz.Kaczynski@USherbrooke.ca}

	\thanks{G.B. was supported by the Natural Sciences and Engineering Council of Canada (NSERC) under grant number 569623 and the Fonds de recherche du Qu{\'e}bec -- Nature et technologies (FRQNT) under grant number 289126. T.K. was supported by a Discovery Grant from the NSERC.}

	\begin{abstract}
		The main objective of this paper is to extend Morse-Forman theory to vector-valued functions. This is mostly motivated by the need to develop new tools and methods to compute multiparameter persistence. To generalize the theory, in addition to adapting the main definitions and results of Forman to this vectorial setting, we use concepts of combinatorial topological dynamics studied in recent years. This approach proves to be successful in the following ways. First, we establish a result which is more general than that of Forman regarding the sublevel sets of a multidimensional discrete Morse function. Second, we find a way to induce a Morse decomposition in critical components from the critical points of such a function. Finally, we deduce a set of Morse equation and inequalities specific to the multiparameter setting.
	\end{abstract}

	\maketitle

	\section{Introduction}\label{sec:introduction}

	Discrete Morse theory (DMT) introduced by \citet{Forman1998, Forman2002} has proven to be extremely useful in a panoply of applications where the topological processing of data
	is a key ingredient. Many domains such as visualization, molecular biology, computer vision, computational geometry, to name but a few, that rely on point cluster generation and
	meshing techniques have already used Forman's theory very successfully. Moreover, this theory has become central in the emerging and fast-growing field called topological data
	analysis (TDA) which aims at providing efficient topological and geometrical tools to extract and organize relevant qualitative information about given data. DMT can be used directly
	for discrete data processing or as a procedure that simplifies and reduces the computation of \emph{Persistent Homology} (PH), another very popular and very efficient tool used in TDA.

	In this paper, we focus exclusively on the discrete Morse theory developed by Forman since the work undertaken is based to some extent on the ideas from~\citet{Allili2019},
	which were developed and expressed using the Forman framework. Nonetheless, there are alternative discrete Morse theories available with equivalent potential than the Forman's approach~\citep{Fugacci2020, Scoville2019}.
	One version that has gained prominence in literature is the piecewise-linear (PL) Morse theory introduced by \citet{Banchoff1967}. Recent additions to this approach made notably by \citet{Bloch2013} and \citet{Grunert2023} provide an insight about the possibility of using it to develop a framework for persistent homology of spaces and maps similar to the one achieved by Forman's theory.

	PH has been introduced in~\citep{Edelsbrunner2002} as a tool for the analysis of homologies of spaces which can be given in terms of meshes or cluster of points by means of a filtration
	defined on the space. This theory constitutes a reliable and efficient method to track the evolution of topological features in data. Its key properties that make it important in TDA are
	the robustness to noise, the independence of dimension, and its computability and algorithmic framework.

	In the standard setting, PH is defined on a nested growing sequence, often called a \emph{filtration}, of sublevel sets of a function $f: |S| \rightarrow \mathbb{R}$ where $S$ is typically a
	simplicial complex built from data and $f$ is a continuous function on $|S|$, usually called a filtering function. It provides topological invariants of the filtration known as barcodes or
	persistence diagrams which measure the persistence of topological features of the data at different resolutions and scales in the space as encoded by the filtration. In this context, the
	filtration is indexed by a totally ordered set $T$ and hence it is called a single-parameter filtration and it gives rise to the single-parameter PH. However, in many applications
	the data is best described using multiple parameters and using single-parameter filtrations can result in missing important information about the data. This is particularly the case for spatial
	complex and heterogeneous noisy data whose topological structure can depend on several parameters such as the scale, the density, the presence of outliers and other artifacts
	\citep{Vipond2021, xia-wei2015}.

	These shortcomings triggered the development of multiparameter persistence homology (MPH) defined on a multiparameter filtration of the sublevel sets of a
	one-parameter family of functions $f_t : |S| \rightarrow \mathbb{R}$ \citep{Carlsson2009}. Typically, the parameter $t$ is taken in some continuous interval such as $[0,1]$. However when the
	number of used parameters is finite, the family of functions can be replaced by a function $f: |S| \rightarrow \mathbb{R}^k$, for some positive integer $k$.
	The topological information provided by MPH can be encoded as multiparameter persistence modules~\citep{Carlsson2010} which do not possess a simple representation comparable to that of
	the persistence diagrams in the PH case. The rank invariant introduced by~\citet{Carlsson2010} is an alternative invariant that contains the same information as the persistence diagram in
	the one parameter case. Software for visualizing the rank invariant of the two-parameter persistent homology (RIVET) is provided by \citet{Lesnick2015}. There are also efforts to find invariants that
	can be combined with statistical and machine learning tools such as the persistent landscapes~\citep{Bubenik2015,Vipond2018}. However the extraction of multiparameter
	persistent information remains a hard task in the general case and the existing methods for the computation of MPH are computationally expensive due to the considerable size of
	complexes built from data.

	One direction explored in some recent works consists in designing algorithms to reduce the original complexes generated from data to enough smaller cellular complexes, homotopically equivalent
	to the initial ones by means of {\em acyclic partial matchings} of discrete Morse theory. This approach used initially for the one-parameter filtrations has been extended to the multiparameter
	case for the first time by~\citet{Allili2017}. More efficient algorithms based on a similar idea are obtained by~\citet{Allili2019, Scaramuccia2020}. Even though the designed algorithms make use of the idea of
	discrete Morse pairings, the works did not provide a systematic extension of the Forman's discrete Morse theory to the multiparameter case although a serious attempt was made by~\citet{Allili2019}
	to achieve this goal. Indeed, new definitions of a multidimensional discrete Morse function, of its gradient field, its regular and critical cells are proposed. Moreover, it was proved that given a filtering
	function $f: |S| \rightarrow \mathbb{R}^k$, there exists a multidimensional discrete Morse (\mdm) function $g$ with the same order of sublevel sets and the same acyclic partial matching as the one associated with $f$.

	In this paper, the combinatorial vector fields framework is used to further develop the concept and the properties of the \mdm theory. Many notions of the classical discrete Morse theory are extended to the \mdm case. The relationship between a \mdm function and its components functions are investigated and the handle decomposition and collapsing theorems are established. Moreover, results on Morse inequalities and Morse decompositions are proved for the first time. An additional contribution achieved in this work consists of a method that allows to partition critical cells of a \mdm function into connected critical components. It is known from smooth singularity theory that the criticalities of smooth vector-valued functions are generally sets and not insolated points. Experimentations~\citep{Allili2019} suggest that it is also the case for a \mdm function
	where each criticality is given as a component that may consist of several cells. We refer the reader to \citep{Budney2021,Smale1975,Wan1975} and references therein for the classical singularity theory setting. In particular, the recent work of \citet{Budney2021} has been largely motivated by the call coming from \citet{Allili2019} for providing an adequate application-driven smooth background and geometric insight that would help us in understanding the discrete counterpart.

	The partition of critical cells in components proposed in this work is a first step in this direction with the goal of linking \mdm theory to the smooth singularity theory mentioned above and the piecewise linear setting \citep{Edelsbrunner2008a,Huettenberger2014}, in which criticalities of vector-valued functions also appear in the form of sets.

	The paper is organized as follows. In Section~\ref{sec:preliminaries}, we recall useful definitions and terminology about simplicial complexes and combinatorial vector fields.
	Section~\ref{sec:CombinatorialDynamics} is devoted to introduce and discuss notions of combinatorial dynamics on simplicial complexes which provide a framework to represent
	discrete vector fields and discrete Morse functions in terms of discrete dynamical systems and flows. This allows to associate to combinatorial vector fields
	the concepts of isolated invariant sets, Conley index, Morse decompositions and Morse inequalities which is achieved in Corollary~\ref{coro:MorseInequalitiesMorseDecomposition}.
	Some new results about Morse decompositions are also discussed. In Section~\ref{sec:MDM},
	we build on preliminary results of~\citet{Allili2019} and define multidimensional, or multiparameter, discrete Morse functions and outline many of their properties. In the sections that follow, many classical
	results of Forman's theory for real-valued functions are extended to vector-valued functions. One of the difficulties in the study of \mdm functions is the classification of their criticalities.
	Unlike real-valued discrete Morse functions for which each criticality is represented by a single cell, experimentations~\citep{Allili2019} and the smooth theory of singularity both
	suggest that a criticality of a \mdm function is a component that may contain several cells. In Section~\ref{sec:CriticalComponents}, a method to group critical cells to form critical
	components is proposed in Definition~\ref{def:CriticalComponents}. This leads to Theorem~\ref{theo:MorseDecompositionfCycle} on Morse decompositions and acyclicity. Our final result concerning
	Morse inequalities is stated in Theorem~\ref{theo:MorseInequalitiesCritComponents}.

	Concluding remarks and future work directions are proposed at the end of the paper.

	\section{Preliminaries}\label{sec:preliminaries}
	\subsection{Maps and relations}
	Consider two sets $X$ and $Y$. A \emph{partial map} $f:X\nrightarrow Y$ is a function whose \emph{domain} $\dom f$ is a subset of $X$. We note $\image f := f(X)$ the \emph{image} of $f$ and $\Fix f := \left\lbrace x\in\dom f\ |\ f(x) = x\right\rbrace$ the set of fixed points of $f$.

	Moreover, a \emph{multivalued map} $F:X\multimap Y$ is a function which associates each $x\in X$ to a non-empty subset $F(x)\subseteq Y$. For every $y\in Y$, we write $F^{-1}(y):=\left\lbrace x\in X\ |\ y\in F(x) \right\rbrace$ and for subsets $A\subseteq X$ and $B\subseteq Y$, we define $F(A):= \bigcup_{x\in A}F(x)$ and $F^{-1}(B) := \bigcup_{y\in B}F^{-1}(y)$.

	Furthermore, for a binary relation $R\subseteq X\times X$, we write $xRy$ when $(x,y)\in R$. We define the \emph{transitive closure} $\bar{R}\subseteq X\times X$ of $R$ as the relation such that $x\bar{R}y$ if there exists a sequence $x=x_0,x_1,...,x_n=y$ in $X$ such that $n\geq 1$ and $x_{i-1}Rx_i$ for each $i=1,...,n$. The relation $\bar{R}\cup\id_X$, where $\id_X$ is the identity relation on $X$, is both reflexive and transitive, making it a preorder, which we call the \emph{preorder induced by $R$}. Note that the preorder induced by a reflexive relation is simply its transitive closure.

	\subsection{Simplicial complexes}
	In the context of this article, we consider a \emph{simplicial complex} $K$ to be a finite collection of non-empty finite sets such that for all $\sigma\in K$ and $\tau\subseteq\sigma$, we have $\tau\in K$. An element of a simplicial complex is called a \emph{simplex}. The \emph{dimension} of a simplex $\sigma$ is $\dim\sigma = \card\sigma-1$, and we note $K_p$ the set of simplices in $K$ of dimension $p$. The superscript $\sigma^{(p)}$ is sometimes used to specify a simplex $\sigma$ is of dimension $p$.

	Furthermore, if $\tau\subseteq\sigma\in K$, we say that $\tau$ is a \emph{face} of $\sigma$ and $\sigma$ a \emph{coface} of $\tau$. If, in addition, we have $\dim\tau=\dim\sigma-1$, then $\tau$ is said to be a \emph{facet} of $\sigma$ and, conversely, $\sigma$ a \emph{cofacet} of $\tau$.

	There are many ways to endow an abstract simplicial complex with a topology. When considering a simplicial complex as a combinatorial or discrete object, the Alexandrov topology is quite convenient \citep{McCord1966, Stong1966}. In this topology, the \emph{closure} of a set of simplices $A\subseteq K$, which we note $\Cl A$, is the set of all faces of all simplices in $A$. We also call $\Ex A := \Cl A\backslash A$ the \emph{exit set} of $A$. Otherwise, since a simplicial complex is a particular case of a CW-complex, we may also identify each simplex $\sigma^{(p)}\in K$ with a cell of dimension $p$, which is homeomorphic to an open ball, in a Hausdorff space. In practice, this Hausdorff space is generally $\R^d$ and each $p$-cell is the convex hull of $p+1$ affinely independent points. This point of view will prove itself to be particularly useful to generalize some of Forman's classical results. We then note $\cl A$ and $\exit A := \cl A\backslash A$ respectively the \emph{closure} and \emph{exit set} of $A$ in $K$ considered as a CW-complex. When using the operators $\Cl$, $\Ex$, $\cl$ or $\exit$ with a singleton, we omit the braces.

	We call $\sigma\in K$ a \emph{free face} of a simplicial complex $K$ if it has a unique cofacet $\tau\supset\sigma$. When $K$ has a free face $\sigma$ with cofacet $\tau\supset\sigma$, then we call an \emph{elementary collapse} the operation of removing $\sigma$ and $\tau$ to obtain a smaller subcomplex $K\backslash\{\sigma,\tau\}$. We say that $K$ \emph{collapses onto} a subcomplex $L$, noted $K\searrow L$, if $L$ can be obtained from $K$ by doing a sequence of elementary collapses. Collapsing a simplicial complex onto a subcomplex may be seen as a deformation retraction. More precisely, if $K\searrow L$, then $L$ is a deformation retract of $K$, thus $K$ and $L$ are \emph{homotopy equivalent} spaces when endowed with the topology of CW complexes, and we note $K\simeq L$.

	\subsection{Combinatorial vector fields}
	We now introduce the concept of discrete vector fields on simplicial complexes. They were first used within the framework of discrete Morse theory by \citet{Forman1998, Forman2002}, who defined them as collections of pairs of simplices. Here, we use the definition proposed by \citet{Batko2020, Kaczynski2016}, which is better suited in our context. Also, note that a discrete vector field is a particular case of a multivector field, as defined by \citet{Mrozek2017}.

	\begin{defn}[Discrete vector field]
		A \emph{discrete vector field}, or a \emph{combinatorial vector field}, on a simplicial complex $K$ is an injective partial self-map $\cV:K\nrightarrow K$ such that
		\begin{enumerate}
			\item for each $\sigma\in\dom\cV$, either $\cV(\sigma) = \sigma$ or $\cV(\sigma)$ is a cofacet of $\sigma$;
			\item $\dom\cV\cup\image\cV = K$;
			\item $\dom\cV\cap\image\cV = \Fix\cV$.
		\end{enumerate}
	\end{defn}

	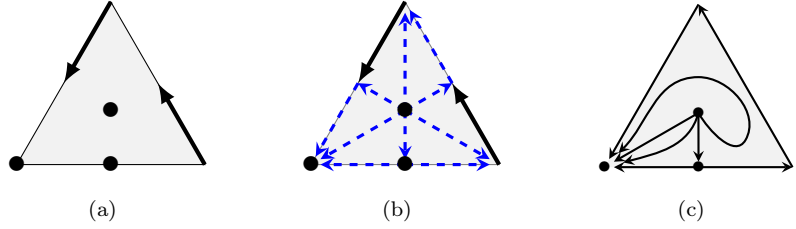
\begin{figure}
		\centering
		\begin{subfigure}[b]{0.3\textwidth}
			\centering
			\begin{tikzpicture}[scale=1.25]
				\coordinate (D) at (0,0);
				\coordinate (A) at (1,1.732);
				\coordinate (E) at (2,0);

				\coordinate (AD) at ($0.5*(A)+0.5*(D)$);
				\coordinate (AE) at ($0.5*(A)+0.5*(E)$);
				\coordinate (DE) at ($0.5*(D)+0.5*(E)$);

				\coordinate (ADE) at ($0.333*(A)+0.333*(D)+0.333*(E)$);

				\fill[black!5] (A) -- (E) -- (D) -- cycle;

				\draw (A) -- (E) -- (D) -- cycle;

				\node at (D){\LARGE $\bullet$};
				\node at (DE){\LARGE $\bullet$};
				\node at (ADE){\LARGE $\bullet$};

				\draw[ultra thick, -latex] (A) -- (AD);
				\draw[ultra thick, -latex] (E) -- (AE);

				%\node[xshift=11pt, yshift=-6pt] at (A) {$A$};
				%\node[xshift=11pt, yshift=6pt] at (D) {$B$};
				%\node[xshift=11pt, yshift=6pt] at (E) {$C$};

			\end{tikzpicture}
			\caption{}\label{fig:DiscreteFieldComb}
		\end{subfigure}
		\begin{subfigure}[b]{0.3\textwidth}
			\centering
			\begin{tikzpicture}[scale=1.25]
				\coordinate (D) at (0,0);
				\coordinate (A) at (1,1.732);
				\coordinate (E) at (2,0);

				\coordinate (AD) at ($0.5*(A)+0.5*(D)$);
				\coordinate (AE) at ($0.5*(A)+0.5*(E)$);
				\coordinate (DE) at ($0.5*(D)+0.5*(E)$);

				\coordinate (ADE) at ($0.333*(A)+0.333*(D)+0.333*(E)$);

				\fill[black!5] (A) -- (E) -- (D) -- cycle;

				\draw[black!50] (A) -- (E) -- (D) -- cycle;

				\draw[very thick, blue, dashed, -stealth] (ADE) -- ($(A)!0.1!(ADE)$);
				\draw[very thick, blue, dashed, -stealth] (ADE) -- (AD);
				\draw[very thick, blue, dashed, -stealth] (ADE) -- ($(D)!0.1!(ADE)$);
				\draw[very thick, blue, dashed, -stealth] (ADE) -- ($(DE)!0.1!(ADE)$);
				\draw[very thick, blue, dashed, -stealth] (ADE) -- ($(E)!0.1!(ADE)$);
				\draw[very thick, blue, dashed, -stealth] (ADE) -- (AE);
				\draw[very thick, blue, dashed, -stealth] (DE) -- ($(D)!0.1!(DE)$);
				\draw[very thick, blue, dashed, -stealth] (DE) -- ($(E)!0.1!(DE)$);
				\draw[very thick, blue, dashed, -stealth] (AD) -- ($(D)!0.1!(AD)$);
				\draw[very thick, blue, dashed, -stealth] (AE) -- ($(A)!0.1!(AE)$);

				\node at (D){\LARGE $\bullet$};
				\node at (DE){\LARGE $\bullet$};
				\node at (ADE){\LARGE $\bullet$};

				\draw[ultra thick, -latex] (A) -- (AD);
				\draw[ultra thick, -latex] (E) -- (AE);

				%\node[xshift=11pt, yshift=-6pt] at (A) {$A$};
				%\node[xshift=11pt, yshift=6pt] at (D) {$B$};
				%\node[xshift=11pt, yshift=6pt] at (E) {$C$};

			\end{tikzpicture}
			\caption{}\label{fig:DiscreteFieldFlow}
		\end{subfigure}
		\begin{subfigure}[b]{0.3\textwidth}
			\centering
			\begin{tikzpicture}[scale=1.25]
				\coordinate (D) at (0,0);
				\coordinate (A) at (1,1.732);
				\coordinate (E) at (2,0);

				\coordinate (AD) at ($0.5*(A)+0.5*(D)$);
				\coordinate (AE) at ($0.5*(A)+0.5*(E)$);
				\coordinate (DE) at ($0.5*(D)+0.5*(E)$);

				\coordinate (ADE) at ($0.333*(A)+0.333*(D)+0.333*(E)$);

				\fill[black!5] (A) -- (E) -- (D) -- cycle;

				\node at (D){$\bullet$};
				\node at (DE){$\bullet$};
				\node at (ADE){$\bullet$};

				\draw[thick, -stealth] (A) -- ($(D)!0.1!(AD)$);
				\draw[thick, -stealth] (E) -- (A);
				\draw[thick, -stealth] (DE) -- ($(D)!0.1!(DE)$);
				\draw[thick, -stealth] (DE) -- (E);
				\draw[thick, -stealth] (ADE) -- ($(DE)!0.1!(ADE)$);
				\draw[thick, -stealth] (ADE) -- ($(D)!0.1!(ADE)$);
				\draw[thick, -stealth] (ADE) to[out=255, in=15] ($(D) + (0.2,0.05)$);
				\draw[thick, -stealth] (ADE) to[out=-60, in=240] ($(ADE)!0.5!(E)$) to[out=60, in=0] ($(ADE)!0.33!(A)$) to[out=180, in=45] ($(D) + (0.15,0.15)$);

				%\node[xshift=11pt, yshift=-6pt] at (A) {$A$};
				%\node[xshift=11pt, yshift=6pt] at (D) {$B$};
				%\node[xshift=11pt, yshift=6pt] at (E) {$C$};

			\end{tikzpicture}
			\caption{}\label{fig:DiscreteFieldCont}
		\end{subfigure}
		\caption{In \subref{fig:DiscreteFieldComb}, a discrete vector field $\cV$. The dots identify the fixed points of $\cV$, while the arrows show the pairs of simplices $(\sigma,\tau)$ such that $\cV(\sigma) = \tau$. In \subref{fig:DiscreteFieldFlow}, the induced flow $\Pi_\cV$. In \subref{fig:DiscreteFieldCont}, a continuous flow on the underlying space which mimics the dynamics of the combinatorial flow}\label{fig:DiscreteField}
	\end{figure}

	For some discrete vector field $\cV$ on $K$, we call a \emph{$\cV$-path} a sequence
	\begin{align*}
		\alpha_0^{(p)},\beta_0^{(p+1)},\alpha_1^{(p)},\beta_1^{(p+1)},\alpha_2^{(p)},...,\beta_{n-1}^{(p+1)},\alpha_n^{(p)}
	\end{align*}
	of simplices in $K$ such that $\alpha_i\in\dom\cV$, $\cV(\alpha_i) = \beta_i$ and $\beta_i\supset\alpha_{i+1}\neq\alpha_i$
	for each $i=0,...,n-1$. A $\cV$-path is \emph{closed} if $\alpha_0=\alpha_n$ and \emph{nontrivial} if $n\geq 1$.
	A discrete vector field $\cV$ is said to be \emph{acyclic} if there is no nontrivial closed $\cV$-path.

	Finally, we define the notion of $\cV$-compatibility as introduced by \citet{Mrozek2017} using the notation of \citet{Batko2020, Kaczynski2016}. Note that this is unrelated to the compatibility of a discrete field with a multifiltration as defined by \citet{Scaramuccia2020}.

	\begin{defn}[$\cV$-compatibility]
		Let $\cV$ be a combinatorial vector field on a simplicial complex $K$. We say that $A\subseteq K$ is \emph{$\cV$-compatible} if for all $\sigma\in K$,
		we have $\sigma^-\in A\Leftrightarrow\sigma^+\in A$ where

		\begin{gather*}
			\sigma^+ := \begin{cases}
				\cV(\sigma) &\text{ if } \sigma\in\dom\cV\\
				\sigma &\text{ otherwise}
			\end{cases}
			\qquad \text{and}\qquad
			\sigma^- := \begin{cases}
				\sigma &\text{ if } \sigma\in\dom\cV\\
				\cV^{-1}(\sigma) &\text{ otherwise}
			\end{cases}.
		\end{gather*}
	\end{defn}

	For all $\sigma\in K$, can see that $\sigma^-=\sigma=\sigma^+$ when $\sigma\in\Fix\cV$, $\sigma^-\subsetneq\sigma=\sigma^+$ when
	$\sigma\in\image\cV\backslash\Fix\cV$ and $\sigma^-=\sigma\subsetneq\sigma^+$ when $\sigma\in\dom\cV\backslash\Fix\cV$.

	\section{Combinatorial dynamics}\label{sec:CombinatorialDynamics}

	Many concepts of dynamical systems theory, notably elements of Conley index theory \citep{Conley1978}, can defined in the combinatorial setting. In this section, we recall such notions, which were mainly discussed in \citep{Batko2020, Kaczynski2016, Mrozek2017}. We also prove some new results, notably Propositions \ref{prop:BasicSetsAcyclicField} and \ref{prop:AcyclictyFieldEquiFlow}, regarding acyclic combinatorial vector fields and their induced flows, and Theorem \ref{theo:ConditionsForMorseDecomposition}.

	\subsection{Flows}

	We define the flow induced by a combinatorial vector field as in \citep{Batko2020,Kaczynski2016}. It is worth noting that this definition coincides
	with the one given in \citep{Mrozek2017} for combinatorial multivector fields when applied to vector fields.

	\begin{defn}[Flow associated to a discrete vector field]\label{def:flow}
		Given a combinatorial vector field $\cV$ on a simplicial complex $K$, the associated \emph{flow} $\Pi_\cV$ is the multivalued map $\Pi_\cV:K\multimap K$ such that
		\begin{align*}
			\Pi_\cV(\sigma) =
			\begin{cases}
				\Cl\sigma & \text{ if } \sigma\in\Fix\cV, \\
				\Ex\sigma\backslash\{\cV^{-1}(\sigma)\} & \text{ if } \sigma\in\image\cV\,\backslash\Fix\cV, \\
				\{\cV(\sigma)\} & \text{ if } \sigma\in\dom\cV\,\backslash\Fix\cV.
			\end{cases}
		\end{align*}
	\end{defn}

	Notice that $\Pi_\cV(\sigma) = \{\cV(\sigma)\}$ when $\sigma\in\dom\cV\backslash\Fix\cV$ and $\Pi_\cV(\sigma)\subseteq\Cl\sigma$ otherwise.

	A \emph{solution} of a flow $\Pi_\cV$ is a partial map $\varrho:\Z\nrightarrow K$ such that $\dom\varrho$ is an interval of $\Z$ and, whenever $i,i+1\in\dom\varrho$, we have $\varrho(i+1)\in\Pi_\cV(\varrho(i))$. A solution is \emph{full} when $\dom\varrho=\Z$. We note $\Sol(\sigma,A)$ the set of full solutions $\varrho:\Z\rightarrow A$ for which $\sigma\in\image\varrho$.

	Moreover, a solution $\varrho$ with $\dom\varrho = \{m,m+1,...,m+n\}$ is \emph{nontrivial} if $n \geq 1$ and it is \emph{closed} if $\varrho(m) = \varrho(m+n)$. We note $\sigma\rightconnects{\cV}\tau$ or $\tau\leftconnects{\cV}\sigma$ if there exists a nontrivial solution going from $\sigma$ to $\tau$. Similarly, for $A,B\subset K$, we write
	\begin{itemize}[label=$\bullet$]
		\item $A\rightconnects{\cV}\tau$ if $\sigma\rightconnects{\cV}\tau$ for some $\sigma\in A$;

		\item $\sigma\rightconnects{\cV} B$ if $\sigma\rightconnects{\cV}\tau$ for some $\tau\in B$;

		\item $A\rightconnects{\cV} B$ if $\sigma\rightconnects{\cV}\tau$ for some $\sigma\in A$ and $\tau\in B$.
	\end{itemize}
	Conversely, we use the symbol $\nrightconnects{\cV}$ if there exists no nontrivial solution going from a simplex (or a set) to another simplex (or set).

	From the definition of a flow, we can make the following observations, which will be useful in Section \ref{sec:MorseTheorems} to deduce properties of a subcomplex $L\subseteq K$ from $\Pi_\cV$.

	\begin{lem}\label{lem:IfFaceConnect}
		Let $\Pi_\cV:K\multimap K$ be a flow and consider $\alpha\subset\sigma\in K$. If $\sigma\nrightconnects{\cV}\alpha$, then $\sigma\in\image\cV\backslash\Fix\cV$ and $\alpha=\cV^{-1}(\sigma)$.
	\end{lem}

	\begin{proof}
		If $\sigma\in\Fix\cV$, then $\Pi_\cV(\sigma)=\Cl\sigma\ni\alpha$, so $\sigma\rightconnects{\cV}\alpha$. If $\sigma\in\dom\cV\backslash\Fix\cV$, then $\Pi_\cV(\sigma)=\{\cV(\sigma)\}$ and $\Pi_\cV\left(\cV(\sigma)\right) = \Ex\cV(\sigma)\backslash\{\sigma\}\ni\alpha$, so $\sigma\rightconnects{\cV}\cV(\sigma)\rightconnects{\cV}\alpha$. Finally, suppose $\sigma\in\image\cV\backslash\Fix\cV$ but $\alpha\neq\cV^{-1}(\sigma)$. Then, $\Pi_\cV(\sigma) = \Ex\sigma\backslash\{\cV^{-1}(\sigma)\}\ni\alpha$ so, again, $\sigma\rightconnects{\cV}\alpha$.
	\end{proof}
	\pagebreak
	\begin{lem}\label{lem:IfConnectsCoFaceConnects}
		Let $\Pi_\cV:K\multimap K$ be a flow and consider $\sigma,\tau\in K$ such that $\sigma\rightconnects{\cV}\tau$.
		\begin{enumerate}
			\item\label{lem:IfConnectsCoFaceConnectsEnum1} For all $\beta\supset\sigma$, we have $\beta\rightconnects{\cV}\tau$.
			\item\label{lem:IfConnectsCoFaceConnectsEnum2} For all $\alpha\subset\tau$, we have $\sigma\rightconnects{\cV}\alpha$.
		\end{enumerate}
	\end{lem}

	\begin{proof}
		To prove \ref{lem:IfConnectsCoFaceConnectsEnum1}, let $\beta\supset\sigma\rightconnects{\cV}\tau$. If $\beta\rightconnects{\cV}\sigma$, the result is obvious. Otherwise, by Lemma \ref{lem:IfFaceConnect}, we have $\beta\in\image\cV\backslash\Fix\cV$ and $\sigma=\cV^{-1}(\beta)$. Hence, $\Pi_\cV(\sigma)=\{\beta\}$ so every solution going from $\sigma$ to $\tau$ necessarily goes through $\beta$, thus $\beta\rightconnects{\cV}\tau$.

		To prove \ref{lem:IfConnectsCoFaceConnectsEnum2}, consider $\sigma\rightconnects{\cV}\tau\supset\alpha$. Again, if $\tau\rightconnects{\cV}\alpha$, we have the result. Otherwise, let $\varrho$ be a nontrivial solution with $\varrho(0)=\sigma$ and $\varrho(n)=\tau$. By Lemma \ref{lem:IfFaceConnect}, we have $\tau\in\image\cV\backslash\Fix\cV$ and $\alpha=\cV^{-1}(\tau)$, so $\varrho(n-1)\neq \tau$.
		\begin{itemize}
			\item If $\varrho(n-1)\in\Fix\cV$, then $\Pi_\cV(\varrho(n-1)) = \Cl\varrho(n-1)\supset\Cl\tau\ni\alpha$, so $\varrho':\Z\nrightarrow K$ such that $\varrho'(i)=\varrho(i)$ for $i=0,...,n-1$ and $\varrho'(n)=\alpha$ is a solution going from $\sigma$ to $\alpha$. Hence, $\sigma\rightconnects{\cV}\alpha$.

			\item If $\varrho(n-1)\in\image\cV\backslash\Fix\cV$, since $\cV(\alpha)=\tau\neq\varrho(n-1)$, we can again verify that $\alpha\in\Pi_\cV(\varrho(n-1))$, so $\varrho'$ as defined previously is still a solution going from $\sigma$ to $\alpha$. An example of such $\varrho'$ is shown in Figure \ref{fig:FlowSolutionAdapted}.

			\item If $\varrho(n-1)\in\dom\cV\backslash\Fix\cV$, we then have $\tau\in\Pi_\cV(\varrho(n-1))=\{\cV(\varrho(n-1))\}$, so $\cV(\varrho(n-1))=\tau=\cV(\alpha)$, hence $\varrho(n-1)=\alpha$ by the injectivity of $\cV$. Thus, $\varrho$ restricted to $[0,n-1]$ is a solution going from $\sigma$ to $\alpha$.\qedhere
		\end{itemize}
	\end{proof}

	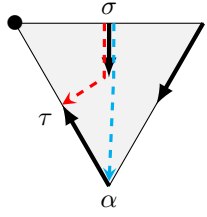
\begin{figure}
		\centering
		\begin{tikzpicture}[scale=1.25]
			\coordinate (A) at (1,1.732);
			\coordinate (E) at (2,0);
			\coordinate (B) at (3,1.732);

			\coordinate (AB) at ($0.5*(A)+0.5*(B)$);
			\coordinate (AE) at ($0.5*(A)+0.5*(E)$);
			\coordinate (BE) at ($0.5*(B)+0.5*(E)$);

			\coordinate (ABE) at ($0.333*(A)+0.333*(B)+0.333*(E)$);

			\fill[black!5] (A) -- (B) -- (E) -- cycle;

			\draw (A) -- (B) -- (E) -- cycle;

			\node at (A){\LARGE $\bullet$};

			\draw[ultra thick, -latex] (AB) -- (ABE);
			\draw[ultra thick, -latex] (B) -- (BE);
			\draw[ultra thick, -latex] (E) -- (AE);

			\draw[very thick, red, dashed, -stealth] ($(AB) - (0.05,0)$) -- ($(ABE) - (0.05,0)$) -- (AE);
			\draw[very thick, cyan, dashed, -stealth] ($(AB) + (0.05,0)$) -- ($(ABE) + (0.05,0)$) -- ($(E) + (0,0.05)$);

			\node[above] at (AB) {$\sigma$};
			\node[below left] at (AE) {$\tau$};
			\node[below] at (E) {$\alpha$};

		\end{tikzpicture}
		\caption{In red, a solution $\varrho$ from a simplex $\sigma$ to another simplex $\tau$. In part \ref{lem:IfConnectsCoFaceConnectsEnum2} of Lemma \ref{lem:IfConnectsCoFaceConnects}, we see that if $\varrho$ does not go through $\alpha=\tau^-$, then we can adapt it to define a new solution $\varrho'$, shown in blue, which goes from $\sigma$ to $\alpha$}\label{fig:FlowSolutionAdapted}
	\end{figure}

	As we will see in this section, many concepts of combinatorial dynamics rely on the idea of solutions between or contained inside sets. In particular, consider a subset $A$ of a simplicial complex $K$. In \citep{Batko2020, Kaczynski2016}, the simplices $\sigma\in A$ for which $\Sol(\sigma,A)\neq\emptyset$ are often considered. In the context of multivector fields \citep{Mrozek2017}, simplices $\sigma\in A$ for which $\Sol(\sigma^+,A)\neq\emptyset$ are mostly of interest. In our context, when $A$ is $\cV$-compatible, we prove that both ideas are equivalent.

	\begin{lem}\label{lem:ExistenceSolutions}
		Let $\Pi_\cV$ be a flow on a simplicial complex $K$ and consider $\sigma\in A\subseteq K$. If $A$ is $\cV$-compatible, then the following statements are equivalent
		\begin{enumerate}
			\item $\Sol(\sigma^-,A)\neq\emptyset$
			\item $\Sol(\sigma^+,A)\neq\emptyset$
			\item $\Sol(\sigma,A)\neq\emptyset$.
		\end{enumerate}
	\end{lem}

	\begin{proof}
		We first show that $\Sol(\sigma^-,A)\neq\emptyset \Leftrightarrow \Sol(\sigma^+,A)\neq\emptyset$. The result is obvious if $\sigma^-=\sigma^+$. Let $\sigma^-\neq\sigma^+$, hence $\sigma^-\in\dom\cV\backslash\Fix\cV$ and $\cV(\sigma^-)=\sigma^+$. If $\varrho:\Z\rightarrow A$ is a full solution with $\varrho(n) = \sigma^-$ for some $n\in\Z$, then $\varrho(n+1) = \sigma^+$, so $\varrho$ is a full solution with $\sigma^+\in\image\varrho$, thus $\Sol(\sigma^+,A)\neq\emptyset$. Conversely, suppose $\varrho:\Z\rightarrow A$ is a full solution with $\varrho(n) = \sigma^+$ for $n\in\Z$. By definition of $\Pi_\cV$, we have
		\begin{align*}
			\sigma^+ = \varrho(n)\in \begin{cases}
				\Cl\varrho(n-1) & \text{ if } \varrho(n-1)\in\Fix\cV, \\
				\Ex\varrho(n-1)\backslash\{\cV^{-1}(\varrho(n-1))\} & \text{ if } \varrho(n-1)\in\image\cV\,\backslash\Fix\cV, \\
				\{\cV(\varrho(n-1))\} & \text{ if } \varrho(n-1)\in\dom\cV\,\backslash\Fix\cV.
			\end{cases}
		\end{align*}
		\begin{itemize}
			\item If $\varrho(n-1)\in\Fix\cV$, then $\sigma^-\subset\sigma^+\in\Cl\varrho(n-1)=\Pi_\cV(\varrho(n-1))$, so $\sigma^-\in\Pi_\cV(\varrho(n-1))$. Thus, the map $\varrho':\Z\rightarrow K$ such that
			\begin{align*}
				\varrho'(i) = \begin{cases}
					\varrho(i+1) &\text{ if } i < n-1\\
					\sigma^- &\text{ if } i=n-1\\
					\varrho(i) & \text{ if } i \geq n
				\end{cases}
			\end{align*}
			is a full solution with $\sigma^-\in\image\varrho'\subseteq A$ because $\image\varrho'=\image\varrho\cup\{\sigma^-\}$ and $\sigma^-\in A$ by the $\cV$-compatibility hypothesis on $A$.

			\item If $\varrho(n-1)\in\image\cV\backslash\Fix\cV$, then $\sigma^-\subset\sigma^+\subset\varrho(n-1)$, so $\sigma^-\in\Ex\varrho(n-1)$. Since $\cV(\sigma^-)=\sigma^+\neq\varrho(n-1)$, it follows that $\sigma^-\neq\cV^{-1}(\varrho(n-1))$ and $\sigma^-\in\Ex\varrho(n-1)\backslash\{\cV^{-1}(\varrho(n-1))\}=\Pi_\cV(\varrho(n-1))$. Thus, the map $\varrho'$ as defined above is again a full solution with $\sigma^-\in\image\varrho'\subseteq A$.

			\item If $\varrho(n-1)\in\dom\cV\backslash\Fix\cV$, then $\cV(\sigma^-) = \sigma^+ = \cV(\varrho(n-1))$, so $\sigma^-=\varrho(n-1)$ by injectivity of $\cV$. Hence, $\varrho\in\Sol(\sigma^-,A)\neq\emptyset$.
		\end{itemize}
		This shows that $\Sol(\sigma^-,A)\neq\emptyset\Leftrightarrow\Sol(\sigma^+,A)\neq\emptyset$. It follows that $\Sol(\sigma,A)\neq\emptyset\Leftrightarrow\Sol(\sigma^+,A)\neq\emptyset$. Indeed, we have either $\sigma=\sigma^+$ or $\sigma=\sigma^-$: the first case is obvious and we have just proven the second.
	\end{proof}

	\subsection{Isolated invariant sets and Conley index}

	\begin{defn}[(Isolated) invariant set]\label{def:IsolatedInvariantSet}
		Let $\Pi_\cV$ be a flow on a simplicial complex $K$ and consider $S\subseteq K$.
		\begin{enumerate}
			\item\label{def:IsolatedInvariantSetEnum1} We say that $S$ is an \emph{invariant set} if for every $\sigma\in S$, we have $\Sol(\sigma,S)\neq\emptyset$.
			\item\label{def:IsolatedInvariantSetEnum2} An invariant set $S$ is \emph{isolated} if $\Ex S$ is closed and there is no solution $\varrho:\{-1,0,1\}\rightarrow K$ such that $\varrho(-1),\varrho(1)\in S$ and $\varrho(0)\in\Ex S$.
		\end{enumerate}
	\end{defn}

	\begin{ex}
		Consider $\cV$ as in Figure \ref{fig:InvariantIsolated}. The set $\{B, AB, BD, ABD\}$, shown in red, is not invariant since every solution that goes through it necessarily exits the set. The orange set $\{C, CD, ACD\}$ is invariant, but not isolated since its exit set is not closed and the map $\varrho:\{-1,0,1\}\rightarrow K$ such that $\varrho(-1)=ACD$, $\varrho(0) = AC$ and $\varrho(1)=C$ is a solution that violates the definition of isolated set. The blue set $\{C, AC, CD, ACD\}$ is also invariant, its exit set is closed, but it is not isolated because there is a map $\varrho:\{-1,0,1\}\rightarrow K$, namely the one such that $\varrho(-1)=ACD$, $\varrho(0) = A$ and $\varrho(1)=AC$, which violates the definition. Finally, we can see that the green set $\{A, C, D, AC, AD, CD, ACD\}$ verifies both conditions to be an isolated invariant set.

		\begin{figure}[h]
			\centering
			\begin{subfigure}[b]{0.32\textwidth}
				\centering
				\begin{tikzpicture}[scale=1.2]
					\coordinate (D) at (0,0);
					\coordinate (A) at (1,1.732);
					\coordinate (E) at (2,0);
					\coordinate (B) at (3,1.732);

					\coordinate (AB) at ($0.5*(A)+0.5*(B)$);
					\coordinate (AD) at ($0.5*(A)+0.5*(D)$);
					\coordinate (AE) at ($0.5*(A)+0.5*(E)$);
					\coordinate (BE) at ($0.5*(B)+0.5*(E)$);
					\coordinate (DE) at ($0.5*(D)+0.5*(E)$);

					\coordinate (ADE) at ($0.333*(A)+0.333*(D)+0.333*(E)$);
					\coordinate (ABE) at ($0.333*(A)+0.333*(B)+0.333*(E)$);

					\fill[black!5] (A) -- (B) -- (E) -- (D) -- cycle;

					\fill[orange!50] ($(E)!\ec!(DE)$) arc(0:-90:\wc) -- ($(D) - (0,\wc)$) arc(270:60:\wc) arc(240:330:\wc) -- ($(A)!\cc!(ADE)$) -- cycle;
					\fill[red!50] ($(A)!\ec!(AB)$) arc(180:90:\wc) -- ($(B) + (0,\wc)$) arc(90:-30:\wc) -- +($(B)!{2-\wc-\ec}!(BE) - (B)$) arc(330:240:\wc) -- cycle;

					\draw (A) -- (B) -- (E) -- (D) -- cycle;
					\draw (A) -- (E);

					\node at (D){\LARGE $\bullet$};
					\node at (DE){\LARGE $\bullet$};
					\node at (ADE){\LARGE $\bullet$};

					\draw[ultra thick, -latex] (A) -- (AD);
					\draw[ultra thick, -latex] (AB) -- (ABE);
					\draw[ultra thick, -latex] (B) -- (BE);
					\draw[ultra thick, -latex] (E) -- (AE);

					\node[xshift=-11pt, yshift=-6pt] at (A) {$A$};
					\node[xshift=-11pt, yshift=-6pt] at (B) {$B$};
					\node[xshift=11pt, yshift=6pt] at (D) {$C$};
					\node[xshift=11pt, yshift=6pt] at (E) {$D$};

				\end{tikzpicture}
				\caption{}\label{fig:InvariantNotIsolatedNot}
			\end{subfigure}
			\begin{subfigure}[b]{0.32\textwidth}
				\centering
				\begin{tikzpicture}[scale=1.2]
					\coordinate (D) at (0,0);
					\coordinate (A) at (1,1.732);
					\coordinate (E) at (2,0);
					\coordinate (B) at (3,1.732);

					\coordinate (AB) at ($0.5*(A)+0.5*(B)$);
					\coordinate (AD) at ($0.5*(A)+0.5*(D)$);
					\coordinate (AE) at ($0.5*(A)+0.5*(E)$);
					\coordinate (BE) at ($0.5*(B)+0.5*(E)$);
					\coordinate (DE) at ($0.5*(D)+0.5*(E)$);

					\coordinate (ADE) at ($0.333*(A)+0.333*(D)+0.333*(E)$);
					\coordinate (ABE) at ($0.333*(A)+0.333*(B)+0.333*(E)$);

					\fill[black!5] (A) -- (B) -- (E) -- (D) -- cycle;

					\fill[cyan!50] ($(E)!\ec!(DE)$) arc(0:-90:\wc) -- ($(D) - (0,\wc)$) arc(270:150:\wc) -- +($(D)!{2-\wc-\ec}!(AD) - (D)$) arc(150:60:\wc) -- cycle;

					\draw (A) -- (B) -- (E) -- (D) -- cycle;
					\draw (A) -- (E);

					\node at (D){\LARGE $\bullet$};
					\node at (DE){\LARGE $\bullet$};
					\node at (ADE){\LARGE $\bullet$};

					\draw[ultra thick, -latex] (A) -- (AD);
					\draw[ultra thick, -latex] (AB) -- (ABE);
					\draw[ultra thick, -latex] (B) -- (BE);
					\draw[ultra thick, -latex] (E) -- (AE);

					\node[xshift=-11pt, yshift=-6pt] at (A) {$A$};
					\node[xshift=-11pt, yshift=-6pt] at (B) {$B$};
					\node[xshift=11pt, yshift=6pt] at (D) {$C$};
					\node[xshift=11pt, yshift=6pt] at (E) {$D$};

				\end{tikzpicture}
				\caption{}\label{fig:InvariantIsolatedNot}
			\end{subfigure}
			\begin{subfigure}[b]{0.32\textwidth}
				\centering
				\begin{tikzpicture}[scale=1.2]
					\coordinate (D) at (0,0);
					\coordinate (A) at (1,1.732);
					\coordinate (E) at (2,0);
					\coordinate (B) at (3,1.732);

					\coordinate (AB) at ($0.5*(A)+0.5*(B)$);
					\coordinate (AD) at ($0.5*(A)+0.5*(D)$);
					\coordinate (AE) at ($0.5*(A)+0.5*(E)$);
					\coordinate (BE) at ($0.5*(B)+0.5*(E)$);
					\coordinate (DE) at ($0.5*(D)+0.5*(E)$);

					\coordinate (ADE) at ($0.333*(A)+0.333*(D)+0.333*(E)$);
					\coordinate (ABE) at ($0.333*(A)+0.333*(B)+0.333*(E)$);

					\fill[black!5] (A) -- (B) -- (E) -- (D) -- cycle;

					\fill[green!50] ($(E) + (0,-\wc)$) arc(-90:30:\wc) -- +($(A)-(E)$) arc(30:150:\wc) -- +($(D)-(A)$) arc(150:270:\wc) -- cycle;

					\draw (A) -- (B) -- (E) -- (D) -- cycle;
					\draw (A) -- (E);

					\node at (D){\LARGE $\bullet$};
					\node at (DE){\LARGE $\bullet$};
					\node at (ADE){\LARGE $\bullet$};

					\draw[ultra thick, -latex] (A) -- (AD);
					\draw[ultra thick, -latex] (AB) -- (ABE);
					\draw[ultra thick, -latex] (B) -- (BE);
					\draw[ultra thick, -latex] (E) -- (AE);

					\node[xshift=-11pt, yshift=-6pt] at (A) {$A$};
					\node[xshift=-11pt, yshift=-6pt] at (B) {$B$};
					\node[xshift=11pt, yshift=6pt] at (D) {$C$};
					\node[xshift=11pt, yshift=6pt] at (E) {$D$};

				\end{tikzpicture}
				\caption{}\label{fig:InvariantIsolatedYes}
			\end{subfigure}
			\caption{In \subref{fig:InvariantNotIsolatedNot}, the red subset is not invariant while the orange subset is invariant, but not isolated. In \subref{fig:InvariantIsolatedNot}, the blue subset is also invariant but not isolated. In \subref{fig:InvariantIsolatedYes}, the green subset is isolated invariant}\label{fig:InvariantIsolated}
		\end{figure}
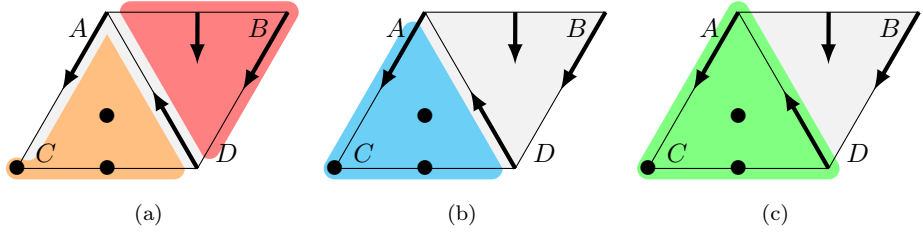
	\end{ex}

	The definition of an invariant set and that of an isolated invariant set given here come from \citep{Batko2020, Kaczynski2016}. In the context of multivector fields, \citet{Mrozek2017} defines a set $S\subseteq K$ as invariant if for all $\sigma\in S$, we have $\Sol(\sigma^+,S_\cV)\neq\emptyset$, where $S_\cV$ is the maximal subset of $S$ which is $\cV$-compatible. In the context of vector fields, we show that these definitions follow the same idea.

	\pagebreak

	\begin{lem}\label{lem:InvariantDefsEqui}
		Let $\Pi_\cV$ be a flow on a simplicial complex $K$.
		\begin{enumerate}
			\item\label{lem:InvariantDefsEquiEnum1} A subset of $K$ is invariant in the sense of \citet[Section 6.4]{Mrozek2017} if and only if it is $\cV$-compatible and invariant in the sense of Definition \ref{def:IsolatedInvariantSet}\ref{def:IsolatedInvariantSetEnum1}.

			\item A subset of $K$ is an isolated invariant set in the sense of \citet[Section 7.1]{Mrozek2017} if and only if it is isolated invariant in the sense of Definition \ref{def:IsolatedInvariantSet}\ref{def:IsolatedInvariantSetEnum2}.
		\end{enumerate}
	\end{lem}

	\begin{proof}
		If $S$ invariant in the sense of \citet{Mrozek2017}, then it is $\cV$-compatible \citep[Proposition 6.4]{Mrozek2017} and since $S_\cV=S$, it follows from Lemma \ref{lem:ExistenceSolutions} that $S$ is invariant in the sense of Definition \ref{def:IsolatedInvariantSet}\ref{def:IsolatedInvariantSetEnum1}. Conversely, if $S$ is $\cV$-compatible and invariant in the sense of Definition \ref{def:IsolatedInvariantSet}\ref{def:IsolatedInvariantSetEnum1}, we again have $S_\cV=S$ and the invariance of $S$ in the sense of \citet{Mrozek2017} follows from Lemma \ref{lem:ExistenceSolutions}.

		The second part of the lemma is shown as follows. From Theorem 7.1 in \citep{Mrozek2017}, $S\subseteq K$ is an isolated invariant set in the sense of \citet{Mrozek2017} if and only if $S$ is invariant in the sense of \citet{Mrozek2017} and $\Ex S$ is closed. From part \ref{lem:InvariantDefsEquiEnum1} of the lemma, $S$ is invariant in the sense of \citet{Mrozek2017} if and only if it is $\cV$-compatible and invariant in the sense of Definition \ref{def:IsolatedInvariantSet}\ref{def:IsolatedInvariantSetEnum1}. Finally, Proposition 3.4 in \citep{Batko2020} states that $S$ is $\cV$-compatible, $S$ is invariant in the sense of Definition \ref{def:IsolatedInvariantSet}\ref{def:IsolatedInvariantSetEnum1} and $\Ex S$ is closed if and only if $S$ is isolated invariant in the sense of Definition \ref{def:IsolatedInvariantSet}\ref{def:IsolatedInvariantSetEnum2}, hence the result.
	\end{proof}

	From Lemma \ref{lem:InvariantDefsEqui}, it follows that most results from the combinatorial dynamics theory defined by \citet{Mrozek2017} may be directly applied within the framework of this article as long as the invariant sets considered here are $\cV$-compatible. In particular, we can define the Conley index of an isolated invariant set.

	\begin{defn}
		Let $S$ be an isolated invariant set of a flow. The \emph{Conley index} of $S$, noted $\Con(S)$, is the (singular) homology of $\Cl S$ relative to $\Ex S$ and the \emph{$p^{\text{th}}$ Conley coefficient} $\beta_p(S)$ of $S$ is the rank of the $p^{\text{th}}$ homology group of $\Con(S)$. We call
		\begin{gather*}
			P_S(t) := \sum_{p\in\N}\beta_p(S)t^p
		\end{gather*}
		the \emph{Conley polynomial} of $S$.
	\end{defn}

	Notice that when considering a flow on a finite simplicial complex $K$, then $K$ is trivially an isolated invariant set and $\Con(K)$ is simply the homology of $K$. The Conley coefficients of $K$ thus coincide with its Betti numbers and its Conley polynomial equals its Poincar{\'e} polynomial.

	\subsection{Morse decompositions}

	For a full solution $\varrho:\Z\rightarrow K$ of a flow $\Pi_\cV$ on $K$, define the \emph{$\alpha$-limit} and the \emph{$\omega$-limit} sets of $\varrho$ as
	\begin{align*}
		\alpha(\varrho) = \bigcap_{k\in\Z}\left\lbrace \varrho(n)\ \big|\ n\leq k\right\rbrace,\qquad \omega(\varrho) = \bigcap_{k\in\Z}\left\lbrace \varrho(n)\ \big|\ k\leq n\right\rbrace.
	\end{align*}

	\begin{defn}[Morse decomposition]\label{def:MorseDecomposition}
		Let $\Pi_\cV$ be a flow on a simplicial complex $K$. Consider a collection $\cM = \{M_r\subseteq K\ |\ r\in\P\}$ indexed by a poset $(\P,\leq)$. We say $\cM$ is a \emph{Morse decomposition} of $\Pi_\cV$ if the following conditions are satisfied:
		\begin{enumerate}
			\item\label{def:MorseDecompositionEnum1} $\cM$ is a collection of mutually disjoint isolated invariant subsets of $K$;
			\item\label{def:MorseDecompositionEnum2} for every full solution $\varrho$ of $\Pi_\cV$, there are some $r,r'\in\P$ such that $r\leq r'$ for which $\alpha(\varrho)\subseteq M_{r'}$ and $\omega(\varrho)\subseteq M_r$;
			\item\label{def:MorseDecompositionEnum3} if there exists a full solution $\varrho$ of $\Pi_\cV$ such that $\alpha(\varrho)\cup\omega(\varrho)\subseteq M_r$ for some $r\in\P$, then $\image\varrho\subseteq M_r$.
		\end{enumerate}
	\end{defn}

	Essentially, the sets $M_r$ of a Morse decomposition $\cM$ are the starting and ending sets of every solution of $\Pi_\cV:K\multimap K$. Thus, we can see a Morse decomposition as an object describing the different connections that exist in a flow.

	\begin{ex}\label{ex:MorseDecomposition}
		Consider the combinatorial field $\cV$ from Figure \ref{fig:MorseDecompositionM}. Then, the collection
		\begin{gather*}
			\cM = \big\lbrace\{D\}, \{F, FG, G, GJ, J, FJ\}, \{DE\}, \{HI\}, \{BC\}, \{ADE\}, \{BEF\}, \{FGJ\}\big\rbrace
		\end{gather*}
		indexed by $(\P,\leq)$ as depicted in Figure \ref{fig:MorseDecompositionP} is a Morse decomposition of $\Pi_\cV$.

		\begin{figure}[h]
			\centering
			\begin{subfigure}[b]{0.55\textwidth}
				\centering
				\begin{tikzpicture}[scale=1]
					\coordinate (D) at (0,0);
					\coordinate (A) at (1,1.732);
					\coordinate (H) at (1,-1.732);
					\coordinate (E) at (2,0);
					\coordinate (B) at (3,1.732);
					\coordinate (I) at (3,-1.732);
					\coordinate (F) at (4,0);
					\coordinate (C) at (5,1.732);
					\coordinate (J) at (5,-1.732);
					\coordinate (G) at (6,0);

					\coordinate (AB) at ($0.5*(A)+0.5*(B)$);
					\coordinate (BC) at ($0.5*(B)+0.5*(C)$);
					\coordinate (AD) at ($0.5*(A)+0.5*(D)$);
					\coordinate (AE) at ($0.5*(A)+0.5*(E)$);
					\coordinate (BE) at ($0.5*(B)+0.5*(E)$);
					\coordinate (BF) at ($0.5*(B)+0.5*(F)$);
					\coordinate (CF) at ($0.5*(C)+0.5*(F)$);
					\coordinate (CG) at ($0.5*(C)+0.5*(G)$);
					\coordinate (DE) at ($0.5*(D)+0.5*(E)$);
					\coordinate (EF) at ($0.5*(E)+0.5*(F)$);
					\coordinate (FG) at ($0.5*(F)+0.5*(G)$);
					\coordinate (DH) at ($0.5*(D)+0.5*(H)$);
					\coordinate (EH) at ($0.5*(E)+0.5*(H)$);
					\coordinate (EI) at ($0.5*(E)+0.5*(I)$);
					\coordinate (FI) at ($0.5*(F)+0.5*(I)$);
					\coordinate (FJ) at ($0.5*(F)+0.5*(J)$);
					\coordinate (GJ) at ($0.5*(G)+0.5*(J)$);
					\coordinate (HI) at ($0.5*(H)+0.5*(I)$);
					\coordinate (IJ) at ($0.5*(I)+0.5*(J)$);

					\coordinate (ADE) at ($0.333*(A)+0.333*(D)+0.333*(E)$);
					\coordinate (ABE) at ($0.333*(A)+0.333*(B)+0.333*(E)$);
					\coordinate (BEF) at ($0.333*(B)+0.333*(E)+0.333*(F)$);
					\coordinate (BCF) at ($0.333*(B)+0.333*(C)+0.333*(F)$);
					\coordinate (CFG) at ($0.333*(C)+0.333*(F)+0.333*(G)$);
					\coordinate (DEH) at ($0.333*(D)+0.333*(E)+0.333*(H)$);
					\coordinate (EHI) at ($0.333*(E)+0.333*(H)+0.333*(I)$);
					\coordinate (EFI) at ($0.333*(E)+0.333*(F)+0.333*(I)$);
					\coordinate (FIJ) at ($0.333*(F)+0.333*(I)+0.333*(J)$);
					\coordinate (FGJ) at ($0.333*(F)+0.333*(G)+0.333*(J)$);

					\fill[black!5] (A) -- (C) -- (G) -- (J) -- (H) -- (D) -- cycle;
					%\fill[red!20] (A) -- (D) -- (E) -- cycle;
					%\fill[red!20] (B) -- (E) -- (F) -- cycle;
					%\fill[red!20] (F) -- (G) -- (J) -- cycle;

					\fill[red!50] (D) circle (5pt);
					\fill[blue!50] ($(D)!\ec!(DE) + (\wc,\wc)$) arc(90:270:\wc)  -- ($(E)!\ec!(DE) + (-\wc,-\wc)$) arc(-90:90:\wc) -- cycle;
					\fill[orange!50] ($(A)!\cc!(ADE)$) -- ($(D)!\cc!(ADE)$)  -- ($(E)!\cc!(ADE)$) -- cycle;
					\fill[green!50] ($(E)!\cc!(BEF)$) -- ($(F)!\cc!(BEF)$)  -- ($(B)!\cc!(BEF)$) -- cycle;
					\fill[pink] ($(B)!\ec!(BC) + (\wc,\wc)$) arc(90:270:\wc)  -- ($(C)!\ec!(BC) + (-\wc,-\wc)$) arc(-90:90:\wc) -- cycle;
					\fill[yellow] ($(H)!\ec!(HI) + (\wc,\wc)$) arc(90:270:\wc)  -- ($(I)!\ec!(HI) + (-\wc,-\wc)$) arc(-90:90:\wc) -- cycle;
					\fill[cyan!50] ($(F) + (0,\wc)$) arc(90:210:\wc) -- +($(J)-(F)$) arc(210:330:\wc) -- +($(G)-(J)$) arc(-30:90:\wc) -- cycle;
					\fill[violet!50] ($(F)!\cc!(FGJ)$) -- ($(G)!\cc!(FGJ)$)  -- ($(J)!\cc!(FGJ)$) -- cycle;

					\draw (A) -- (C) -- (G) -- (J) -- (H) -- (D) -- cycle;
					\draw (D) -- (G)  (A) -- (I) -- (C)  (H) -- (B) -- (J);

					%\draw[red,very thick] (B) -- (C)  (D) -- (E)  (H) -- (I);

					%\node at (A){$\bullet$};
					%\node at (B){$\bullet$};
					%\node at (C){$\bullet$};
					%\node[red] at (D){$\bullet$};
					%\node at (E){$\bullet$};
					%\node at (F){$\bullet$};
					%\node at (G){$\bullet$};
					%\node at (H){$\bullet$};
					%\node at (I){$\bullet$};
					%\node at (J){$\bullet$};

					\node at (D){\LARGE $\bullet$};
					\node at (DE){\LARGE $\bullet$};
					\node at (ADE){\LARGE $\bullet$};
					\node at (BEF){\LARGE $\bullet$};
					\node at (BC){\LARGE $\bullet$};
					\node at (HI){\LARGE $\bullet$};
					\node at (FGJ){\LARGE $\bullet$};

					\draw[ultra thick, -latex] (A) -- (AD);
					\draw[ultra thick, -latex] (AB) -- (ABE);
					\draw[ultra thick, -latex] (B) -- (BE);
					\draw[ultra thick, -latex] (C) -- (CG);
					\draw[ultra thick, -latex] (BF) -- (BCF);
					\draw[ultra thick, -latex] (CF) -- (CFG);
					\draw[ultra thick, -latex] (E) -- (AE);
					\draw[ultra thick, -latex] (EF) -- (EFI);
					\draw[ultra thick, -latex] (F) -- (FJ);
					\draw[ultra thick, -latex] (G) -- (FG);
					\draw[ultra thick, -latex] (EH) -- (DEH);
					\draw[ultra thick, -latex] (EI) -- (EHI);
					\draw[ultra thick, -latex] (FI) -- (FIJ);
					\draw[ultra thick, -latex] (H) -- (DH);
					\draw[ultra thick, -latex] (I) -- (IJ);
					\draw[ultra thick, -latex] (J) -- (GJ);

					\node[xshift=11pt, yshift=-6pt] at (A) {$A$};
					\node[xshift=11pt, yshift=-6pt] at (B) {$B$};
					\node[xshift=11pt, yshift=-6pt] at (C) {$C$};
					\node[xshift=11pt, yshift=6pt] at (D) {$D$};
					\node[xshift=11pt, yshift=6pt] at (E) {$E$};
					\node[xshift=11pt, yshift=6pt] at (F) {$F$};
					\node[xshift=11pt, yshift=6pt] at (G) {$G$};
					\node[xshift=11pt, yshift=6pt] at (H) {$H$};
					\node[xshift=11pt, yshift=6pt] at (I) {$I$};
					\node[xshift=11pt, yshift=6pt] at (J) {$J$};

				\end{tikzpicture}
				\caption{}\label{fig:MorseDecompositionM}
			\end{subfigure}
			\begin{subfigure}[b]{0.37\textwidth}
				\centering
				\begin{tikzpicture}[scale=1]
					\draw[white, opacity=0] (8,-1.732-\wc) rectangle (11, 1.732+\wc);

					\begin{scope}[every node/.style={rectangle,thick,draw}]
						\node[fill=red!50] (M1) at (9,-1.5) {$1$};
						\node[fill=cyan!50] (M2) at (11,-1.5) {$2$};
						\node[fill=blue!50] (M3) at (8,0) {$3$};
						\node[fill=yellow] (M4) at (9,0) {$4$};
						\node[fill=pink] (M5) at (10,0) {$5$};
						\node[fill=orange!50] (M6) at (8,1.5) {$6$};
						\node[fill=green!50] (M7) at (9.5,1.5) {$7$};
						\node[fill=violet!50] (M8) at (11,1.5) {$8$};
					\end{scope}

					\begin{scope}[>={latex},
						every edge/.style={draw,very thick}]
						\path [->] (M3) edge (M1);
						\path [->] (M4) edge (M1);
						\path [->] (M5) edge (M1);
						\path [->] (M4) edge (M2);
						\path [->] (M5) edge (M2);
						\path [->] (M8) edge (M2);
						\path [->] (M6) edge (M3);
						\path [->] (M7) edge (M3);
						\path [->] (M7) edge (M4);
						\path [->] (M7) edge (M5);
					\end{scope}
				\end{tikzpicture}
				\caption{}\label{fig:MorseDecompositionP}
			\end{subfigure}
			\caption{In \subref{fig:MorseDecompositionM}, a discrete vector field $\cV$. The dots identify the fixed points of $\cV$, and each color represents a Morse set of $\cM$ as defined in Example \ref{ex:MorseDecomposition}. In \subref{fig:MorseDecompositionP}, the Hasse diagram of the poset $\P$ used to index $\cM$ }\label{fig:MorseDecomposition}
		\end{figure}
	\end{ex}

	Moreover, information on the homology of a simplicial complex $K$ can be deduced from a Morse decomposition defined on $K$. Indeed, a Morse decomposition leads to a Morse equation, from which are derived the famous Morse inequalities. The following proposition is an adaptation of Theorem 9.11 in \citep{Mrozek2017}.

	\begin{prop}\label{prop:MorseEquationMorseDecomposition}
		Let $\cM=\{M_r\ |\ r\in\P\}$ be a Morse decomposition of a simplicial complex $K$. Then,
		\begin{gather*}
			\sum_{r\in\P}P_{M_r}(t) = P_K(t) + (1+t)Q(t)
		\end{gather*}
		for some polynomial $Q(t)$ with non-negative coefficients.
	\end{prop}

	\begin{cor}\label{coro:MorseInequalitiesMorseDecomposition}
		Let $\cM=\{M_r\ |\ r\in\P\}$ be a Morse decomposition of a simplicial complex $K$ with $\dim K = n$. Let
		\begin{gather*}
			m_p := \sum_{r\in\P}\beta_p(M_r).
		\end{gather*}
		For all $p=0,1,...,n$, we have
		\begin{enumerate}
			\item the strong Morse inequalities:
			\begin{gather*}
				m_p - m_{p-1} + \cdots +(-1)^p m_0 \geq \beta_p(K) - \beta_{p-1}(K) + \cdots +(-1)^p \beta_0(K);
			\end{gather*}
			\item the weak Morse inequalities:
			\begin{gather*}
				m_p \geq \beta_p(K);
			\end{gather*}
			\item an alternative expression for the Euler characteristic $\chi(K)$ of $K$:
			\begin{gather*}
				m_0 - m_1 + \cdots +(-1)^n m_n = \beta_0(K) - \beta_1(K) + \cdots +(-1)^n \beta_n(K) = \chi(K)
			\end{gather*}
		\end{enumerate}
	\end{cor}

	\begin{proof}
		The proof is similar to that of Theorem 9.12 in \citep{Mrozek2017}. Recall that $P_K(t)=\sum_{p=0}^n\beta_p(K)t^p$. To prove the strong Morse inequalities, it suffices to compare the coefficients of the polynomials obtained by multiplying each side of the equation in Proposition \ref{prop:MorseEquationMorseDecomposition} by $(1+t)^{-1} = 1-t+t^2-\cdots$. The weak Morse inequalities follow. Also, we have the last equality by substituting $t=-1$ in the equation of Proposition \ref{prop:MorseEquationMorseDecomposition}.
	\end{proof}

	\subsection{Basic sets and acyclic flows}

	Let $\Pi_\cV:K\multimap K$ be a flow on a simplicial complex and for two simplices $\sigma,\tau\in K$, we write $\sigma\leftrightconnects{\cV}\tau$ when $\sigma\rightconnects{\cV}\tau$ and $\sigma\leftconnects{\cV}\tau$. Consider the \emph{chain recurrent set} of $\Pi_\cV$ defined as
	\begin{gather*}
		\cR = \left\lbrace \sigma\in K\ |\ \sigma\leftrightconnects{\cV}\sigma\right\rbrace.
	\end{gather*}
	In $\cR$, $\leftrightconnects{\cV}$ may be seen as an equivalence relation. A \emph{basic set} of $\Pi_\cV$ is an equivalence class of $\leftrightconnects{\cV}$ in $\cR$.

	\begin{thm}[Theorems 9.2 and 9.3 in \citep{Mrozek2017}]\label{theo:BasicMinimalMorseDecomposition}
		Let $\cB$ be the collection of basic sets of a flow $\Pi_\cV:K\multimap K$. Consider the relation $\leq$ on $\cB$ such that for all $B,B'\in\cB$,
		\begin{align*}
			B\leq B' \quad\Leftrightarrow\quad B\leftconnects{\cV}B'.
		\end{align*}
		The relation $\leq$ is a partial order on $\cB$ making it a Morse decomposition.

		Moreover, $\cB$ is the finest Morse decomposition of $\Pi_\cV$, meaning that for any Morse decomposition $\cM$ of $\Pi_\cV$, for each $B\in\cB$, there exists a $M\in\cM$ such that $B\subseteq M$.
	\end{thm}

	The Morse decomposition of a flow in basic sets essentially consists of fixed points and cycles. For instance, the Morse decomposition from Example \ref{ex:MorseDecomposition} is the finest that exists for the given flow. Hence, we deduce that the basic sets for a flow $\Pi_\cV$ with no cycle are exactly the fixed points of $\cV$.

	More formally, we say a flow $\Pi_\cV:K\multimap K$ is \emph{acyclic} if, for all $\sigma,\tau\in K$, $\sigma\leftrightconnects{\cV}\tau$ implies $\sigma = \tau$. When $\Pi_\cV$ is acyclic, we see that $\cR=\Fix\cV$ and the following result is shown in a straightforward manner.

	\begin{prop}\label{prop:BasicSetsAcyclicField}
		Let $\Pi_\cV:K\multimap K$ be an acyclic flow. The collection $\cB$ of basic sets of $\Pi_\cV$ is
		\begin{gather*}
			\cB = \big\lbrace\,\{\sigma\}\subseteq K\ |\ \sigma\in\Fix\cV\big\rbrace.
		\end{gather*}
	\end{prop}

	Furthermore, it is worth noting that the acyclicity of $\cV$ is equivalent to that of the associated flow $\Pi_\cV$. Indeed, we know that gradient fields of discrete Morse functions are acyclic \citep{Forman1998}, so this result implies that acyclic combinatorial flows play an essential role in the study of the dynamics of these functions. This next proposition was not shown explicitly in previous works, but it follows from results in \citep{DesjardinsCote2020}.

	\begin{prop}\label{prop:AcyclictyFieldEquiFlow}
		A discrete vector field $\cV:K\nrightarrow K$ is acyclic if and only if its associated flow $\Pi_\cV:K\multimap K$ is also acyclic.
	\end{prop}

	\begin{proof}
		If $\cV$ is not acyclic, there must exist a $\cV$-path $\alpha_0^{(p)},\beta_0^{(p+1)},\alpha_1^{(p)},...,\beta_{n-1}^{(p+1)},\alpha_n^{(p)}$ where $\alpha_0=\alpha_n$ and $n\geq 1$. We can easily verify that this path forms a solution in $\Pi_\cV$ going from $\alpha_0$ to $\alpha_n$. Hence, we have $\alpha_0\rightconnects{\cV}\beta_0\rightconnects{\cV}\alpha_n=\alpha_0$, so $\alpha_0\leftrightconnects{\cV}\beta_0$. Thus, $\Pi_\cV$ is not acyclic.

		Conversely, suppose that $\Pi_\cV$ is not acyclic, meaning that there exists two simplices $\sigma,\tau\in K$ such that  $\sigma\leftrightconnects{\cV}\tau$ and $\sigma\neq\tau$. We can show there exists a nontrivial closed $\cV$-path going from $\sigma$ to $\tau$. Indeed, since $\sigma\rightconnects{\cV}\tau\rightconnects{\cV}\sigma$, we know there is a solution $\varrho:\{0,1,...,m\}\rightarrow K$ with $m\geq 2$ such that $\varrho(0)=\varrho(m)=\sigma$. We also know from \citet[Proposition 5.3, Lemma 5.4]{DesjardinsCote2020} that $\image\varrho$ contains no fixed point of $\cV$ and that, for some $p\in\N$, the sequence $\varrho(0),\varrho(1),...,\varrho(m)$ alternates between simplices in $\dom\cV$ of dimension $p$ and simplices in $\image\cV$ of dimension $p+1$. Thus, $m=2n$ for some $n\in\N$. Assuming, without loss of generality, that $\varrho(0)\in\dom\cV$, we can define a nontrivial closed $\cV$-path $\alpha_0^{(p)},\beta_0^{(p+1)},\alpha_1^{(p)},...,\beta_{n-1}^{(p+1)},\alpha_n^{(p)}$ by considering $\alpha_i:=\varrho(2i)$ and $\beta_i :=\varrho(2i+1)$ for each $i=0,...,n-1$ as well as $\alpha_n:=\varrho(2n)=\varrho(0)$. Therefore, $\cV$ is not acyclic.
	\end{proof}

	\subsection{Coarsening a Morse decomposition}

	We saw how to find the finest Morse decomposition associated to some flow, but a coarse decomposition may also be useful to describe the dynamics of the flow in a more global manner. Thus, considering a fine enough Morse decomposition $\cM$, we show here under which conditions we can group together some elements of $\cM$ to obtain a coarser one.

	Consider a flow $\Pi_\cV:K\multimap K$. For $A,A'\subseteq K$, define the \emph{connecting set} $C(A',A)$ as the set of simplices $\sigma\in K$ for which there exists a full solution $\varrho:\Z\rightarrow K$ with $\sigma\in\image\varrho$, $\alpha(\varrho)\subseteq A'$ and $\omega(\varrho)\subseteq A$. A slightly different definition is proposed in \citep{Mrozek2017}, but we could verify that, in our context, it is equivalent to this one. Also, if $A=\{\sigma\}$ and $A'=\{\sigma'\}$, we simply write $C(\sigma',\sigma) := C(\{\sigma'\},\{\sigma\})$. When $A$ and $A'$ are invariant, we can show that
	\begin{gather*}
		C(A',A) = \left\lbrace \sigma\in K\ \big|\ A'\rightconnects{\cV}\sigma\rightconnects{\cV} A\right\rbrace.
	\end{gather*}
	From the definitions of connecting sets and of a Morse decomposition, the next result follows \citep[see][Proposition 9.1]{Mrozek2017}.

	\begin{prop}\label{prop:MorseDecompositionProperties}
		Let $\cM = \{M_r\ |\ r\in\P\}$ be a Morse decomposition of a flow $\Pi_\cV:K\multimap K$. For all $r,r'\in\P$, the following statements are true.
		\begin{enumerate}
			\item\label{prop:MorseDecompositionPropertiesEnum1} $C(M_{r'},M_r)\neq\emptyset\Leftrightarrow M_{r'}\rightconnects{\cV}M_r$.
			\item\label{prop:MorseDecompositionPropertiesEnum2} $C(M_{r'},M_r)$ is $\cV$-compatible.
			\item\label{prop:MorseDecompositionPropertiesEnum3} $C(M_r,M_r) = M_r$.
			\item\label{prop:MorseDecompositionPropertiesEnum4} If $C(M_{r'},M_r) \neq \emptyset$, then $r\leq r'$. From the contrapositive, it follows that $r'<r$ implies $C(M_{r'},M_r) = \emptyset$.
		\end{enumerate}
	\end{prop}

	The connecting sets can also be used to define Morse sets.

	\begin{defn}[Morse sets]\label{def:MorseSet}
		Let $\cM = \{M_r\ |\ r\in\P\}$ be a Morse decomposition of a flow $\Pi_\cV:K\multimap K$ and consider $I\subseteq \P$. The \emph{Morse set} associated to $I$ is
		\begin{align*}
			M(I) = \bigcup_{r,r'\in I}C(M_{r'},M_r).
		\end{align*}
	\end{defn}

	\begin{prop}[Theorem 9.4 in \citep{Mrozek2017}]\label{prop:MorseSetIsolatedInvariant}
		Every Morse set $M(I)$ is an isolated invariant set.
	\end{prop}

	Note that Proposition \ref{prop:MorseSetIsolatedInvariant} implies that Conley indexes are well defined for Morse sets $M(I)$. Also, it suggests that Morse sets make good candidates to build a Morse decomposition $\cM'$ from a finer decomposition $\cM = \{M_r\ |\ r\in\P\}$.
	Theorem~\ref{theo:ConditionsForMorseDecomposition} establishes the necessary and sufficient conditions for a partition $\{I_s\subseteq\P\ |\ s\in\S\}$ of $\P$ to induce a Morse decomposition $\cM' = \{M(I_s)\ |\ s\in\S\}$. Before stating and proving it formally, we give the intuition of the theorem with the two following examples.

	\begin{ex}\label{ex:MorseDecompNotInduced}
		Consider the Morse decomposition $\cM$ from Example \ref{ex:MorseDecomposition} and the partition $\big\{\{1,3,4,6\}, \{2,5,8\}, \{7\}\big\}$ of its index set $\P=\{1,...,8\}$. The three induced Morse sets are shown in Figure \ref{fig:MorseDecompNotInducedM}. The connections between each of these Morse sets are summarized by the directed graph in Figure \ref{fig:MorseDecompNotInducedP}. Since this graph is not acyclic, it does not represent a partial order on $\big\{\{1,3,4,6\}, \{2,5,8\}, \{7\}\big\}$. In other words, the connections between the induced Morse sets do not describe a partial order on the given partition of $\P$. Hence, from Theorem \ref{theo:ConditionsForMorseDecomposition}, it follows that the collection $\cM'$ of Morse sets induced by the partition $\big\{\{1,3,4,6\}, \{2,5,8\}, \{7\}\big\}$ is not a Morse decomposition.

		\begin{figure}[h]
			\centering
			\begin{subfigure}[b]{0.55\textwidth}
				\centering
				\begin{tikzpicture}[scale=1]
					\coordinate (D) at (0,0);
					\coordinate (A) at (1,1.732);
					\coordinate (H) at (1,-1.732);
					\coordinate (E) at (2,0);
					\coordinate (B) at (3,1.732);
					\coordinate (I) at (3,-1.732);
					\coordinate (F) at (4,0);
					\coordinate (C) at (5,1.732);
					\coordinate (J) at (5,-1.732);
					\coordinate (G) at (6,0);

					\coordinate (AB) at ($0.5*(A)+0.5*(B)$);
					\coordinate (BC) at ($0.5*(B)+0.5*(C)$);
					\coordinate (AD) at ($0.5*(A)+0.5*(D)$);
					\coordinate (AE) at ($0.5*(A)+0.5*(E)$);
					\coordinate (BE) at ($0.5*(B)+0.5*(E)$);
					\coordinate (BF) at ($0.5*(B)+0.5*(F)$);
					\coordinate (CF) at ($0.5*(C)+0.5*(F)$);
					\coordinate (CG) at ($0.5*(C)+0.5*(G)$);
					\coordinate (DE) at ($0.5*(D)+0.5*(E)$);
					\coordinate (EF) at ($0.5*(E)+0.5*(F)$);
					\coordinate (FG) at ($0.5*(F)+0.5*(G)$);
					\coordinate (DH) at ($0.5*(D)+0.5*(H)$);
					\coordinate (EH) at ($0.5*(E)+0.5*(H)$);
					\coordinate (EI) at ($0.5*(E)+0.5*(I)$);
					\coordinate (FI) at ($0.5*(F)+0.5*(I)$);
					\coordinate (FJ) at ($0.5*(F)+0.5*(J)$);
					\coordinate (GJ) at ($0.5*(G)+0.5*(J)$);
					\coordinate (HI) at ($0.5*(H)+0.5*(I)$);
					\coordinate (IJ) at ($0.5*(I)+0.5*(J)$);

					\coordinate (ADE) at ($0.333*(A)+0.333*(D)+0.333*(E)$);
					\coordinate (ABE) at ($0.333*(A)+0.333*(B)+0.333*(E)$);
					\coordinate (BEF) at ($0.333*(B)+0.333*(E)+0.333*(F)$);
					\coordinate (BCF) at ($0.333*(B)+0.333*(C)+0.333*(F)$);
					\coordinate (CFG) at ($0.333*(C)+0.333*(F)+0.333*(G)$);
					\coordinate (DEH) at ($0.333*(D)+0.333*(E)+0.333*(H)$);
					\coordinate (EHI) at ($0.333*(E)+0.333*(H)+0.333*(I)$);
					\coordinate (EFI) at ($0.333*(E)+0.333*(F)+0.333*(I)$);
					\coordinate (FIJ) at ($0.333*(F)+0.333*(I)+0.333*(J)$);
					\coordinate (FGJ) at ($0.333*(F)+0.333*(G)+0.333*(J)$);

					\fill[black!5] (A) -- (C) -- (G) -- (J) -- (H) -- (D) -- cycle;
					%\fill[red!20] (A) -- (D) -- (E) -- cycle;
					%\fill[red!20] (B) -- (E) -- (F) -- cycle;
					%\fill[red!20] (F) -- (G) -- (J) -- cycle;

					\fill[red!50] ($(I)!\ec!(HI) + (-\wc,-\wc)$) arc(-90:90:\wc) -- ($(H)!\ec!(EH)$) -- ($(D)!\cc!(DEH)$) -- ($(E) + (0,-\wc)$) arc(-90:30:\wc) -- +($(A)-(E)$) arc(30:150:\wc) -- +($(D)-(A)$) arc(150:210:\wc) -- +($(H)-(D)$) arc(210:270:\wc) -- cycle;
					\fill[cyan!50] ($(B)!\ec!(BC) + (\wc,\wc)$) arc(90:270:\wc) -- ($(C)!\ec!(CF)$) -- ($(G)!\cc!(CFG)$) -- ($(F) + (0,\wc)$) arc(90:210:\wc) -- +($(J)-(F)$) arc(210:330:\wc) -- +($(G)-(J)$) arc(-30:30:\wc) -- +($(C)-(G)$) arc(30:90:\wc) -- cycle;
					\fill[green!50] ($(E)!\cc!(BEF)$) -- ($(F)!\cc!(BEF)$)  -- ($(B)!\cc!(BEF)$) -- cycle;

					\draw (A) -- (C) -- (G) -- (J) -- (H) -- (D) -- cycle;
					\draw (D) -- (G)  (A) -- (I) -- (C)  (H) -- (B) -- (J);

					%\draw[red,very thick] (B) -- (C)  (D) -- (E)  (H) -- (I);

					%\node at (A){$\bullet$};
					%\node at (B){$\bullet$};
					%\node at (C){$\bullet$};
					%\node[red] at (D){$\bullet$};
					%\node at (E){$\bullet$};
					%\node at (F){$\bullet$};
					%\node at (G){$\bullet$};
					%\node at (H){$\bullet$};
					%\node at (I){$\bullet$};
					%\node at (J){$\bullet$};

					\node at (D){\LARGE $\bullet$};
					\node at (DE){\LARGE $\bullet$};
					\node at (ADE){\LARGE $\bullet$};
					\node at (BEF){\LARGE $\bullet$};
					\node at (BC){\LARGE $\bullet$};
					\node at (HI){\LARGE $\bullet$};
					\node at (FGJ){\LARGE $\bullet$};

					\draw[ultra thick, -latex] (A) -- (AD);
					\draw[ultra thick, -latex] (AB) -- (ABE);
					\draw[ultra thick, -latex] (B) -- (BE);
					\draw[ultra thick, -latex] (C) -- (CG);
					\draw[ultra thick, -latex] (BF) -- (BCF);
					\draw[ultra thick, -latex] (CF) -- (CFG);
					\draw[ultra thick, -latex] (E) -- (AE);
					\draw[ultra thick, -latex] (EF) -- (EFI);
					\draw[ultra thick, -latex] (F) -- (FJ);
					\draw[ultra thick, -latex] (G) -- (FG);
					\draw[ultra thick, -latex] (EH) -- (DEH);
					\draw[ultra thick, -latex] (EI) -- (EHI);
					\draw[ultra thick, -latex] (FI) -- (FIJ);
					\draw[ultra thick, -latex] (H) -- (DH);
					\draw[ultra thick, -latex] (I) -- (IJ);
					\draw[ultra thick, -latex] (J) -- (GJ);

					\node[xshift=11pt, yshift=-6pt] at (A) {$A$};
					\node[xshift=11pt, yshift=-6pt] at (B) {$B$};
					\node[xshift=11pt, yshift=-6pt] at (C) {$C$};
					\node[xshift=11pt, yshift=6pt] at (D) {$D$};
					\node[xshift=11pt, yshift=6pt] at (E) {$E$};
					\node[xshift=11pt, yshift=6pt] at (F) {$F$};
					\node[xshift=11pt, yshift=6pt] at (G) {$G$};
					\node[xshift=11pt, yshift=6pt] at (H) {$H$};
					\node[xshift=11pt, yshift=6pt] at (I) {$I$};
					\node[xshift=11pt, yshift=6pt] at (J) {$J$};
				\end{tikzpicture}
				\caption{}\label{fig:MorseDecompNotInducedM}
			\end{subfigure}
			\begin{subfigure}[b]{0.37\textwidth}
				\centering
				\begin{tikzpicture}[scale=1]
					\draw[white, opacity=0] (8,-1.732-\wc) rectangle (11, 1.732+\wc);

					\begin{scope}[every node/.style={rectangle,thick,draw}]
						\node[fill=red!50] (M1) at (8,-1) {$\{1,3,4,6\}$};
						\node[fill=cyan!50] (M2) at (11,-1) {$\{2,5,8\}$};
						\node[fill=green!50] (M7) at (9.5,1) {$\{7\}$};
					\end{scope}

					\begin{scope}[>={latex},
						every edge/.style={draw,very thick}]
						\path [->] (M7) edge (M1);
						\path [->] (M7) edge (M2);
						\path [->, bend right=15] (M1) edge (M2);
						\path [->, bend right=15] (M2) edge (M1);
					\end{scope}
				\end{tikzpicture}
				\caption{}\label{fig:MorseDecompNotInducedP}
			\end{subfigure}
			\caption{In \subref{fig:MorseDecompNotInducedM}, the discrete vector field $\cV$ from example \ref{ex:MorseDecomposition} and the induced Morse sets $M(\{1,3,4,6\})$ in red, $M(\{2,5,8\})$ in blue and $M(\{7\})$ in green. In the directed graph in \subref{fig:MorseDecompNotInducedP}, there is an arrow between two nodes iff there is a solution going from one of the associated Morse sets to the other}\label{fig:MorseDecompNotInduced}
		\end{figure}
	\end{ex}

	\begin{ex}\label{ex:MorseDecompInduced}
		Consider again the Morse decomposition $\cM$ from Example \ref{ex:MorseDecomposition} and, this time, the partition $\big\{\{1,3,6\}, \{2,8\}, \{4\}, \{5,7\}\big\}$ of $\P=\{1,...,8\}$. The induced Morse sets are shown in Figure \ref{fig:MorseDecompInducedM} and their connections are as in Figure \ref{fig:MorseDecompInducedP}. The directed graph obtained here is acyclic, which means it represents a partial order on $\big\{\{1,3,6\}, \{2,8\}, \{4\}, \{5,7\}\big\}$. The existence of this partial order is necessary and sufficient to deduce that the induced Morse sets from Figure \ref{fig:MorseDecompInducedM} form a Morse decomposition of the given flow.

		\begin{figure}[h]
			\centering
			\begin{subfigure}[b]{0.55\textwidth}
				\centering
				\begin{tikzpicture}[scale=1]
					\coordinate (D) at (0,0);
					\coordinate (A) at (1,1.732);
					\coordinate (H) at (1,-1.732);
					\coordinate (E) at (2,0);
					\coordinate (B) at (3,1.732);
					\coordinate (I) at (3,-1.732);
					\coordinate (F) at (4,0);
					\coordinate (C) at (5,1.732);
					\coordinate (J) at (5,-1.732);
					\coordinate (G) at (6,0);

					\coordinate (AB) at ($0.5*(A)+0.5*(B)$);
					\coordinate (BC) at ($0.5*(B)+0.5*(C)$);
					\coordinate (AD) at ($0.5*(A)+0.5*(D)$);
					\coordinate (AE) at ($0.5*(A)+0.5*(E)$);
					\coordinate (BE) at ($0.5*(B)+0.5*(E)$);
					\coordinate (BF) at ($0.5*(B)+0.5*(F)$);
					\coordinate (CF) at ($0.5*(C)+0.5*(F)$);
					\coordinate (CG) at ($0.5*(C)+0.5*(G)$);
					\coordinate (DE) at ($0.5*(D)+0.5*(E)$);
					\coordinate (EF) at ($0.5*(E)+0.5*(F)$);
					\coordinate (FG) at ($0.5*(F)+0.5*(G)$);
					\coordinate (DH) at ($0.5*(D)+0.5*(H)$);
					\coordinate (EH) at ($0.5*(E)+0.5*(H)$);
					\coordinate (EI) at ($0.5*(E)+0.5*(I)$);
					\coordinate (FI) at ($0.5*(F)+0.5*(I)$);
					\coordinate (FJ) at ($0.5*(F)+0.5*(J)$);
					\coordinate (GJ) at ($0.5*(G)+0.5*(J)$);
					\coordinate (HI) at ($0.5*(H)+0.5*(I)$);
					\coordinate (IJ) at ($0.5*(I)+0.5*(J)$);

					\coordinate (ADE) at ($0.333*(A)+0.333*(D)+0.333*(E)$);
					\coordinate (ABE) at ($0.333*(A)+0.333*(B)+0.333*(E)$);
					\coordinate (BEF) at ($0.333*(B)+0.333*(E)+0.333*(F)$);
					\coordinate (BCF) at ($0.333*(B)+0.333*(C)+0.333*(F)$);
					\coordinate (CFG) at ($0.333*(C)+0.333*(F)+0.333*(G)$);
					\coordinate (DEH) at ($0.333*(D)+0.333*(E)+0.333*(H)$);
					\coordinate (EHI) at ($0.333*(E)+0.333*(H)+0.333*(I)$);
					\coordinate (EFI) at ($0.333*(E)+0.333*(F)+0.333*(I)$);
					\coordinate (FIJ) at ($0.333*(F)+0.333*(I)+0.333*(J)$);
					\coordinate (FGJ) at ($0.333*(F)+0.333*(G)+0.333*(J)$);

					\fill[black!5] (A) -- (C) -- (G) -- (J) -- (H) -- (D) -- cycle;
					%\fill[red!20] (A) -- (D) -- (E) -- cycle;
					%\fill[red!20] (B) -- (E) -- (F) -- cycle;
					%\fill[red!20] (F) -- (G) -- (J) -- cycle;

					\fill[red!50] ($(E) + (0,-\wc)$) arc(-90:30:\wc) -- +($(A)-(E)$) arc(30:150:\wc) -- +($(D)-(A)$) arc(150:270:\wc) -- cycle;
					\fill[cyan!50] ($(F) + (0,\wc)$) arc(90:210:\wc) -- +($(J)-(F)$) arc(210:330:\wc) -- +($(G)-(J)$) arc(-30:90:\wc) -- cycle;
					\fill[green!50] ($(E)!\cc!(BEF)$) -- ($(F)!\ec!(BF)$)  -- ($(C)!\ec!(BC)$) arc(0:90:\wc) --  ($(B)!\ec!(BC) + (\wc,\wc)$) arc(90:180:\wc) arc(360:330:\wc) -- cycle;
					\fill[yellow] ($(H)!\ec!(HI) + (\wc,\wc)$) arc(90:270:\wc)  -- ($(I)!\ec!(HI) + (-\wc,-\wc)$) arc(-90:90:\wc) -- cycle;

					\draw (A) -- (C) -- (G) -- (J) -- (H) -- (D) -- cycle;
					\draw (D) -- (G)  (A) -- (I) -- (C)  (H) -- (B) -- (J);

					%\draw[red,very thick] (B) -- (C)  (D) -- (E)  (H) -- (I);

					%\node at (A){$\bullet$};
					%\node at (B){$\bullet$};
					%\node at (C){$\bullet$};
					%\node[red] at (D){$\bullet$};
					%\node at (E){$\bullet$};
					%\node at (F){$\bullet$};
					%\node at (G){$\bullet$};
					%\node at (H){$\bullet$};
					%\node at (I){$\bullet$};
					%\node at (J){$\bullet$};

					\node at (D){\LARGE $\bullet$};
					\node at (DE){\LARGE $\bullet$};
					\node at (ADE){\LARGE $\bullet$};
					\node at (BEF){\LARGE $\bullet$};
					\node at (BC){\LARGE $\bullet$};
					\node at (HI){\LARGE $\bullet$};
					\node at (FGJ){\LARGE $\bullet$};

					\draw[ultra thick, -latex] (A) -- (AD);
					\draw[ultra thick, -latex] (AB) -- (ABE);
					\draw[ultra thick, -latex] (B) -- (BE);
					\draw[ultra thick, -latex] (C) -- (CG);
					\draw[ultra thick, -latex] (BF) -- (BCF);
					\draw[ultra thick, -latex] (CF) -- (CFG);
					\draw[ultra thick, -latex] (E) -- (AE);
					\draw[ultra thick, -latex] (EF) -- (EFI);
					\draw[ultra thick, -latex] (F) -- (FJ);
					\draw[ultra thick, -latex] (G) -- (FG);
					\draw[ultra thick, -latex] (EH) -- (DEH);
					\draw[ultra thick, -latex] (EI) -- (EHI);
					\draw[ultra thick, -latex] (FI) -- (FIJ);
					\draw[ultra thick, -latex] (H) -- (DH);
					\draw[ultra thick, -latex] (I) -- (IJ);
					\draw[ultra thick, -latex] (J) -- (GJ);

					\node[xshift=11pt, yshift=-6pt] at (A) {$A$};
					\node[xshift=11pt, yshift=-6pt] at (B) {$B$};
					\node[xshift=11pt, yshift=-6pt] at (C) {$C$};
					\node[xshift=11pt, yshift=6pt] at (D) {$D$};
					\node[xshift=11pt, yshift=6pt] at (E) {$E$};
					\node[xshift=11pt, yshift=6pt] at (F) {$F$};
					\node[xshift=11pt, yshift=6pt] at (G) {$G$};
					\node[xshift=11pt, yshift=6pt] at (H) {$H$};
					\node[xshift=11pt, yshift=6pt] at (I) {$I$};
					\node[xshift=11pt, yshift=6pt] at (J) {$J$};
				\end{tikzpicture}
				\caption{}\label{fig:MorseDecompInducedM}
			\end{subfigure}
			\begin{subfigure}[b]{0.37\textwidth}
				\centering
				\begin{tikzpicture}[scale=1]
					\draw[white, opacity=0] (8,-1.732-\wc) rectangle (11, 1.732+\wc);

					\begin{scope}[every node/.style={rectangle,thick,draw}]
						\node[fill=red!50] (M1) at (8,-1.5) {$\{1,3,6\}$};
						\node[fill=cyan!50] (M2) at (11,-1.5) {$\{2,8\}$};
						\node[fill=yellow] (M4) at (9.5,0) {$\{4\}$};
						\node[fill=green!50] (M7) at (9.5,1.5) {$\{5,7\}$};
					\end{scope}

					\begin{scope}[>={latex},
						every edge/.style={draw,very thick}]
						\path [->] (M4) edge (M1);
						\path [->] (M4) edge (M2);
						\path [->] (M7) edge (M1);
						\path [->] (M7) edge (M2);
						\path [->] (M7) edge (M4);
					\end{scope}
				\end{tikzpicture}
				\caption{}\label{fig:MorseDecompInducedP}
			\end{subfigure}
			\caption{In \subref{fig:MorseDecompInducedM}, the discrete vector field $\cV$ from example \ref{ex:MorseDecomposition} and the induced Morse sets $M(\{1,3,6\})$ in red, $M(\{2,8\})$ in blue, $M(\{4\})$ in yellow and $M(\{5,7\})$ in green. In the directed graph in \subref{fig:MorseDecompInducedP}, there is an arrow between two nodes iff there is a solution going from one of the associated Morse sets to the other}\label{fig:MorseDecompInduced}
		\end{figure}
	\end{ex}

	In order to prove Theorem \ref{theo:ConditionsForMorseDecomposition}, the following lemmas are needed.

	\begin{lem}\label{lem:MrSubsetMIs}
		Let $\cM = \{M_r\ |\ r\in\P\}$ be a Morse decomposition and consider a partition $\{I_s\subseteq\P\ |\ s\in\S\}$ of $\P$. If $r\in I_s$ for some $s\in\S$, then $M_r\subseteq M(I_s)$. Moreover, when the Morse sets in $\cM' = \{M(I_s)\ |\ s\in\S\}$ are mutually disjoint, we have $M_r\cap M(I_{s'})=\emptyset$ for all $s'\in\S$ such that $r\notin I_{s'}$.
	\end{lem}

	\begin{proof}
		From Proposition \ref{prop:MorseDecompositionProperties}\ref{prop:MorseDecompositionPropertiesEnum3}, we know that $M_r = C(M_r,M_r)$, so if $r\in I_s$,
		\begin{gather*}
			M_r\subseteq\bigcup_{r',r''\in I_s}C(M_{r''},M_{r'}) = M(I_s).
		\end{gather*}
		It follows that $M_r\cap M(I_{s'})\subseteq M(I_s)\cap M(I_{s'}) = \emptyset$ for all $s'\neq s$ when the sets in $\cM'$ are mutually disjoint.
	\end{proof}

	\begin{lem}\label{lem:MrConnectedIffMIsConnected}
		Let $\cM = \{M_r\ |\ r\in\P\}$ be a Morse decomposition of a flow $\Pi_\cV$ and consider a partition $\{I_s\subseteq\P\ |\ s\in\S\}$ of $\P$. For all $s,s'\in\S$, we have
		\begin{gather*}
			M(I_{s'})\rightconnects{\cV}M(I_s) \quad\Leftrightarrow\quad M_{r'}\rightconnects{\cV}M_r\text{ for some } r\in I_s,r'\in I_{s'}.
		\end{gather*}
	\end{lem}

	\begin{proof}
		If $M_{r'}\rightconnects{\cV}M_r$ for some $r\in I_s$ and $r'\in I_{s'}$, then obviously $M(I_{s'})\rightconnects{\cV}M(I_s)$ from Lemma \ref{lem:MrSubsetMIs}. Now, suppose $M(I_{s'})\rightconnects{\cV}M(I_s)$. Then, there exists some $\sigma\in M(I_s)$ and some $\sigma'\in M(I_{s'})$ such that $\sigma'\rightconnects{\cV}\sigma$. Also, by definition of a Morse set, $\sigma\in C(M_{r_1},M_{r_2})$ and $\sigma'\in C(M_{r_1'},M_{r_2}')$ for some $r_1,r_2\in I_s$ and $r_1',r_2'\in I_{s'}$, meaning that $M_{r_1}\rightconnects{\cV}\sigma\rightconnects{\cV}M_{r_2}$ and $M_{r_1'}\rightconnects{\cV}\sigma'\rightconnects{\cV}M_{r_2'}$. Hence,
		\begin{gather*}
			M_{r_1'}\rightconnects{\cV}\sigma'\rightconnects{\cV}\sigma\rightconnects{\cV}M_{r_2}.
		\end{gather*}
		Since $M_{r_1'}\subseteq M(I_{s'})$ and $M_{r_2}\subseteq M(I_s)$ by Lemma \ref{lem:MrSubsetMIs}, we have the result.
	\end{proof}

	We are now ready to state Theorem \ref{theo:ConditionsForMorseDecomposition}. It is proven in a general setting, but will be particularly useful in Section \ref{sec:CriticalComponents} to determine under which conditions the critical components of a \mdm function form a Morse decomposition.

	\pagebreak

	\begin{defn}
		Let $\Pi_\cV:K\multimap K$ and consider a collection $\cM=\{M_r\ |\ r\in\P\}$ of subsets of $K$. A \emph{$\cM$-path} is a sequence $r_0,r_1,...,r_n\in\P$ such that $M_{r_0}\rightconnects{\cV}M_{r_1}\rightconnects{\cV}\cdots\rightconnects{\cV} M_{r_n}$ and a \emph{$\cM$-cycle} is a $\cM$-path for which $r_0=r_n$. We say a $\cM$-path or a $\cM$-cycle is \emph{trivial} if $r_0=r_1=\cdots =r_n$.
	\end{defn}

	\begin{thm}\label{theo:ConditionsForMorseDecomposition}
		Let $\cM = \{M_r\ |\ r\in\P\}$ be a Morse decomposition of a flow $\Pi_\cV:K\multimap K$. Consider a partition $\{I_s\subseteq\P\ |\ s\in\S\}$ of $\P$ and the induced collection $\cM' = \{M(I_s)\ |\ s\in\S\}$. The three following statements are equivalent:
		\begin{enumerate}[label=(\alph{enumi})]
			\item\label{prop:ConditionsForMorseDecompositionEnumDecomp} There exists a partial order on $\S$ for which $\cM'$ is a Morse decomposition.

			\item\label{prop:ConditionsForMorseDecompositionEnumNoCycle} There exists no nontrivial $\cM'$-cycle.

			\item\label{prop:ConditionsForMorseDecompositionEnumPreorder} The preorder induced by the relation $R$ defined on $\S$ such that
			\begin{gather*}
				sRs'\quad\Leftrightarrow\quad M(I_s)\leftconnects{\cV}M(I_{s'})
			\end{gather*}
			is a partial order.
		\end{enumerate}
	\end{thm}

	\begin{proof}
		We will prove that \ref{prop:ConditionsForMorseDecompositionEnumDecomp}~$\Rightarrow$~\ref{prop:ConditionsForMorseDecompositionEnumNoCycle}~$\Rightarrow$~\ref{prop:ConditionsForMorseDecompositionEnumPreorder}~$\Rightarrow$~\ref{prop:ConditionsForMorseDecompositionEnumDecomp}.

		First, we see that the proof of \ref{prop:ConditionsForMorseDecompositionEnumDecomp}~$\Rightarrow$~\ref{prop:ConditionsForMorseDecompositionEnumNoCycle} follows from statements \ref{prop:MorseDecompositionPropertiesEnum1} and \ref{prop:MorseDecompositionPropertiesEnum4} of Proposition \ref{prop:MorseDecompositionProperties}. Indeed, if $\cM'$ is a Morse decomposition for some partial order $\leq$ on $\S$ and $s_0,s_1,...,s_n\subseteq\S$ is a sequence such that $s_0=s_n$ and \begin{gather*}
			M(I_{s_0})\rightconnects{\cV}M(I_{s_1})\rightconnects{\cV}\cdots\rightconnects{\cV} M(I_{s_n}),
		\end{gather*}
		then $s_n\leq s_{n-1}\leq\cdots\leq s_1\leq s_0 = s_n$, thus $s_0=s_1=\cdots =s_n$.

		We now show \ref{prop:ConditionsForMorseDecompositionEnumNoCycle}~$\Rightarrow$~\ref{prop:ConditionsForMorseDecompositionEnumPreorder}. Since each $M(I_s)$ is invariant, we can easily see that $R$ is reflexive, meaning that the preorder induced by $R$ is its transitive closure $\bar{R}$. To prove that it is a partial order, we only have to show that $\bar{R}$ is antisymmetric. Let $s,s'\in\S$ be such that $s\bar{R}s'$ and $s'\bar{R}s$. Then, there exists two sequences $s=s_0,s_1,...,s_m=s'$ and $s'=s_m,s_{m+1},...,s_{m+n}=s$ in $\S$ such that $s_{i-1}Rs_i$ for each $i=1,...,m+n$. By definition of $R$ and by \ref{prop:ConditionsForMorseDecompositionEnumNoCycle}, it follows that $s_0 = s_1 =\cdots = s_{m+n}$ and, in particular, $s=s'$.

		We finally prove that \ref{prop:ConditionsForMorseDecompositionEnumPreorder}~$\Rightarrow$~\ref{prop:ConditionsForMorseDecompositionEnumDecomp}. Consider the collection $\cM'$ indexed by the set $\S$ partially ordered by $\bar{R}$. We show that the three conditions of Definition \ref{def:MorseDecomposition} of a Morse decomposition are satisfied.
		\begin{enumerate}
			\item\label{proof:ConditionsForMorseDecompositionEnum1} We know that $M(I_s)$ is an isolated invariant set by Proposition \ref{prop:MorseSetIsolatedInvariant}. Thus, we only have to show that $M(I_s)\cap M(I_{s'})=\emptyset$ for all $s,s'\in\S$ such that $s\neq s'$.

			Suppose $\sigma\in M(I_s)\cap M(I_{s'})$. Then, there exists $r_1,r_2\in I_s$ and $r_1',r_2'\in I_{s'}$ such that $\sigma\in C(M_{r_1},M_{r_2})\cap C(M_{r_1'},M_{r_2'})$, meaning that $M_{r_1}\rightconnects{\cV}\sigma\rightconnects{\cV} M_{r_2}$ and $M_{r_1'}\rightconnects{\cV}\sigma\rightconnects{\cV} M_{r_2'}$. Hence, $M_{r_1}\rightconnects{\cV}\sigma\rightconnects{\cV} M_{r_2'}$ and $M_{r_1'}\rightconnects{\cV}\sigma\rightconnects{\cV} M_{r_2}$. From Lemma \ref{lem:MrConnectedIffMIsConnected}, it follows that $M(I_s)\rightconnects{\cV} M(I_{s'})$ and $M(I_{s'})\rightconnects{\cV} M(I_s)$. By the definition of the partial order on $\S$, we have $s'\leq s$ and $s\leq s'$, thus $s=s'$.

			\item For any solution $\varrho:\Z\rightarrow K$ of $\Pi_\cV$, since $\cM=\{M_r\ |\ r\in\P\}$ is a Morse decomposition, there are some $r\leq r'\in\P$ such that $\alpha(\varrho)\subseteq M_{r'}$ and $\omega(\varrho)\subseteq M_r$. Consider the unique indexes $s,s'\in\S$ such that $r\in I_s$ and $r'\in I_{s'}$. From Lemma \ref{lem:MrSubsetMIs}, we see that $\alpha(\varrho)\subseteq M_{r'}\subseteq M(I_{s'})$ and $\omega(\varrho)\subseteq M_{r}\subseteq M(I_{s})$, where $s\leq s'$ since we then have $M(I_{s'})\rightconnects{\cV}M(I_s)$.

			\item Now, consider a solution $\varrho:\Z\rightarrow K$ such that $\alpha(\varrho)\cup\omega(\varrho)\subseteq M(I_s)$ for some $s\in\S$. Since $\cM$ is a Morse decomposition, we know there exists some $r\leq r'\in\P$ such that $\alpha(\varrho)\subseteq M_{r'}$ and $\omega(\varrho)\subseteq M_r$. Also, we have $\image\varrho\subseteq C(M_{r'}, M_r)$ by definition of a connecting set. Moreover, we see that $M_r\cap M(I_s)\supseteq\omega(\varrho)\neq\emptyset$. Since the sets in $\cM'$ are mutually disjoint, as shown in \ref{proof:ConditionsForMorseDecompositionEnum1}, it follows from Lemma \ref{lem:MrSubsetMIs} that $r\in I_s$. Similarly, $r'\in I_s$. We conclude that $\image\varrho\subseteq C(M_{r'}, M_r)\subseteq M(I_s)$ by definition of a Morse set. \qedhere
		\end{enumerate}
	\end{proof}

	Furthermore, from Lemma \ref{lem:MrConnectedIffMIsConnected}, we see that $\cM'$-cycles can be characterized as follows.

	\begin{prop}\label{prop:MpathsCharacterization}
		Let $\cM = \{M_r\ |\ r\in\P\}$ be a Morse decomposition. Consider a partition $\{I_s\subseteq\P\ |\ s\in\S\}$ of $\P$ and the collection $\cM'=\{M(I_s)\ |\ s\in\S\}$. The sequence $s_0,s_1,...,s_n\in\S$ is a $\cM'$-path if and only if there exists a sequence $r_0',r_1,r_1',r_2,r_2',...,r_{n-1}',r_n\in\P$ such that
		\begin{itemize}[label=$\bullet$]
			\item $r_0'\in I_{s_0}$, $r_i,r_i'\in I_{s_i}$ for each $i=1,...,n-1$ and $r_n\in I_{s_n}$;
			\item $M_{r_{i-1}'}\rightconnects{\cV}M_{r_i}$ for each $i=1,...,n$.
		\end{itemize}
		Moreover, the sequence $s_0,s_1,...,s_n\in\S$ is a $\cM'$-cycle iff there exists such a sequence $r_0',r_1,r_1',r_2,r_2',...,r_{n-1}',r_n\in\P$ and $s_0=s_n$.
	\end{prop}

	\section{Multidimensional discrete Morse functions}\label{sec:MDM}

	In this section, many notions of Morse-Forman theory \citep{Forman1998, Forman2002} are extended to vector-valued functions $f:K\rightarrow\R^\maxdim$. A few concepts on multidimensional discrete Morse (\mdm) functions discussed by~\citet{Allili2019}, which we also call multiparameter in reference to multiparameter persistence, are first presented. The gradient vector field of a \mdm function is then defined and some of its properties are outlined.

	\subsection{Main definitions}

	For the remaining of the article, we note $\preceq$ the partial order on $\R^\maxdim$ such that, for any $a=(a_1,...,a_\maxdim)$ and $b=(b_1,...,b_\maxdim)$ in $\R^\maxdim$,

	\begin{gather*}
		a\preceq b \Leftrightarrow a_i\leq b_i\text{ for each } i=1,...,\maxdim.
	\end{gather*}
	We also write $a\precneqq b$ whenever $a\preceq b$ and $a\neq b$. Moreover, for $f:K\rightarrow\R^\maxdim$ and $\sigma\in K_p$, consider
	\begin{align*}
		H_f(\sigma) &= \left\lbrace \beta\in K_{p+1}\ |\ \beta\supset\sigma\text{ and } f(\beta)\preceq f(\sigma)\right\rbrace;\\
		T_f(\sigma) &= \left\lbrace \alpha\in K_{p-1}\ |\ \alpha\subset\sigma\text{ and } f(\alpha)\succeq f(\sigma)\right\rbrace.
	\end{align*}
	When $f_1,...,f_\maxdim:K\rightarrow\R$ are clear from the context, we also write
	\begin{align*}
		H_i(\sigma) := H_{f_i}(\sigma) &= \left\lbrace\beta\in K_{p+1}\ |\ \beta \supset \sigma\text{ and } f_i(\beta)\leq f_i(\sigma)\right\rbrace,\\
		T_i(\sigma) := T_{f_i}(\sigma) &= \left\lbrace\alpha\in K_{p-1}\ |\ \alpha \subset \sigma\text{ and } f_i(\alpha)\geq f_i(\sigma)\right\rbrace.
	\end{align*}
	We immediately see that for all $f=(f_1,...,f_\maxdim):K\rightarrow\R^\maxdim$ and $\sigma\in K$, we have
	\begin{gather*}
		H_f(\sigma) = \bigcap_{i=1}^\maxdim H_i(\sigma),\qquad T_f(\sigma) = \bigcap_{i=1}^\maxdim T_i(\sigma),
	\end{gather*}
	thus $H_f(\sigma)\subseteq H_i(\sigma)$ and $T_f(\sigma)\subseteq T_i(\sigma)$ for all $i=1,...,\maxdim$. Those observations will come in handy later in this section.

	\pagebreak

	\begin{defn}[Multidimensional discrete Morse function]\label{def:MDM}
		A \emph{multiparameter} or \emph{multidimensional discrete Morse function} (or simply a \mdm function) defined on a simplicial complex $K$ is a function $f:K\rightarrow\R^\maxdim$ such that, for all $\sigma\in K_p$:
		\begin{enumerate}
			\item\label{def:MDMenum1} $\card H_f(\sigma) \leq 1$;
			\item\label{def:MDMenum2} $\card T_f(\sigma) \leq 1$;
			\item\label{def:MDMenum3} if $\beta^{(p+1)}\supset\sigma$ is not in $H_f(\sigma)$, then $f(\beta)\succneqq f(\sigma)$;
			\item\label{def:MDMenum4} if $\alpha^{(p-1)}\subset\sigma$ is not in $T_f(\sigma)$, then $f(\alpha)\precneqq f(\sigma)$.
		\end{enumerate}
	\end{defn}

	Conditions \ref{def:MDMenum1} and \ref{def:MDMenum2} are analogous to those of a one-dimensional discrete Morse function. Conditions \ref{def:MDMenum3} and \ref{def:MDMenum4}, on the other hand, need to be added in the multidimensional setting to ensure the values of $f$ are comparable at least for simplices that are facets and cofacets of each other.

	Furthermore, a key concept in discrete Morse theory is that of critical points. They can be defined just as in the original setting.

	\begin{defn}[Critical point]
		Let $f:K\rightarrow\R^\maxdim$ be \mdm. A simplex $\sigma\in K_p$ is said to be a \emph{critical simplex} or a \emph{critical point of index $p$ of $f$} if
		\begin{align*}
			\card H_f(\sigma) = \card T_f(\sigma) = 0.
		\end{align*}
		A simplex that is not critical is \emph{regular}.
	\end{defn}

	It was shown by \citet{Allili2019} that, as in the one-dimensional setting, for all \mdm function $f:K\rightarrow\R^\maxdim$ and all $\sigma\in K$, one of the sets $H_f(\sigma)$ or $T_f(\sigma)$ must have cardinality zero. The next result follows.

	\begin{prop}
		Let $f:K\rightarrow\R^\maxdim$ be \mdm. Every $\sigma\in K$ verifies exactly one of these conditions:
		\begin{itemize}
			\item $\sigma$ is critical;
			\item $\card H_f(\sigma) =0$ and $\card T_f(\sigma) = 1$;
			\item $\card H_f(\sigma) =1$ and $\card T_f(\sigma) = 0$.
		\end{itemize}
	\end{prop}

	This last observation leads to the definition of the gradient vector field of a \mdm function.

	\begin{defn}[Gradient vector field]
		The \emph{gradient vector field}, or simply the \emph{gradient}, of a \mdm function $f:K\rightarrow\R^\maxdim$ is the discrete vector field $\cV$ such that $\dom\cV = \left\lbrace \sigma\in K\ |\ \card T_f(\sigma) = 0\right\rbrace$ and, for all $\sigma\in\dom\cV$,
		\begin{align*}
			\cV(\sigma) = \begin{cases}
				\sigma & \text{ if } \card H_f(\sigma)=0,\\
				\beta & \text{ if } H_f(\sigma) = \{\beta\}\text{ for some }\beta\supset\sigma.
			\end{cases}
		\end{align*}
	\end{defn}
	We could easily verify that a gradient vector field $\cV$ of a \mdm function $f$ as defined above is indeed a discrete vector field. Moreover:
	\begin{itemize}
		\item the fixed points of $\cV$ are the critical points of $f$;
		\item $\image\cV\backslash\Fix\cV = \left\lbrace \sigma\in K\ |\ \card T_f(\sigma)=1\right\rbrace$;
		\item $\dom\cV\backslash\Fix\cV = \left\lbrace \sigma\in K\ |\ \card H_f(\sigma)=1\right\rbrace$.
	\end{itemize}
	Thus, we can see that the gradient vector field defined here, although it is seen as a partial map $\cV:K\nrightarrow K$, follows the idea of the gradient of a real-valued discrete Morse function as defined by \citet{Forman1998}. Indeed, every $\sigma\in K$ such that $\card T_f(\sigma)=1$ is at the head of an arrow of $\cV$ (the element of $T_f(\sigma)$ being at the tail of that arrow) while every $\sigma\in K$ such that $\card H_f(\sigma)=1$ is at the tail of an arrow (the element of $H_f(\sigma)$ being at the head of that arrow). The main difference here is that the critical points of $f$ are considered as fixed points of its gradient field, whereas they were simply not included in the original definition.

	\begin{ex}\label{ex:MDMfunction}
		Consider $f=(f_1,f_2):K\rightarrow\R^2$, the function defined as in Figure \ref{fig:MDMfunctionExample}. We could verify that it is \mdm and that its gradient field is as represented in the figure. Notice that $f_2$ is not discrete Morse. Indeed, we can see that there are some $\sigma\in K$ for which $\card H_2(\sigma) = 2$ or $\card T_2(\sigma) = 2$.

		\begin{figure}[h]
			\centering
			\begin{tikzpicture}[scale=2.5]
				\coordinate (A) at (0,0);
				\coordinate (B) at (1,1.732);
				\coordinate (C) at (2,0);
				\coordinate (D) at (3,1.732);
				\coordinate (E) at (4,0);

				\coordinate (AB) at ($0.5*(A)+0.5*(B)$);
				\coordinate (AC) at ($0.5*(A)+0.5*(C)$);
				\coordinate (BC) at ($0.5*(B)+0.5*(C)$);
				\coordinate (BD) at ($0.5*(B)+0.5*(D)$);
				\coordinate (CD) at ($0.5*(C)+0.5*(D)$);
				\coordinate (CE) at ($0.5*(C)+0.5*(E)$);
				\coordinate (DE) at ($0.5*(D)+0.5*(E)$);

				\coordinate (ABC) at ($0.333*(A)+0.333*(B)+0.333*(C)$);
				\coordinate (BCD) at ($0.333*(B)+0.333*(C)+0.333*(D)$);
				\coordinate (CDE) at ($0.333*(C)+0.333*(D)+0.333*(E)$);

				\fill[black!5] (A) -- (B) -- (D) -- (E) -- cycle;
				\fill[red!20] (B) -- (C) -- (D) -- cycle;
				\draw (C) -- (A) -- (B) -- (C) -- (E) -- (D);
				\draw[red,very thick] (B) -- (D)node{$\bullet$} -- (C);
				\node[red] at (A){$\bullet$};
				\node at (B){$\bullet$};
				\node at (C){$\bullet$};
				\node at (E){$\bullet$};
				\draw[ultra thick, -latex] (B) -- (AB);
				\draw[ultra thick, -latex] (BC) -- (ABC);
				\draw[ultra thick, -latex] (C) -- (AC);
				\draw[ultra thick, -latex] (DE) -- (CDE);
				\draw[ultra thick, -latex] (E) -- (CE);
				\node[below left] at (A) {$(0,0)$};
				\node[above left] at (B) {$(1,1)$};
				\node[below] at (C) {$(2,0)$};
				\node[above right] at (D) {$(4,1)$};
				\node[below right] at (E) {$(6,0)$};
				\node[left] at (AB) {$(1,1)$};
				\node[below] at (AC) {$(1,0)$};
				\node[right] at (BC) {$(3,1)$};
				\node[above] at (BD) {$(4,2)$};
				\node[left] at (CD) {$(5,1)$};
				\node[below] at (CE) {$(5,0)$};
				\node[right] at (DE) {$(7,1)$};
				\node[below] at (ABC) {$(3,1)$};
				\node[above] at (BCD) {$(5,2)$};
				\node[below] at (CDE) {$(7,1)$};
			\end{tikzpicture}
			\caption{A \mdm function and its gradient vector field. The critical simplices are represented in red}\label{fig:MDMfunctionExample}
		\end{figure}
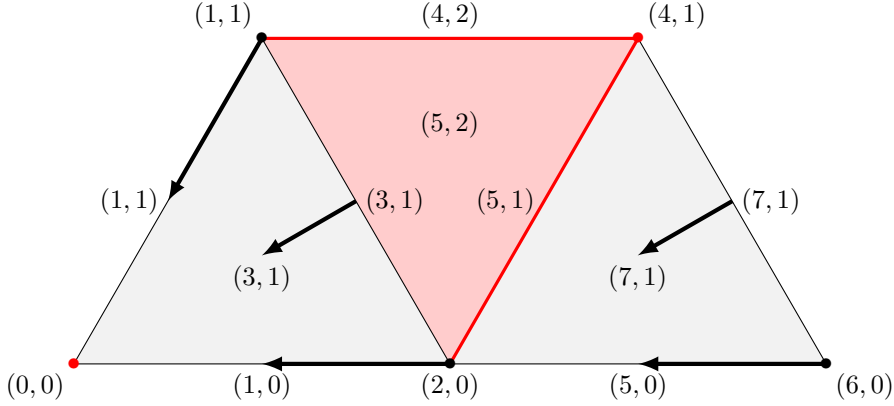
	\end{ex}

	\subsection{Vector-valued functions with discrete Morse components}

	Consider a vector-valued function $f=(f_1,...,f_\maxdim):K\rightarrow\R^\maxdim$. From Example \ref{ex:MDMfunction}, we can see that $f$ being \mdm does not guarantee each $f_i$ is discrete Morse. Conversely, $f$ is not always \mdm even if all its components are discrete Morse. Indeed, if each $f_i$ is discrete Morse, we could easily verify that $f$ satisfies conditions \ref{def:MDMenum1} and \ref{def:MDMenum2} of the definition of a \mdm function, but not necessarily conditions \ref{def:MDMenum3} and \ref{def:MDMenum4}. Actually, we could show that $\maxdim$ discrete Morse functions $f_1,...,f_\maxdim$ form a \mdm function $f=(f_1,...,f_\maxdim)$ if and only if for every pair of facet and cofacet $\alpha^{(p)}\subset\beta^{(p+1)}$, the vectors $f(\alpha)$ and $f(\beta)$ are comparable in the partial order $\preceq$.

	Whenever a \mdm function $f=(f_1,...,f_\maxdim)$ is such that each of its component $f_i$ is discrete Morse, we have the following result on the gradient vector fields of $f$ and each $f_i$.

	\begin{prop}\label{prop:GradientMorseComponents}
		Let $f=(f_1,...,f_\maxdim):K\rightarrow\R^\maxdim$ be \mdm and each $f_i$ be discrete Morse. Consider $\cV$ and $\cV_i$, the gradient vector fields of $f$ and $f_i$, respectively, for each $i=1,...,\maxdim$. Then, $\dom\cV = \left\lbrace\sigma\in K\ |\ \card\left(\bigcap_{i=1}^\maxdim T_i(\sigma)\right) = 0 \right\rbrace$ and
		\begin{align*}
			\cV(\sigma) = \begin{cases}
				\beta & \begin{aligned}&\text{ if }\cV_i(\sigma)\text{ is defined for each }i=1,...,\maxdim\\ &\text{ and } \cV_1(\sigma)= \cV_2(\sigma) = \cdots = \cV_\maxdim(\sigma) =: \beta \supset\sigma,\end{aligned}\\
				\sigma & \text{ otherwise.}
			\end{cases}
		\end{align*}
	\end{prop}

	\begin{proof}
		First, for each $\sigma\in K$, we know that $T_f(\sigma) = \bigcap_{i=1}^\maxdim T_i(\sigma)$. Hence, by definition of $\dom\cV$, we have directly $\dom\cV = \left\lbrace\sigma\in K\ |\ \card\left(\bigcap_{i=1}^\maxdim T_i(\sigma)\right) = 0\right\rbrace$.

		Now, consider $\sigma\in\dom\cV$. By definition of the gradient $\cV$, we have that $\cV(\sigma)~=~\beta$ if $H_f(\sigma)~=~\{\beta\}$ for some $\beta\supset\sigma$ and $\cV(\sigma)=\sigma$ otherwise. Also, we know that $H_f(\sigma)=\bigcap_{i=1}^\maxdim H_i(\sigma)$ and $\card H_i(\sigma)\leq 1$ for each $i=1,...,\maxdim$, so $H_f(\sigma)=\{\beta\}$ if and only if $H_i(\sigma)=\{\beta\}$ for each $i=1,...,\maxdim$. Hence, by definition of each $\cV_i$, we have $\cV(\sigma)=\beta$ for some $\beta\supset\sigma$ if and only if $\cV_i(\sigma)=\beta$ for each $i=1,...,\maxdim$.
	\end{proof}

	Put simply, this last result states that, for a \mdm function $f$ with discrete Morse components $f_1,...,f_\maxdim$, there is an arrow in $\cV$ going from a simplex $\sigma$ to its cofacet $\beta$ if and only if there is an arrow in each $\cV_i$ going from $\sigma$ to $\beta$.

	\begin{ex}\label{ex:GradientMorseComponents}
		Assume $f=(f_1,f_2)$ is \mdm and its components $f_1$ and $f_2$ are discrete Morse. Suppose the gradient vector fields of $f_1$ and $f_2$ are the ones represented in orange and blue in Figure \ref{fig:GradientMorseComponentsExample1}, where the dots represent fixed points. Then, the gradient of $f$ has to be the one shown in Figure \ref{fig:GradientMorseComponentsExample2}, where the critical simplices are shown in red.

		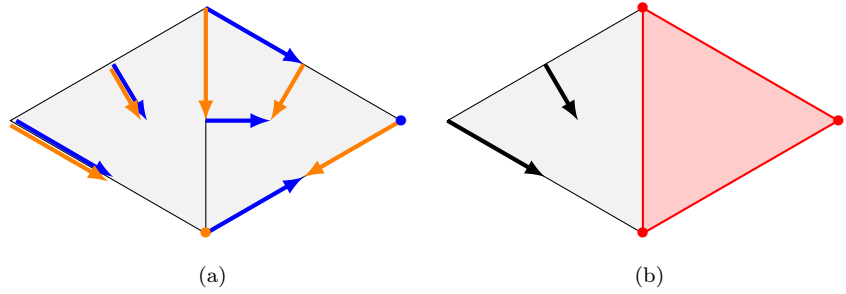
\begin{figure}[h]
			\centering
			\begin{subfigure}{0.45\textwidth}
				\centering
				\begin{tikzpicture}[scale=3]
					\fill[black!5] (0,0) -- (-0.866,0.5) -- (0,1) -- (0.866,0.5) -- cycle;
					\draw (0,0) -- (0,1) -- (-0.866,0.5) -- (0,0) -- (0.866,0.5) -- (0,1);
					\draw[ultra thick, blue, -latex] (-0.84,0.5) -- (-0.41,0.25);
					\draw[ultra thick, blue, -latex] (-0.41,0.75) -- (-0.26,0.5);
					\draw[ultra thick, blue, -latex] (0,0) -- (0.433,0.25);
					\draw[ultra thick, blue, -latex] (0,0.5) -- (0.287,0.5);
					\draw[ultra thick, blue, -latex] (0,1) -- (0.433,0.75);
					\draw[ultra thick, orange, -latex] (-0.866,0.48) -- (-0.433,0.23);
					\draw[ultra thick, orange, -latex] (-0.42,0.73) -- (-0.287,0.5);
					\draw[ultra thick, orange, -latex] (0,1) -- (0,0.5);
					\draw[ultra thick, orange, -latex] (0.433,0.75) -- (0.287,0.5);
					\draw[ultra thick, orange, -latex] (0.866,0.5) -- (0.433,0.25);
					\draw[blue] (0.866,0.5)node{$\bullet$};
					\draw[orange] (0,0)node{$\bullet$};
				\end{tikzpicture}
				\caption{}\label{fig:GradientMorseComponentsExample1}
			\end{subfigure}
			\begin{subfigure}{0.45\textwidth}
				\centering
				\begin{tikzpicture}[scale=3]
					\fill[black!5] (0,0) -- (0,1) -- (-0.866,0.5) -- cycle;
					\fill[red!20] (0,0) -- (0,1) -- (0.866,0.5) -- cycle;
					\draw (0,1) -- (-0.866,0.5) -- (0,0);
					\draw[red,thick] (0,0)node{$\bullet$} -- (0,1)node{$\bullet$} -- (0.866,0.5)node{$\bullet$} -- cycle;
					\draw[ultra thick, -latex] (-0.866,0.5) -- (-0.433,0.25);
					\draw[ultra thick, -latex] (-0.433,0.75) -- (-0.287,0.5);
				\end{tikzpicture}
				\caption{}\label{fig:GradientMorseComponentsExample2}
			\end{subfigure}
			\caption{In \subref{fig:GradientMorseComponentsExample1}, gradient fields of discrete Morse functions $f_1$ and $f_2$ are represented in orange and blue with critical simplices represented with dots. In \subref{fig:GradientMorseComponentsExample2}, the gradient field of \mdm function $f~=~(f_1,f_2)$ with critical simplices in red }\label{fig:GradientMorseComponentsExample}
		\end{figure}
	\end{ex}

	\subsection{Acyclicity of a gradient vector field}

	As in the one-dimensional setting, gradient vector fields of \mdm functions are necessarily acyclic and, conversely, every acyclic vector field represents the gradient of some \mdm function.

	\begin{lem}\label{lem:ComponentIdenticalVectorFieldsMDM}
		Let $f_1,...,f_\maxdim:K\rightarrow\R$ be discrete Morse functions. If $f_1,...,f_\maxdim$ have identical gradient vector fields $\cV_1 = \cV_2 = \cdots = \cV_\maxdim =:\cV$, then $f=(f_1,...,f_\maxdim)$ is \mdm and its gradient is $\cV$.
	\end{lem}

	\begin{proof}
		Let $\sigma\in K$. Notice that if $\cV_1 = \cV_2 = \cdots = \cV_\maxdim$, then we must have $H_1(\sigma)~=~H_2(\sigma)~=~\cdots~=~H_\maxdim(\sigma)$ and $T_1(\sigma) = T_2(\sigma) = \cdots = T_\maxdim(\sigma)$. Indeed, if $H_j(\sigma)=\{\beta\}$ for some $j=1,...,\maxdim$, then $\cV_j(\sigma) = \beta$, so $\cV_i(\sigma) = \beta$ for all $i=1,...,\maxdim$. It follows that $H_1(\sigma) = H_2(\sigma) = \cdots = H_\maxdim(\sigma)$. Moreover, if $T_j(\sigma) = \{\alpha\}$ for some $j=1,...,\maxdim$, then $H_j(\alpha) = \{\sigma\}$. Thus, $H_i(\alpha) = \{\sigma\}$ and $T_i(\sigma) = \{\alpha\}$ for all $i=1,...,\maxdim$, hence $T_1(\sigma) = T_2(\sigma) = \cdots = T_\maxdim(\sigma)$.

		From these observations and using Proposition \ref{prop:GradientMorseComponents}, assuming $f$ is indeed \mdm, we easily see that the gradient $\cV$ of $f$ is such that $\dom\cV=\dom\cV_i$ for all $i=1,...,\maxdim$ and $\cV(\sigma)=\cV_i(\sigma)$ for all $\sigma\in\dom\cV$, so $\cV=\cV_1=\cV_2=\cdots=\cV_\maxdim$.

		All that is left to prove now is that $f$ is a \mdm function. Because $H_f(\sigma) = \bigcap_{i=1}^\maxdim H_i(\sigma)$ and $T_f(\sigma) = \bigcap_{i=1}^\maxdim T_i(\sigma)$ for all $\sigma\in K$, we see that $H_f(\sigma) = H_i(\sigma)$ and $T_f(\sigma) = T_i(\sigma)$ for every $i=1,...,\maxdim$, so the four conditions of a \mdm function follow.
		\begin{itemize}
			\item[\ref{def:MDMenum1}] $\card H_f(\sigma) = \card H_i(\sigma)\leq 1$ for any $i\in\{1,...,\maxdim\}$.
			\item[\ref{def:MDMenum2}] $\card T_f(\sigma) = \card T_i(\sigma)\leq 1$ for any $i\in\{1,...,\maxdim\}$.
			\item[\ref{def:MDMenum3}] If $\beta^{(p+1)}\supset\sigma$ is not in $H_f(\sigma)$, since $H_f(\sigma)=H_i(\sigma)$ for each $i=1,...,\maxdim$, it follows that $\beta\notin H_i(\sigma)$ for each $i=1,...,\maxdim$. Consequently, $f_i(\beta) > f_i(\sigma)$ for all $i=1,...,\maxdim$, so $f(\beta)\succneqq f(\sigma)$.
			\item[\ref{def:MDMenum4}] Is shown similarly to \ref{def:MDMenum3}.\qedhere
		\end{itemize}
	\end{proof}

	\begin{prop}\label{prop:AcyclicityGradient}
		Let $\cV$ be a discrete vector field on a simplicial complex $K$. The field $\cV$ is acyclic if and only if $\cV$ is the gradient of some \mdm function $f:K\rightarrow\R^\maxdim$.
	\end{prop}

	\begin{proof}
		First, suppose $\cV$ is acyclic. It is known that for any acyclic discrete vector field $\cV$, there exists a discrete Morse function $g:K\rightarrow\R$ for which $\cV$ is the gradient \citep{Forman1998}. From Lemma \ref{lem:ComponentIdenticalVectorFieldsMDM}, it follows that $f=(g,g,...,g):K\rightarrow\R^\maxdim$ is \mdm and its gradient vector field is $\cV$.

		Now, let $\cV$ be the gradient of some \mdm function $f:K\rightarrow\R^\maxdim$. For all $\cV$-path $\alpha_0^{(p)},\beta_0^{(p+1)},\alpha_1^{(p)},...,\beta_{r-1}^{(p+1)},\alpha_r^{(p)}$, we see that
		\begin{align*}
			f(\alpha_0)\succeq f(\beta_0) \succneqq f(\alpha_1) \succeq f(\beta_1) \succneqq \cdots \succeq f(\beta_{r-1})\succneqq f(\alpha_r).
		\end{align*}
		If this $\cV$-path is non-trivial, i.e. $r\geq 1$, it follows that $f(\alpha_0)\neq f(\alpha_r)$, thus $\alpha_0\neq\alpha_r$. Hence, a non-trivial $\cV$-path cannot be closed, meaning that $\cV$ is acyclic.
	\end{proof}

	This last proposition is well known in the one-dimensional setting. Actually, the reasoning used in the second part of the previous proof is also used to prove the analogous result in the original theory.

	Furthermore, notice that there is no restriction on the dimension of the codomain $\R^\maxdim$ of the \mdm function in Proposition \ref{prop:AcyclicityGradient}. Therefore, for any integers $\maxdim,\maxdim'\geq 1$, there exists a \mdm function $f:K\rightarrow\R^\maxdim$ having $\cV$ as a gradient field if and only if $\cV$ is acyclic if and only if there exists a \mdm function $g:K\rightarrow\R^{\maxdim'}$ having $\cV$ as a gradient field, thus the following corollary.

	\begin{cor}\label{coro:ExistenceMDMwithSameGradient}
		Let $f:K\rightarrow\R^\maxdim$ be \mdm. For all integer $\maxdim'\geq 1$, there exists a \mdm function $g:K\rightarrow\R^{\maxdim'}$ which has the same gradient vector field, hence the same critical points, as $f$.
	\end{cor}

	In particular, we see from this result that for any \mdm function, there exists a real-valued discrete Morse function having the same gradient vector field.

	\section{Flow of a multidimensional discrete Morse function}\label{sec:flow-MDM}

	Here, we outline some results on the flow associated to the gradient field of a \mdm function. We first present some direct consequences of the acyclicity of a gradient field, including the Morse inequalities which are central in other variants of Morse theory. Then, we show some interesting properties of the image of a solution for such a flow.

	\subsection{Gradient flow and finest Morse decomposition}

	Let $\cV$ be the gradient field of a \mdm function $f:K\rightarrow\R^\maxdim$. We note $\Pi_f:=\Pi_\cV$ the \emph{gradient flow} of $f$. Similarly, we write $\sigma\rightconnects{f}\tau$ rather than $\sigma\rightconnects{\cV}\tau$ when there is a solution in $\Pi_f$ going from $\sigma$ to $\tau$. From the definition of a flow, we find directly the following characterization of $\Pi_f$.

	\begin{prop}\label{prop:GradientFlow}
		Let $f:K\rightarrow\R^\maxdim$ be \mdm. The gradient flow of $f$ is
		\begin{align*}
			\Pi_f(\sigma) = \begin{cases}
				\Cl\sigma & \text{ if }\sigma \text{ is critical for }f, \\
				\Ex\sigma\backslash\{\alpha\} & \text{ if } T_f(\sigma) = \{\alpha\} \text{ for some facet } \alpha\subset\sigma, \\
				\{\beta\} & \text{ if } H_f(\sigma) = \{\beta\}\text{ for some cofacet } \beta\supset\sigma.
			\end{cases}
		\end{align*}
	\end{prop}

	For a gradient flow, we see that $\Pi_f(\sigma) = H_f(\sigma)$ when $\card H_f(\sigma)=1$ and $\Pi_f(\sigma)\subseteq\Cl\sigma$ otherwise. Moreover, we know from Proposition \ref{prop:AcyclicityGradient} that the gradient of a \mdm function is always acyclic and from Proposition \ref{prop:AcyclictyFieldEquiFlow} that a discrete vector field is acyclic if and only if its associated flow also is. The next result follows.

	\begin{prop}\label{prop:AcyclicityGradientFlow}
		The gradient flow $\Pi_f$ of a \mdm function $f:K\rightarrow\R^\maxdim$ is always acyclic.
	\end{prop}

	Also, the following proposition is a direct consequence of Theorem \ref{theo:BasicMinimalMorseDecomposition} and Proposition \ref{prop:BasicSetsAcyclicField}.

	\begin{prop}\label{prop:MinimalMorseDecompositionMDM}
		Let $f:K\rightarrow\R^\maxdim$ be \mdm. The collection of basic sets of $\Pi_f$ is
		\begin{gather*}
			\cM = \big\lbrace\,\{\sigma\}\subseteq K\ |\ \sigma\text{ is critical for } f \big\rbrace.
		\end{gather*}
		Therefore, $\cM$ is the finest Morse decomposition of $\Pi_f$.
	\end{prop}

	For any critical point $\sigma^{(p)}\in K$ of a \mdm function $f:K\rightarrow\R^\maxdim$, we can verify that the Poincar{\'e} polynomial of $\{\sigma\}$ is simply $P_{\{\sigma\}}(t) = t^p$. Hence, using Proposition \ref{prop:MorseEquationMorseDecomposition} and Corollary \ref{coro:MorseInequalitiesMorseDecomposition}, we see that this Morse decomposition leads to the following Morse equation and inequalities.

	\begin{prop}\label{prop:MorseEquationMDMfunction}
		Let $f:K\rightarrow\R^\maxdim$ be \mdm with $\dim K = n$. Let $m_p$ be the number of critical points of index $p$ of $f$. We have the following Morse equation:
		\begin{gather*}
			\sum_{p=0}^nm_pt^p = P_K(t) + (1+t)Q(t)
		\end{gather*}
		for some polynomial $Q(t)$ with non-negative coefficients.
	\end{prop}

	\begin{cor}\label{coro:MorseInequalitiesMDMfunction}
		Let $f:K\rightarrow\R^\maxdim$ be \mdm with $\dim K = n$. Let $m_p$ be the number of critical points of index $p$ of $f$. For all $p=0,1,...,n$, we have
		\begin{enumerate}
			\item the strong Morse inequalities:
			\begin{gather*}
				m_p - m_{p-1} + \cdots +(-1)^p m_0 \geq \beta_p(K) - \beta_{p-1}(K) + \cdots +(-1)^p \beta_0(K);
			\end{gather*}
			\item the weak Morse inequalities:
			\begin{gather*}
				m_p \geq \beta_p(K);
			\end{gather*}
			\item the following equality linking the Euler characteristic $\chi(K)$ of $K$ and the critical points of $f$:
			\begin{gather*}
				m_0 - m_1 + \cdots +(-1)^n m_n = \beta_0(K) - \beta_1(K) + \cdots +(-1)^n \beta_n(K) = \chi(K)
			\end{gather*}
		\end{enumerate}
	\end{cor}

	These Morse inequalities are completely analogous to those from the classical and discrete Morse theories. In section \ref{sec:CriticalComponents}, we will see that there also exists other inequalities which are specific to the multiparameter extension.

	\subsection{Properties of the image of a solution}
	As mentioned in the proof of Proposition \ref{prop:AcyclicityGradient}, given a \mdm function $f:K\rightarrow\R^\maxdim$ with gradient $\cV$, it is clear that for all $\cV$-path $\alpha_0^{(p)},\beta_0^{(p+1)},\alpha_1^{(p)},...,\beta_{r-1}^{(p+1)},\alpha_r^{(p)}$, we have
	\begin{align*}
		f(\alpha_0)\succeq f(\beta_0) \succneqq f(\alpha_1) \succeq f(\beta_1) \succneqq \cdots \succeq f(\beta_{r-1})\succneqq f(\alpha_r).
	\end{align*}
	We could think a similar property exists for solutions of $\Pi_f$. For instance, we could expect that $f(\sigma)\succeq f(\tau)$ for all $\sigma\in K$ and all $\tau\in\Pi_f(\sigma)$. However, it is not always the case, not even in the one-dimensional setting, as shown in Figure \ref{fig:NonDescendingSolution}.

	\begin{figure}[h]
		\centering
		\begin{tikzpicture}[scale=3]
			\draw (0,0.866) -- (1,0.866) -- (0.5,0) -- cycle;
			\draw[red] (0.5,0)node{$\bullet$};
			\draw[ultra thick,-latex] (0,0.866) -- (0.25,0.433);
			\draw[ultra thick,-latex] (0.5,0.866) -- (0.5,0.433);
			\draw[ultra thick,-latex] (1,0.866) -- (0.75,0.433);
			\draw (0.5,0)node[below]{$0$}
			(0.25,0.433)node[below left]{$1$}
			(0.5,0.433)node[below]{$3$}
			(0.75,0.433)node[below right]{$2$}
			(0,0.866)node[left]{$4$}
			(0.5,0.866)node[above]{$6$}
			(1,0.866)node[right]{$5$};
		\end{tikzpicture}
		\caption{Discrete Morse function $f$ such that, for some simplices $\sigma$ and $\tau$, we have $f(\sigma) < f(\tau)$ and $\tau\in\Pi_f(\sigma)$. The red dot represents the critical point of $f$}\label{fig:NonDescendingSolution}
	\end{figure}
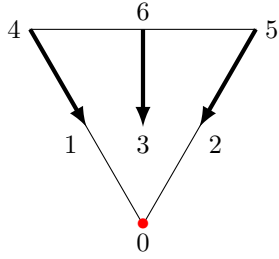

	Indeed, let $f:K\rightarrow\R$ be the discrete Morse function defined in Figure \ref{fig:NonDescendingSolution} and note $\sigma,\tau_1,\tau_2\in K$ the simplices for which $f(\sigma) = 3$, $f(\tau_1) = 4$ and $f(\tau_2) = 5$. We see that for $i=1,2$, we have $\tau_i\in\Pi_f(\sigma)$ but $f(\sigma) < f(\tau_i)$.

	Nonetheless, we can prove that, in most cases, $\sigma\rightconnects{f}\tau$ implies $f(\sigma)\succeq f(\tau)$.

	\begin{lem}\label{lem:DescendingImagePi}
		Let $f:K\rightarrow\R^\maxdim$ be \mdm and consider two simplices $\sigma,\tau\in K$ such that $\tau\in \Cl\sigma\backslash\Cl T_f(\sigma)$. Necessarily, we have $f(\sigma)\succeq f(\tau)$, and the equality is only verified when $\sigma=\tau$.
	\end{lem}

	\begin{proof}
		Let $\tau\in\Cl\sigma\backslash \Cl T_f(\sigma)$ and write $\dim\tau=p$ and $\dim\sigma=p+q$. The result is shown by induction on $q = \dim\sigma-\dim\tau$. For $q=0$, we obviously have $\sigma=\tau$, so $f(\sigma)=f(\tau)$. For $q=1$, $\tau$ is then a facet of $\sigma$ such that $\tau\notin T_f(\sigma)$, so $f(\sigma)\succneqq f(\tau)$ by definition of a \mdm function.

		When $q>1$, we could show there are two different simplices $\beta_1,\beta_2\in K_{p+1}$ such that $\tau\subset\beta_1\subset\sigma$ and $\tau\subset\beta_2\subset\sigma$. By definition of a \mdm function, at least one $\beta\in\{\beta_1,\beta_2\}$ is such that $\beta\notin H_f(\tau)$, so $f(\beta)\succneqq f(\tau)$. Moreover, since $\tau\subset\beta$ and $\tau\notin\Cl T_f(\sigma)$, we see that $\beta\notin\Cl T_f(\sigma)$, meaning that $\beta\in\Cl\sigma\backslash\Cl T_f(\sigma)$. By the induction hypothesis, we then have $f(\sigma)\succneqq f(\beta)$, so
		\begin{gather*}
			f(\sigma) \succneqq f(\beta) \succneqq f(\tau). \qedhere
		\end{gather*}
	\end{proof}

	\begin{lem}\label{lem:DescendingConditions}
		Let $f:K\rightarrow\R^\maxdim$ be \mdm and consider $\sigma,\tau\in K$ such that $\tau\in \Pi_f(\sigma)$. Suppose at least one of the following conditions is satisfied:
		\begin{enumerate}[label=(\alph*)]
			\item\label{lem:DescendingConditionsEnum1} $\card T_f(\sigma)=0$;
			\item\label{lem:DescendingConditionsEnum2} $\tau\notin\Ex T_f(\sigma)$;
			\item\label{lem:DescendingConditionsEnum3} for every $\beta^{(p+1)}\supset\tau^{(p)}$ such that $\beta\in\Cl\sigma\backslash\Cl T_f(\sigma)$, we have $\beta\notin H_f(\tau)$.
		\end{enumerate}
		We then have $f(\sigma)\succeq f(\tau)$, where the equality $f(\sigma)=f(\tau)$ can only hold if either $\sigma=\tau$ or $H_f(\sigma)=\{\tau\}$.
	\end{lem}

	\begin{proof}
		First, suppose $\card T_f(\sigma)=0$. In the case where $\card H_f(\sigma)=1$, we have $\tau\in\Pi_f(\sigma)=H_f(\sigma)$ by Proposition \ref{prop:GradientFlow}, hence $f(\sigma)\succeq f(\tau)$ by definition of $H_f(\sigma)$. In the case $\card H_f(\sigma)=0$ ($\sigma$ is then critical), we have $\tau\in\Pi_f(\sigma)=\Cl\sigma$ by Proposition \ref{prop:GradientFlow} and $\Cl\sigma=\Cl\sigma\backslash\Cl T_f(\sigma)$ because $T_f(\sigma)=\emptyset$. Thus, from Lemma \ref{lem:DescendingImagePi}, we see that $f(\sigma)\succeq f(\tau)$ and $f(\sigma)=f(\tau)\Rightarrow\sigma=\tau$. This proves \ref{lem:DescendingConditionsEnum1}.

		We now show that \ref{lem:DescendingConditionsEnum2} implies $f(\sigma)\succeq f(\tau)$ assuming \ref{lem:DescendingConditionsEnum1} is not true, meaning that $\card T_f(\sigma)=1$. From Proposition \ref{prop:GradientFlow}, we see that $\tau\in\Pi_f(\sigma)=\Ex\sigma\backslash T_f(\sigma)\subset\Cl\sigma\backslash T_f(\sigma)$. Since $\tau\notin\Ex T_f(\sigma)$ by hypothesis, it follows that $\tau\in\Cl\sigma\backslash\Cl T_f(\sigma)$. The result follows from Lemma \ref{lem:DescendingImagePi}.

		Finally, suppose that \ref{lem:DescendingConditionsEnum1} and \ref{lem:DescendingConditionsEnum2} are false and condition \ref{lem:DescendingConditionsEnum3} is satisfied. Let $\tau\in K_p$ and $\sigma\in K_{p+q}$. Since $\tau\in\Ex T_f(\sigma)$ by hypothesis, we can deduce that $q \geq 2$. From the definition of a simplicial complex, we could show that there exists a $\beta\in K_{p+1}$ such that $\tau\subset\beta\subset\sigma$ and $\beta\notin\Cl T_f(\sigma)$. In other words, $\beta\in\Cl\sigma\backslash\Cl T_f(\sigma)$. From Lemma \ref{lem:DescendingImagePi}, it follows that $f(\sigma)\succeq f(\beta)$. Moreover, by hypothesis \ref{lem:DescendingConditionsEnum3}, $\beta\notin H_f(\tau)$, so $f(\beta)\succneqq f(\tau)$ by definition of a \mdm function. Also, it follows from Lemma \ref{lem:DescendingImagePi} that $f(\sigma)\succeq f(\beta)$. Hence, $f(\sigma)\succeq f(\beta)\succneqq f(\tau)$.
	\end{proof}

	\begin{prop}\label{prop:AlmostDescendingPi}
		Let $f:K\rightarrow\R^\maxdim$ be \mdm and consider $\sigma,\tau\in K$ such that $\sigma\rightconnects{f}\tau$. Consider the three conditions from Lemma \ref{lem:DescendingConditions}:
		\begin{enumerate}[label=(\alph*)]
			\item\label{prop:AlmostDescendingPiEnum1} $\card T_f(\sigma)=0$;
			\item\label{prop:AlmostDescendingPiEnum2} $\tau\notin\Ex T_f(\sigma)$;
			\item\label{prop:AlmostDescendingPiEnum3} for every $\beta^{(p+1)}\supset\tau^{(p)}$ such that $\beta\in\Cl\sigma\backslash\Cl T_f(\sigma)$, we have $\beta\notin H_f(\tau)$.
		\end{enumerate}
		If at least one of these conditions is satisfied, then $f(\sigma)\succeq f(\tau)$, where the equality can only hold if either $\sigma=\tau$ or $H_f(\sigma)=\{\tau\}$.
	\end{prop}

	\begin{proof}
		First, we see that when $\tau\in\Cl\sigma\backslash\Cl T_f(\sigma)$, the result is obvious from Lemma \ref{lem:DescendingImagePi}. Now, suppose $\tau\in\Cl T_f(\sigma)$. We see that $\tau\in\Ex\sigma$. Also, $\tau\notin T_f(\sigma)$: otherwise, it would imply that $\Pi_f(\tau)=\{\sigma\}$ and we would have a cycle $\sigma\leftrightconnects{f}\tau$. Thus, $\tau\in\Ex\sigma\backslash T_f(\sigma)$ and we see from Proposition \ref{prop:GradientFlow} that $\tau\in\Pi_f(\sigma)$, so the result follows from Lemma \ref{lem:DescendingConditions}.

		All that is left to show is the case $\tau\notin\Cl\sigma$. Since $\sigma\rightconnects{f}\tau$, there exists a solution $\varrho:\Z\nrightarrow K$ with $\dom \varrho =\{0,1,...,n\}$ where $n\geq 1$ such that $\varrho(0)=\sigma$ and $\varrho(n)=\tau$. This part of the proposition is proved by induction on $n$.

		For $n=1$, we have $\tau\in\Pi_f(\sigma)$ and the result is straightforward from Lemma \ref{lem:DescendingConditions}. When $n>1$, notice that $\tau\notin\Cl\sigma$ implies $\tau\notin\Ex T_f(\sigma)$ where $\sigma=\varrho(0)$. Thus, let $i_0$ be the greatest $i=0,1,...,n-1$ such that $\tau\notin\Ex T_f(\varrho(i))$ and consider $\upsilon:=\varrho(i_0)$. We make the following statement:
		\begin{gather}
			f(\upsilon)\succeq f(\tau)\text{ where } f(\upsilon)=f(\tau) \text{ implies } \upsilon=\tau \text{ or }H_f(\upsilon)=\{\tau\}.\label{eq:proofAlmostDescendingPi1}
		\end{gather}
		Indeed:
		\begin{itemize}
			\item If $i_0=n-1$, then $\tau\in\Pi_f(\upsilon)$ and statement \ref{eq:proofAlmostDescendingPi1} follows from Lemma \ref{lem:DescendingConditions} because condition \ref{lem:DescendingConditionsEnum2} is verified by definition of $\upsilon$.

			\item If $i_0<n-1$, since $i_0$ is the greatest $i=0,1,...,n-1$ for which $\tau\notin\Ex T_f(\varrho(i))$, necessarily, $\tau\in\Ex T_f(\varrho(i_0+1))$.
			\begin{itemize}
				\item When $\card H_f(\upsilon)=1$, then $\Pi_f(\upsilon)= H_f(\upsilon) = \{\varrho(i_0+1)\}$, thus $T_f(\varrho(i_0+1)) = \{\upsilon\}$ and $\tau\in\Ex\upsilon\subset\Cl\upsilon$. Also, $\card T_f(\upsilon)$ must be zero, so $\tau\in\Cl\upsilon = \Cl\upsilon\backslash \Cl T_f(\upsilon)$.
				\item When $\card H_f(\upsilon)=0$, then $\varrho(i_0+1)\in\Pi_f(\upsilon)\subseteq\Cl\upsilon$ and, because $\tau\in\Ex T_f(\varrho(i_0+1))\subset\Cl\varrho(i_0+1)$, we deduce $\tau\in\Cl\upsilon$. Also, $\tau\notin\Ex T_f(\upsilon)$ by definition of $\upsilon$ and $\tau\notin T_f(\upsilon)$: otherwise, we would have a cycle $\upsilon\leftrightconnects{f}\tau$, which would contradict the acyclicity of $\Pi_f$. Hence, $\tau\notin\Cl T_f(\upsilon)$, so $\tau\in\Cl\upsilon\backslash\Cl T_f(\upsilon)$.
			\end{itemize}
			In both cases, statement \ref{eq:proofAlmostDescendingPi1} follows from Lemma \ref{lem:DescendingImagePi}.
		\end{itemize}

		We now show this second statement, still considering $\upsilon$ as defined above and $\tau\notin\Cl\sigma$:
		\begin{gather}
			f(\sigma)\succeq f(\upsilon)\text{ where } f(\sigma)=f(\upsilon) \text{ implies } \sigma=\upsilon \text{ or }H_f(\sigma)=\{\upsilon\}.\label{eq:proofAlmostDescendingPi2}
		\end{gather}
		This follows from the induction hypothesis, which can be used since $\sigma\rightconnects{f}\upsilon$ and at least one of the conditions \ref{prop:AlmostDescendingPiEnum1}, \ref{prop:AlmostDescendingPiEnum2} or \ref{prop:AlmostDescendingPiEnum3} is verified for $\sigma$ and $\upsilon$. Indeed, suppose all three conditions are false, meaning that $\upsilon\in\Ex T_f(\sigma)$ and there exists a $\beta\in H_f(\upsilon)$ such that $\beta\in\Cl\sigma\backslash T_f(\sigma)$. Then, we have $\Pi_f(\upsilon) = H_f(\upsilon) = \{\beta\}$, so $\varrho(i_0+1) = \beta$. Also, by definition of $i_0$, we see that $\tau\in\Ex T_f(\beta) = \Ex\upsilon\subset\Cl\upsilon$. Since $\upsilon\in\Ex T_f(\sigma)$ implies $\upsilon\in\Cl\sigma$, we find $\tau\in\Cl\sigma$, a contradiction.

		The proposition follows from statements \ref{eq:proofAlmostDescendingPi1} and \ref{eq:proofAlmostDescendingPi2}.
	\end{proof}

	In particular, since $T_f(\sigma)=H_f(\sigma)=\emptyset$ when $\sigma$ is critical, we can deduce the following.

	\begin{cor}\label{coro:ConnectionCriticalImpliesGreaterImage}
		Let $f:K\rightarrow\R^\maxdim$ be \mdm and consider $\sigma,\tau\in K$ such that $\sigma\rightconnects{f}\tau$ and $\sigma\neq \tau$. If either $\sigma$ or $\tau$ is critical, then $f(\sigma)\succneqq f(\tau)$.
	\end{cor}

	\section{Morse theorems}\label{sec:MorseTheorems}

	Here, we explain how the main theorems from Morse-Forman theory \citep[see][Section 3]{Forman1998} generalize in the multiparameter setting. A theorem extending some of those results is also proposed.

	\subsection{Classical results}

	One of the main theorems in both classical and discrete Morse theory is that a given topological space on which is defined a Morse function $f$ is always homotopy equivalent to a CW complex having $m_p(f)$ cells of dimension $p$ for each $p=0,1,2,...$, where $m_p(f)$ is the number of critical points of $f$ of index $p$. This still holds in this setting.

	\begin{prop}\label{prop:HomotopyEquivalenceCriticalComplex}
		Let $K$ be a simplicial complex. Suppose there exists a \mdm function $f:K\rightarrow\R^\maxdim$ and let $m_p(f)$ be the number of critical points of index $p$ of $f$. The complex $K$ is homotopy equivalent to a CW complex with exactly
		$m_p(f)$ cells of dimension $p$.
	\end{prop}

	\begin{proof}
		From Corollary \ref{coro:ExistenceMDMwithSameGradient}, if there exists a \mdm function $f:K\rightarrow\R^\maxdim$, we know there exists a discrete Morse function $g:K\rightarrow\R$ having the same gradient field as $f$, meaning that $m_p(f) = m_p(g)$. Since Proposition \ref{prop:HomotopyEquivalenceCriticalComplex} is well known for $k=1$ \citep[Corollary 3.5]{Forman1998}, we have the result.
	\end{proof}

	Note that, from this proposition, we could find an alternative proof to the famous Morse inequalities from Corollary \ref{coro:MorseInequalitiesMDMfunction}.

	Furthermore, Forman's results on the homotopy type of a sublevel set still hold in the multiparameter setting. For a \mdm function $f:K\rightarrow\R^\maxdim$ and a vector $a\in\R^\maxdim$, the \emph{sublevel set} $K(a)$ is the smallest subcomplex of $K$ containing all $\sigma\in K$ for which $f(\sigma)\preceq a$:
	\begin{gather*}
		K(a) := \bigcup_{\substack{\sigma\in K\\ f(\sigma)\preceq a}}\bigcup_{\alpha\subseteq\sigma}\alpha.
	\end{gather*}
	We could easily show that, if we note $\cV$ the gradient field of $f$, $K(a)$ is $\cV$-compatible for all $a\in\R^\maxdim$.
	Moreover, we see that a simplex $\sigma\in K_p$ is in $K(a)$ when either $f(\sigma)\preceq a$ or $f(\tau)\preceq a$ for some $\tau \supset \sigma$. To check if the second condition is true, it is sufficient to consider cofacets $\tau^{(p+1)}\supset\sigma$, as suggests the next lemma. It is proven similarly to its one-dimensional analogue \citep[Lemma 3.2]{Forman1998}.

	\begin{lem}\label{lem:CofacetSmallerValue}
		Let $f:K\rightarrow\R^\maxdim$ be \mdm and $\sigma\in K_p$. For all $\tau \supset \sigma$, there exists a  $\beta\in K_{p+1}$ such that $\sigma \subset \beta\subseteq\tau$ and $f(\beta)\preceq f(\tau)$.
	\end{lem}

	To state Propositions \ref{prop:RegularSimplicesCollapse} and \ref{prop:CriticalSimplexHomotopyChange}, which are analogous to Theorems 3.3 and 3.4 in \citep{Forman1998}, we consider the following subset of $\R^\maxdim$:

	\begin{gather*}
		Q^b_a := \left\lbrace c=(c_1,...,c_\maxdim)\in\R^\maxdim\ \big|\ c\preceq b\text{ and } c_i>a_i\text{ for some }i=1,...,\maxdim\right\rbrace.
	\end{gather*}

	\begin{prop}\label{prop:RegularSimplicesCollapse}
		Let $f:K\rightarrow\R^\maxdim$ be \mdm and consider $a\precneqq b\in\R^\maxdim$. If there is no critical simplex $\sigma\in K$ such that $f(\sigma)\in Q_a^b$, then
		\begin{gather*}
			K(b)\searrow K(a).
		\end{gather*}
	\end{prop}

	\begin{prop}\label{prop:CriticalSimplexHomotopyChange}
		Let $f:K\rightarrow\R^\maxdim$ be \mdm and consider $a\precneqq b\in\R^\maxdim$. Suppose $\sigma\in K_p$ is the only critical simplex of $f$ with $f(\sigma)\in Q_a^b$, then
		\begin{gather*}
			K(b)\simeq K(a)\bigcup_{\sbdy e^p}e^p
		\end{gather*}
		where $e^p$ is a cell of dimension $p$.
	\end{prop}

	The proofs of these two propositions are omitted: Proposition \ref{prop:RegularSimplicesCollapse} is a direct consequence of Lemma \ref{lem:AcyclicNoCriticalCollapse} whereas Proposition \ref{prop:CriticalSimplexHomotopyChange} is a particular case of Theorem \ref{theo:ExtendedMorseTheorem}. Both of these results will be proved hereafter.

	\subsection{Extended Morse theorem}

	In the one-dimensional setting, the last two propositions suffice to describe all changes in homotopy type of the sublevel set $K(a)$ as $a$ increases. Indeed, for a discrete Morse function $f:K\rightarrow\R$, consider a critical simplex $\sigma$. From Corollary \ref{coro:ConnectionCriticalImpliesGreaterImage}, we see that for any critical simplex $\alpha\subsetneq\sigma$, we have $f(\alpha) < f(\sigma)$ and conversely, for any critical simplex $\beta\supsetneq\sigma$, we have $f(\sigma) < f(\beta)$. Hence, we can always choose parameters $a$ large enough and $b$ small enough so that, at least locally, $\sigma$ is the unique critical simplex with $f(\sigma)\in Q_a^b$.

	When $f:K\rightarrow\R^\maxdim$ is \mdm with $\maxdim>1$, we can see that this is not necessarily true by considering the following very simple example. Let $f:K\rightarrow\R^2$ be as defined in Figure \ref{fig:NecessityExtendedMorseTheorem} and consider the critical edge $\sigma$ with $f(\sigma)=(1,1)$. If we choose $a\in\R^2$ so that $f(\alpha)\preceq a$ for each (critical) vertex $\alpha\subset\sigma$, we then have $(1,1)\preceq a$, so $f(\sigma)\notin Q_a^b$ for all $b\in\R^2$. Thus, in order to have $f(\sigma)\in Q_a^b$, we must have at least one vertex $\alpha$ with $f(\alpha)\in Q_a^b$ as well. Consequently, the hypotheses of Proposition \ref{prop:CriticalSimplexHomotopyChange} may not be verified for $\sigma$.

	\begin{figure}[h]
		\centering
		\begin{tikzpicture}[scale=3]
			\draw[very thick] (0,0)node{$\bullet$} -- (1,0)node{$\bullet$};
			\draw (0,0)node[below]{$(0,1)$} (0.5,0)node[above]{$(1,1)$} (1,0)node[below]{$(1,0)$};
		\end{tikzpicture}
		\caption{A \mdm function $f:K\rightarrow\R^2$ for which there exists no $a,b\in\R^2$ such that $\sigma := f^{-1}(1,1)$ is the unique critical simplex with $f(\sigma)\in Q_a^b$}\label{fig:NecessityExtendedMorseTheorem}
	\end{figure}
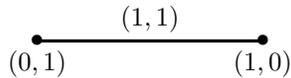

	To overcome this gap, we come up with an additional result, Theorem \ref{theo:ExtendedMorseTheorem}, for which no hypothesis on the number of critical simplices $\sigma$ with $f(\sigma)\in Q_a^b$ is needed. To prove it, preliminary results are necessary.

	\begin{lem}\label{lem:SubsetNotConnectSubcomplex}
		Let $\Pi_\cV:K\multimap K$ be a flow. A $\cV$-compatible subset $L\subseteq K$ is a subcomplex of $K$ if and only if $\Pi_\cV(L)\subseteq L$.
	\end{lem}

	\begin{proof}
		Assume $L$ is a subcomplex of $K$, meaning that $\Cl\sigma\subseteq L$ for all $\sigma\in L$. For all $\sigma\in L$, by definition of $\Pi_\cV$, we have either $\Pi_\cV(\sigma)\subseteq\Cl\sigma$, which is a subset of $L$ because it is a subcomplex, or $\Pi_\cV(\sigma)=\{\cV(\sigma)\}\subset L$ because $L$ is $\cV$-compatible.

		Now, suppose $\Pi_\cV(L)\subseteq L$. Let $\sigma\in L$ and consider $\alpha\in\Cl\sigma$. If $\sigma\rightconnects{\cV}\alpha$, it is easy to verify that $\alpha\in L$. Otherwise, if $\sigma\nrightconnects{\cV}\alpha$, we have by Lemma \ref{lem:IfFaceConnect} that $\sigma\in\image\cV\backslash\Fix\cV$ and $\alpha=\cV^{-1}(\sigma) = \sigma^-$. By the $\cV$-compatibility of $L$, it follows that $\alpha\in L$, so $L$ is a subcomplex of $K$.
	\end{proof}

	\begin{lem}\label{lem:AcyclicNoCriticalCollapse}
		Let $\Pi_\cV:K\multimap K$ be an acyclic flow and consider a $\cV$-compatible subcomplex $L\subseteq K$. If $\Fix\cV\subseteq L$, then $K\searrow L$.
	\end{lem}

	\begin{proof}
		If $K\backslash L$ contains no fixed point of $\cV$, then $\card K\backslash L=2n$ for some $n\geq 0$. We prove the result by induction on $n$. If $n=0$, it is obvious since $L=K$.

		When $n\geq 1$, we first show that there exists a $\sigma\in K\backslash L$ such that $\Pi^{-1}_\cV(\sigma)=\emptyset$. Indeed, consider the preorder $\leq_\cV$ induced by the relation $\leftconnects{\cV}$ on $K$, i.e. $\sigma\leq_\cV\tau$ iff $\sigma=\tau$ or $\sigma\leftconnects{\cV}\tau$. By the acyclicity of $\Pi_\cV$, we see that $\leq_\cV$ is actually a partial order on $K$. Hence, because $K\backslash L$ is finite, we can choose $\sigma$ maximal in $K\backslash L$ relatively to $\leq_\cV$. In other words, we can choose $\sigma\in K\backslash L$ so that $\sigma\nleftconnects{\cV}\tau$ for all $\tau\in K\backslash L$ such that $\tau\neq\sigma$. Moreover, since $\Pi_\cV(L)\subseteq L$ by Lemma \ref{lem:SubsetNotConnectSubcomplex}, we see that $\tau\nrightconnects{\cV}\sigma$ also holds for $\tau\in L$, so $\tau\nrightconnects{\cV}\sigma$ for all $\tau\in K\backslash\{\sigma\}$. Finally, since $\sigma\notin\Fix\cV\subseteq L$ by hypothesis, we have $\sigma\nrightconnects{\cV}\sigma$, thus $\Pi_\cV^{-1}(\sigma)=\emptyset$.

		We now assume that $\sigma\in K\backslash L$ is such that $\Pi_\cV^{-1}(\sigma)=\emptyset$. We can easily check that $\sigma\in\dom\cV\backslash\Fix\cV$: if $\sigma$ were in $\image\cV$, we would have $\cV^{-1}(\sigma)\in\Pi^{-1}_\cV(\sigma)$, a contradiction. Also, $\sigma$ is a free face of $\cV(\sigma)\in K\backslash L$. Indeed, suppose $\sigma\subset\beta$ for some $\beta\in K$. We know that $\beta\nrightconnects{\cV}\sigma$ because $\Pi_\cV^{-1}(\sigma)=\emptyset$. Hence, from Lemma \ref{lem:IfFaceConnect}, we necessarily have $\cV(\sigma) = \beta$, so $\sigma$ is a free face of $\cV(\sigma)$, which is in $K\backslash L$ by the $\cV$-compatibility of $L$.

		Finally, we see that $K\searrow K\backslash\{\sigma,\cV(\sigma)\}$ and, using the induction hypothesis, we then have that $K\backslash\{\sigma,\cV(\sigma)\}\searrow L$, which concludes the proof.
	\end{proof}

	Before we present the main theorem of this section, recall from Proposition \ref{prop:MinimalMorseDecompositionMDM} that
	\begin{gather*}
		\cM = \big\lbrace\,\{\tau\}\subseteq K\ |\ \tau\text{ is critical for } f \big\rbrace
	\end{gather*}
	is a Morse decomposition of $\Pi_f$, so by considering a set $I$ of critical simplices of $f$, we can define the Morse set
	\begin{gather*}
		M(I) = \bigcup_{\tau,\tau'\in I}C(\tau',\tau),
	\end{gather*}
	for which $\Con(M(I))$ is defined from Proposition \ref{prop:MorseSetIsolatedInvariant}, where
	\begin{gather*}
		C(\tau',\tau) = \left\lbrace \sigma\in K\ \big|\ \tau'\rightconnects{f}\sigma\rightconnects{f}\tau\right\rbrace.
	\end{gather*}

	\begin{thm}\label{theo:ExtendedMorseTheorem}
		Let $f=(f_1,...,f_\maxdim):K\rightarrow\R^\maxdim$ be \mdm and consider $a,b\in\R^\maxdim$ such that $a\precneqq b$. Let
		\begin{gather*}
			I := \left\lbrace\sigma\in K\text{ critical}\ |\ f(\sigma)\in Q_a^b\right\rbrace.
		\end{gather*}
		Then,
		\begin{gather*}
			K(b) \simeq K(a)\bigcup_{\exit M(I)} M(I).
		\end{gather*}
	\end{thm}

	\begin{proof}
		Here is the idea of the proof. Consider the following subsets of $K(b)$:
		\begin{align*}
			A &= \left\lbrace \sigma\in K(b)\backslash K(a)\ |\ \sigma\nrightconnects{f} I\right\rbrace,\\
			B &= \left\lbrace \sigma\in K(b)\backslash M(I)\ |\ \sigma\rightconnects{f} I\right\rbrace.
		\end{align*}
		We will show that $A$, $M(I)$ and $B$ are mutually disjoint and, using Lemma \ref{lem:SubsetNotConnectSubcomplex}, that they form the following nested sequence of subcomplexes of $K(b)$:
		\begin{align*}
			K(b)&= K(a)\cup A\cup M(I) \cup B,\\
			&\supseteq K(a)\cup A\cup M(I),\\
			&\supseteq K(a)\cup A,\\
			&\supseteq K(a).
		\end{align*}
		Finally, after proving that $A$ and $B$ do not contain any critical points, it will follow from Lemma \ref{lem:AcyclicNoCriticalCollapse} that $K(b)\searrow K(a)\cup A\cup M(I)$ and $K(a)\cup A\searrow K(a)$.

		First, recall that $K(a)$, $K(b)$ and $M(I)$ are all $\cV$-compatible sets, where $\cV$ is the gradient field of $f$. Also, we could easily show that $\sigma^+\rightconnects{f}I\Leftrightarrow\sigma^-\rightconnects{f}I$. Using this argument, it would be straightforward to verify the $\cV$-compatibility of $A$ and $B$. Moreover, from the definitions of $A$, $B$ and $M(I)$, it is obvious that they are mutually disjoint sets.

		Now, we show that $K(b)\backslash K(a)= A\cup M(I)\cup B$. Verifying the inclusion from left to right is straightforward: if $\sigma\in K(b)\backslash K(a)$ is not in $M(I)$, then $\sigma\in A$ if $\sigma\nrightconnects{f}I$ and $\sigma\in B$ if $\sigma\rightconnects{f}I$. We now prove $A\cup M(I)\cup B \subseteq K(b)\backslash K(a)$.
		\begin{itemize}
			\item By definition of $A$, we have $A\subseteq K(b)\backslash K(a)$.

			\item If $\sigma\in M(I)$, by definition of a Morse set, there exists critical simplices $\tau,\tau'\in I$ such that $\tau'\rightconnects{f}\sigma\rightconnects{f}\tau$. By definition of $Q_a^b$, we thus have $f(\tau')\preceq b$ and $f_i(\tau)>a_i$ for some $i=1,...,\maxdim$. By Corollary \ref{coro:ConnectionCriticalImpliesGreaterImage}, it follows that $b\succeq f(\tau')\succeq f(\sigma)$, so $\sigma\in K(b)$. Moreover, we also have by Corollary \ref{coro:ConnectionCriticalImpliesGreaterImage} that $f(\sigma)\succeq f(\tau)$, so $f_i(\sigma)\geq f_i(\tau) > a_i$ for some $i=1,...,\maxdim$. In addition, for all coface $\beta\supset\sigma$, we know from Lemma \ref{lem:IfConnectsCoFaceConnects}\ref{lem:IfConnectsCoFaceConnectsEnum1} that $\beta\rightconnects{f}\tau$, so $f_i(\beta)>a_i$ by a similar reasoning. Hence, $\sigma\notin K(a)$.

			\item If $\sigma\in B$, by definition of $B$, we have $\sigma\in K(b)$. Also, there exists a critical $\tau\in I$ such that $\sigma\rightconnects{f}\tau$. Hence, from Corollary \ref{coro:ConnectionCriticalImpliesGreaterImage}, we see that $f(\sigma)\succneqq f(\tau)$, thus $f_i(\sigma)\geq f_i(\tau)>a_i$ for some $i=1,...,\maxdim$ because $f(\tau)\in Q_a^b$. From Lemma \ref{lem:IfConnectsCoFaceConnects}\ref{lem:IfConnectsCoFaceConnectsEnum1}, we also have $\beta\rightconnects{f}\tau$ for all coface $\beta\supset\sigma$, hence $f_i(\beta)>a_i$. We conclude that $\sigma\notin K(a)$.
		\end{itemize}
		Consequently, $K(b)$ is the union $K(a)\cup A\cup M(I)\cup B$ of mutually disjoint $\cV$-compatible sets.

		Next, we prove that $K(a)\cup A$ and $K(a)\cup A\cup M(I)$ are subcomplexes of $K(b)$ using Lemma \ref{lem:SubsetNotConnectSubcomplex}. To do so, we have to show that $\Pi_f(K(a)\cup A)\subseteq K(a)\cup A$ and $\Pi_f(K(a)\cup A\cup M(I))\subseteq K(a)\cup A\cup M(I)$. Since $K(a)$ is a subcomplex of $K(b)$, we already know that $\Pi_f(K(a))\subseteq K(a)$ by Lemma \ref{lem:SubsetNotConnectSubcomplex}. Thus, it suffices to show that $\Pi_f(A)\subseteq K(a)\cup A$ and $\Pi_f(M(I))\subseteq K(a)\cup A\cup M(I)$.
		\begin{itemize}
			\item Let $\sigma\in A$ and $\gamma\in\Pi_f(\sigma)$ and suppose $\gamma\in M(I)\cup B$. By definition of $M(I)$ and $B$, we have $\gamma\rightconnects{f}\tau$ for some $\tau\in I$, hence $\sigma\rightconnects{f}\gamma\rightconnects{f}\tau$, which is a contradiction since $\sigma\nrightconnects{f}I$ by definition of $A$. Thus, $\gamma\notin M(I)\cup B$, meaning that $\gamma\in K(a)\cup A$.

			\item Let $\sigma\in M(I)$ and $\gamma\in\Pi_f(\sigma)$. By definition of $M(I)$, there exists a $\tau\in I$ such that $\tau\rightconnects{f}\sigma\rightconnects{f}\gamma$. If there exists a $\tau'\in I$ such that $\gamma\rightconnects{f}\tau'$, then $\gamma\in M(I)$. Otherwise, if $\gamma\nrightconnects{f}I$, then $\gamma\notin B$, so $\gamma\in K(a)\cup A\cup M(I)$. Either way, we have the desired result.
		\end{itemize}

		We now show that all critical points of $f$ in $K(b)\backslash K(a)$ are necessarily in $I$. To do so, we verify that $f(\sigma)\in Q_a^b$ for all critical $\sigma\in K(b)\backslash K(a)$. Since $\sigma\notin K(a)$, we obviously have $f_i(\sigma)>a_i$ for some $i=1,...,\maxdim$. Also, $\sigma\in K(b)$ implies that either $f(\sigma)\preceq b$ or $f(\gamma)\preceq b$ for some coface $\gamma\supset\sigma$. When $f(\sigma)\preceq b$, we immediately have $f(\sigma)\in Q_a^b$. When $f(\gamma)\preceq b$ for some coface $\gamma\supset\sigma$, we see from Lemma \ref{lem:CofacetSmallerValue} that there exists a cofacet $\beta$ of $\sigma$ such that $f(\beta)\preceq f(\gamma)\preceq b$. It follows from the definition of a critical point that $f(\sigma)\precneqq f(\beta)\preceq b$, so $f(\sigma)\in Q_a^b$.

		We deduce that the sets $A$ and $B$ do not contain any critical points of $f$. Also, recall from Proposition \ref{prop:AcyclicityGradientFlow} that a gradient flow is always acyclic. Hence, from Lemma \ref{lem:AcyclicNoCriticalCollapse}, it follows that $K(b)= K(a)\cup A\cup M(I)\cup B\searrow K(a)\cup A\cup M(I)$ and $K(a)\cup A\searrow K(a)$. By endowing $K$ with the topology of a CW-complex, we conclude that
		\begin{gather*}
			K(b)\simeq K(a)\cup A\cup M(I) \simeq K(a)\cup M(I)
		\end{gather*}
		where $M(I)$ is attached to $K(a)$ along its exit set $\exit M(I)$.
	\end{proof}

	We want to emphasize the fact that Theorem \ref{theo:ExtendedMorseTheorem} is a generalization of both Proposition \ref{prop:RegularSimplicesCollapse} and Proposition \ref{prop:CriticalSimplexHomotopyChange}. Indeed, when $I=\emptyset$, then $M(I)=\emptyset$ and we see from the proof that $K(b)\searrow K(a)$, which proves Proposition \ref{prop:RegularSimplicesCollapse}. Then, when $I=\{\sigma^{(p)}\}$, we have $M(I)=\{\sigma^{(p)}\}$, which is homeomorphic to a cell of dimension $p$, and $\exit M(I) = \cl\sigma\backslash\sigma=\sbdy\sigma$, so Proposition \ref{prop:CriticalSimplexHomotopyChange} follows.

	\subsection{Morse set examples}

	We conclude this section with a few examples to illustrate Theorem \ref{theo:ExtendedMorseTheorem}.

	\begin{ex}\label{ex:MorseSet1}
		Let $K$ and $f:K\rightarrow\R^2$ be as defined in Figure \ref{fig:MorseSetEx1}. We could show that $f$ is \mdm and that the red simplices are critical.

		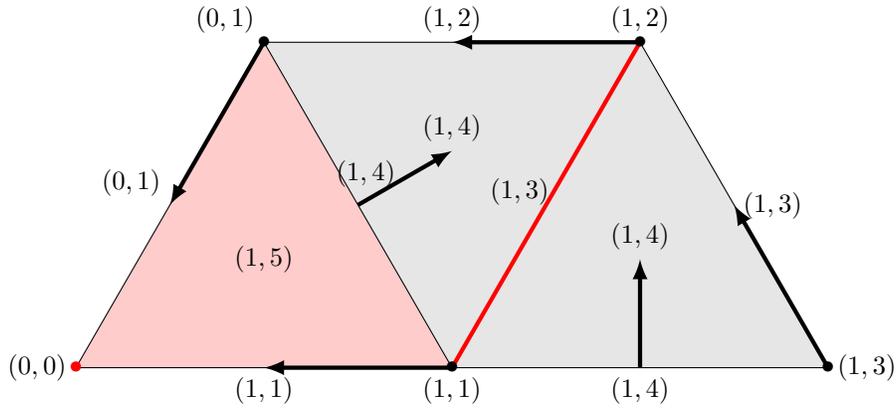
\begin{figure}[h]
			\centering
			\begin{tikzpicture}[scale=5]
				\fill[red!20] (0,0) -- (1,0) -- (0.5,0.866) -- cycle;
				\fill[black!10] (1,0) -- (2,0) -- (1.5,0.866) -- (0.5,0.866) -- cycle;
				\draw[red,ultra thick] (1,0) -- (1.5,0.866);
				\draw (0.5,0.866) -- (0,0) -- (1,0)node{$\bullet$} -- (2,0)node{$\bullet$} -- (1.5,0.866)node{$\bullet$} -- (0.5,0.866)node{$\bullet$} -- (1,0);
				\draw[red] (0,0)node{$\bullet$};
				\draw[ultra thick, -latex] (1,0) -- (0.5,0);
				\draw[ultra thick, -latex] (0.5,0.866) -- (0.25,0.433);
				\draw[ultra thick, -latex] (0.75,0.433) -- (1,0.577);
				\draw[ultra thick, -latex] (1.5,0.866) -- (1,0.866);
				\draw[ultra thick, -latex] (1.5,0) -- (1.5,0.289);
				\draw[ultra thick, -latex] (2,0) -- (1.75,0.433);
				\draw (0,0)node[left]{$(0,0)$}
				(0.25,0.433)node[above left]{$(0,1)$}
				(0.5,0.866)node[above left]{$(0,1)$}
				(0.5,0)node[below]{$(1,1)$}
				(0.5,0.289)node{$(1,5)$}
				(0.75,0.433)node[xshift=3pt, yshift=12pt]{$(1,4)$}
				(1,0)node[below]{$(1,1)$}
				(1,0.577)node[above]{$(1,4)$}
				(1,0.866)node[above]{$(1,2)$}
				(1.25,0.433)node[xshift=-10pt, yshift=5pt]{$(1,3)$}
				(1.5,0)node[below]{$(1,4)$}
				(1.5,0.289)node[above]{$(1,4)$}
				(1.5,0.866)node[above]{$(1,2)$}
				(1.75,0.433)node[right]{$(1,3)$}
				(2,0)node[right]{$(1,3)$};
			\end{tikzpicture}
			\caption{A \mdm function, with critical simplices in red}\label{fig:MorseSetEx1}
		\end{figure}

		For $a = (0,5)$ and $b= (1,5)$, we see in Figure \ref{fig:MorseSetEx1Pieces}, as defined in the proof of Theorem \ref{theo:ExtendedMorseTheorem}, the sets $K(a)$, $A$, $M(I)$ and $B$. Here, notice that $\Con(M(I))$ is trivial and $K(b)\searrow K(a)$.

		\begin{figure}[h]
			\centering
			\begin{tikzpicture}[scale=5]
				\fill[green!20] (1,0) -- (1.5,0.866) -- (2,0) -- cycle;
				\draw[ultra thick, green] (1,0) -- (2,0);
				\fill[orange!20] (0,0) -- (1,0) -- (1.5,0.866) -- (0.5,0.866) -- cycle;
				\draw[ultra thick, orange] (0.5,0.866) -- (1,0) -- (1.5,0.866);
				\draw[ultra thick, blue] (0,0) -- (1,0)node{$\bullet$}
				(0.5,0.866) -- (1.5,0.866)node{$\bullet$} -- (2,0)node{$\bullet$};
				\draw[ultra thick] (0,0)node{$\bullet$} -- (0.5,0.866)node{$\bullet$};
			\end{tikzpicture}
			\caption{The sets $K(a)$ (black), $A$ (blue), $M(I)$ (orange) and $B$ (green), as defined in the proof of Theorem \ref{theo:ExtendedMorseTheorem}, of the \mdm function in Figure \ref{fig:MorseSetEx1} for $a = (0,5)$ and $b= (1,5)$}\label{fig:MorseSetEx1Pieces}
		\end{figure}
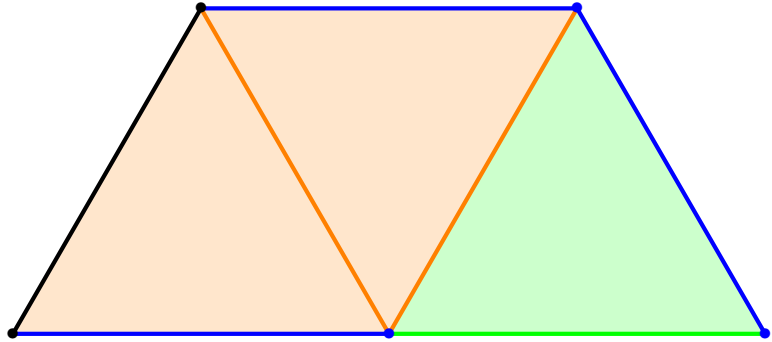
	\end{ex}

	\begin{ex}
		Consider the ($1$-dimensional) simplicial complex $K$ and the \mdm function $f:K\rightarrow\R^2$ defined in Figure \ref{fig:MorseSetEx2a}.

		\begin{figure}[h]
			\centering
			\begin{subfigure}{0.5\linewidth}
				\centering
				\begin{tikzpicture}[scale=1.5]
					\draw[ultra thick] (-1,0)node{$\bullet$} -- (0,-1)node{$\bullet$} -- (1,1)node{$\bullet$} -- cycle;
					\draw (-1,0)node[left]{$(-1,0)$}
					(0,-1)node[below]{$(0,-1)$}
					(1,1)node[above right]{$(1,1)$}
					(-0.5,-0.5)node[below left]{$(0,0)$}
					(0,0.5)node[above left]{$(1,2)$}
					(0.5,0)node[below right]{$(2,1)$};
				\end{tikzpicture}
				\caption{}\label{fig:MorseSetEx2a}
			\end{subfigure}

			\begin{subfigure}{0.3\linewidth}
				\centering
				\begin{tikzpicture}[scale=1.5]
					\draw[dashed] (-1,0)node{$\bullet$} -- (0,-1)node{$\bullet$} -- (1,1)node{$\bullet$} -- cycle;
					\draw[ultra thick, orange] (0,-1) -- (1,1)node{$\bullet$};
					\draw[ultra thick] (-1,0)node{$\bullet$} -- (0,-1)node{$\bullet$};
				\end{tikzpicture}
				\caption{}\label{fig:MorseSetEx2b}
			\end{subfigure}
			\begin{subfigure}{0.3\linewidth}
				\centering
				\begin{tikzpicture}[scale=1.5]
					\draw[orange ,ultra thick] (1,1) -- (-1,0);
					\draw[ultra thick] (-1,0)node{$\bullet$} -- (0,-1)node{$\bullet$} -- (1,1)node{$\bullet$};
				\end{tikzpicture}
				\caption{}\label{fig:MorseSetEx2c}
			\end{subfigure}
			\caption{In \subref{fig:MorseSetEx2a}, a \mdm function having only critical simplices. In \subref{fig:MorseSetEx2b} and \subref{fig:MorseSetEx2c}, $K(a)$ (black) and $M(I)$ (orange) are represented, in \subref{fig:MorseSetEx2b}, for parameters $a = (2,0)$ and $b=(2,1)$ as well as, in \subref{fig:MorseSetEx2c}, for $a = (2,1)$ and $b=(2,2)$}\label{fig:MorseSetEx2}
		\end{figure}
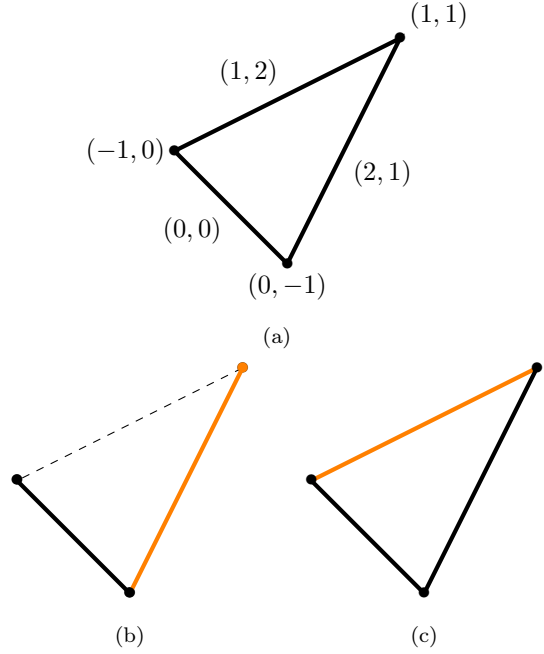

		\begin{itemize}
			\item If $a = (2,0)$ and $b=(2,1)$, we have $K(b) = K(a)\cup M(I)$ as in Figure \ref{fig:MorseSetEx2b}. More precisely, $I = M(I) = \{f^{-1}(1,1), f^{-1}(2,1)\}$. We could show that $\Con(M(I))$ is trivial and, although $K(b)$ is obtained by adding two critical simplices to $K(a)$, we have $K(b)\searrow K(a)$.

			\pagebreak

			\item If $a = (2,1)$ and $b=(2,2)$, we have $K(b) = K = K(a)\cup M(I)$ as in Figure \ref{fig:MorseSetEx2c}. Here, $I = M(I) = \{f^{-1}(1,2)\}$. Since a unique critical simplex of dimension $1$ is added to $K(a)$ to obtain $K(b)$, we see from either Proposition \ref{prop:CriticalSimplexHomotopyChange} or Theorem \ref{theo:ExtendedMorseTheorem} that $K(b)\simeq K(a)\cup e^1$, where $e^1$ is a cell of dimension $1$.
		\end{itemize}
	\end{ex}

	\begin{ex}
		Finally, consider $K$ and $f:K\rightarrow\R^2$ in Figure \ref{fig:MorseSetEx3a}, where the red simplices are critical.

		\begin{figure}[h]
			\centering
			\begin{subfigure}{0.5\linewidth}
				\centering
				\begin{tikzpicture}[scale=1.5]
					\draw[thick] (1,2) -- (0,1);
					\draw[red,ultra thick] (1,2) -- (2,0)node{$\bullet$} -- (0,1)node{$\bullet$};
					\draw[ultra thick, -latex] (1,2)node{$\bullet$} -- (0.5,1.5);
					\draw (0,1)node[left]{$(0,1)$}
					(1,2)node[above]{$(1,2)$}
					(2,0)node[right]{$(2,0)$}
					(0.5,1.5)node[above left]{$(1,2)$}
					(1.5,1)node[above right]{$(2,2)$}
					(1,0.5)node[below left]{$(2,1)$};
				\end{tikzpicture}
				\caption{}\label{fig:MorseSetEx3a}
			\end{subfigure}

			\begin{subfigure}{0.3\linewidth}
				\centering
				\begin{tikzpicture}[scale=1.5]
					\draw[orange, ultra thick] (1,2) -- 	(2,0);
					\draw[blue, ultra thick] (1,2)node{$\bullet$} -- (0,1);
					\draw[ultra thick] (2,0)node{$\bullet$} -- (0,1)node{$\bullet$};
				\end{tikzpicture}
				\caption{}\label{fig:MorseSetEx3b}
			\end{subfigure}
			\begin{subfigure}{0.3\linewidth}
				\centering
				\begin{tikzpicture}[scale=1.5]
					\draw[orange, ultra thick] (1,2) -- 	(2,0)node{$\bullet$} -- (0,1);
					\draw[ultra thick] (1,2)node{$\bullet$} -- (0,1)node{$\bullet$};
				\end{tikzpicture}
				\caption{}\label{fig:MorseSetEx3c}
			\end{subfigure}
			\caption{In \subref{fig:MorseSetEx3a}, a \mdm function, with critical simplices in red. In \subref{fig:MorseSetEx3b} and \subref{fig:MorseSetEx3c}, the sets $K(a)$ (black), $A$ (blue) and $M(I)$ (orange) as defined in the proof of Theorem \ref{theo:ExtendedMorseTheorem} are represented, in \subref{fig:MorseSetEx3b}, for parameters $a=(2,1)$ and $b=(2,2)$ and, in \subref{fig:MorseSetEx3c}, for $a=(1,2)$ and $b=(2,2)$}\label{fig:MorseSetEx3}
		\end{figure}
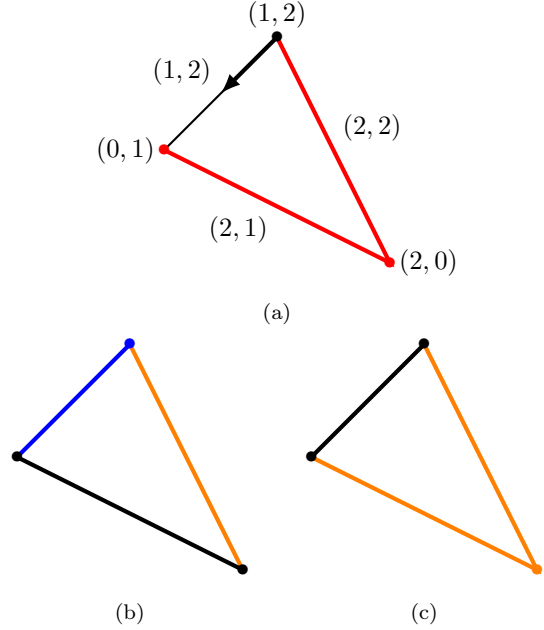

		\begin{itemize}
			\item Let $a=(2,1)$ and $b=(2,2)$. We have $K(b) = K(a)\cup A\cup M(I)$ as in Figure \ref{fig:MorseSetEx3b} where $M(I)$ contains only one critical simplex.

			\item Now, consider $a=(1,2)$ and $b=(2,2)$. We then have $K(b)=K(a)\cup M(I)$ as in Figure \ref{fig:MorseSetEx3c}. Here, we see that $M(I)$ contains three critical simplices.
		\end{itemize}
		In both examples from Figures \ref{fig:MorseSetEx3b} and \ref{fig:MorseSetEx3c}, even though the sets $M(I)$ do not have the same number of critical simplices, we see that the homology change from $K(a)$ to $K(b)$ is the same because $K(b)\simeq K(a)\cup e^1$ in both cases.
	\end{ex}

	From the previous examples, we see that critical simplices can, more or less, interact with each other. In particular, notice that when $I=\{\sigma^{(p)},\tau^{(p+1)}\}$ and there exists a unique path in the gradient field going from $\tau$ to $\sigma$, then $\Con(M(I))$ is trivial and $K(b)\searrow K(a)$. This is a consequence of a well-known result of \citet{Forman1998} which states that, under these conditions, we can find another acyclic field in which $\sigma$ and $\tau$ are not critical by reversing the arrows along the unique path going from $\tau$ to $\sigma$. Informally, we could say that the two critical simplices cancel each other out.

	\section{Critical components of a multidimensional discrete Morse function}\label{sec:CriticalComponents}

	In the original Morse theory, which studies smooth real-valued Morse functions defined on smooth manifolds, it is well known that critical points are isolated. However, when extending the theory either by considering multiple Morse functions \citep{Edelsbrunner2004, Smale1975, Wan1975} or a vector-valued function \citep{Budney2021}, we see that the set of critical points obtained is separated in connected components. Experimentally, it seems that critical points of a \mdm function form clusters much like the critical components that appear in the smooth setting \citep{Allili2019}.

	In this section, we propose a way to partition the critical simplices of a \mdm function as introduced in Section~\ref{sec:MDM} in order to define such critical components in the discrete setting. Moreover, we establish acyclicity conditions under which this partition induces a Morse decomposition introduced in Section~\ref{sec:CombinatorialDynamics} and studied further in Section~\ref{sec:flow-MDM}. That enables us to extend results on the Poincar\'e polynomial and Morse inequalities to our acyclic \mdm function setting in Theorem~\ref{theo:MorseInequalitiesCritComponents}. This result complements the theoretical goals of \citet{Allili2019} which provided the main motivation for our work.

	\subsection{Partitioning the critical points}

	Let $f:K\rightarrow\R^\maxdim$ be \mdm and consider $\cC$, the set of critical points of $f$. Since, in the smooth setting, the critical components of either multiple Morse functions or a vector-valued function are the connected components formed by the critical points, we could simply define the critical components of $f$ as the connected components of $\cC$. However, it seems that the topological connectedness is neither a necessary nor a sufficient criteria to define the critical components of $f$. Indeed, we can see in \citep{Allili2019} that, in practice, the clusters formed by critical simplices are not necessarily connected. Also, Example \ref{ex:ParetoCircle} shows that $\cC$ may be connected even though we could expect $f$ to have multiple critical components.

	\begin{ex}\label{ex:ParetoCircle}
		Consider the inclusion map $\iota:S^1\hookrightarrow\R^2$ and the \mdm function $f:K\rightarrow\R^2$ defined as in Figure \ref{fig:ParetoCircleDiscrete}. We see that $f$ is a discretization of $\iota$, in the sense that every vertex $v\in K$ is a point of $S^1$ and $f(v)=\iota(v)$. Moreover, we could show that the set of points in $S^1$ which are critical for $\iota$ is the one in red in Figure \ref{fig:ParetoCircleSmooth}. This is the \emph{Pareto set} of $\iota$, which we will discuss below. Here, the Pareto set clearly has two connected components. However, the critical points of $f$ are all connected.
		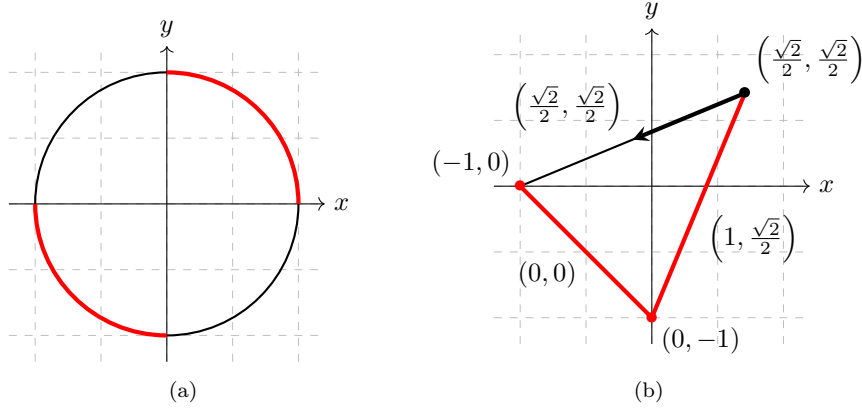
\begin{figure}[h]
			\centering
			\begin{subfigure}{0.48\textwidth}
				\centering
				\begin{tikzpicture}[scale=1.75]
					\draw[help lines, color=gray!50, dashed, step=0.5] (-1.2,-1.2) grid (1.2,1.2);
					\draw[->] (-1.2,0)--(1.2,0) node[right]{$x$};
					\draw[->] (0,-1.2)--(0,1.2) node[above]{$y$};
					\draw[thick] (0,0) circle (1);
					\draw[ultra thick, red] (180:1) arc (180:270:1);
					\draw[ultra thick, red] (0:1) arc (0:90:1);
				\end{tikzpicture}
				\caption{}\label{fig:ParetoCircleSmooth}
			\end{subfigure}
			\begin{subfigure}{0.48\textwidth}
				\centering
				\begin{tikzpicture}[scale=1.75]
					\draw[help lines, color=gray!50, dashed, step=0.5] (-1.2,-1.2) grid (1.2,1.2);
					\draw[->] (-1.2,0)--(1.2,0) node[right]{$x$};
					\draw[->] (0,-1.2)--(0,1.2) node[above]{$y$};
					\draw[ultra thick,red] (0,-1) -- ({cos(45)},{sin(45)});
					\draw[thick] ({cos(45)},{sin(45)})node{$\bullet$} -- (-1,0);
					\draw[ultra thick,red] (-1,0)node{$\bullet$} -- (0,-1)node{$\bullet$};
					\draw[ultra thick, -stealth] (45:1)node{$\bullet$} -- ({(cos(45)+cos(180))/2},{(sin(45)+sin(180))/2});
					\draw (-1,0)node[above left]{$(-1,0)$}
					(0,-1)node[below right]{$(0,-1)$}
					({cos(45)},{sin(45)})node[above right]{$\left(\frac{\sqrt{2}}{2},\frac{\sqrt{2}}{2}\right)$}
					(-0.5,-0.5)node[below left]{$(0,0)$}
					({(cos(45)+cos(270))/2},{(sin(45)+sin(270))/2})node[below right]{$\left(1,\frac{\sqrt{2}}{2}\right)$}
					({(cos(45)+cos(180))/2},{(sin(45)+sin(180))/2})node[above left]{$\left(\frac{\sqrt{2}}{2},\frac{\sqrt{2}}{2}\right)$};
				\end{tikzpicture}
				\caption{}
				\label{fig:ParetoCircleDiscrete}
			\end{subfigure}
			\caption{In \subref{fig:ParetoCircleSmooth}, the critical points of $\iota:S^1\rightarrow\R^2$ are shown in red. In \subref{fig:ParetoCircleDiscrete}, a \mdm function discretizing $\iota$}\label{fig:ParetoCircle}
		\end{figure}
	\end{ex}

	We consider another approach to partition $\cC$. We still want the critical simplices in a same component to be connected in some way, but it should not necessarily be topologically. Hence, we consider the dynamical connections between the critical simplices, meaning that we consider $\sigma,\tau\in\cC$ to be neighbors if either $\sigma\rightconnects{f}\tau$ or $\sigma\leftconnects{f}\tau$.

	Moreover, in \citep{Allili2019}, where \mdm functions were first introduced, an algorithm to generate gradient fields was developed and it was noticed that the sets of critical simplices found computationally resembled Pareto sets, which were studied in the setting of smooth Morse theory \citep{Smale1975,Wan1975}. Since the concepts of Pareto sets and optima are widely used in the literature, many different variants exist: here, for a smooth vector-valued function $g:M\rightarrow\R^\maxdim$, we see a local Pareto optimum as a point $x\in M$ such that for all $y$ in a small enough neighborhood of $x$, we have $g_i(y) \leq g_i(x)$ for at least one $i=1,...,\maxdim$.

	Following this idea, for a \mdm function $f:K\rightarrow\R^\maxdim$, we could require that for two neighboring (i.e. dynamically connected) critical simplices $\sigma,\tau\in\cC$ to belong in the same component, they should be such that $f_i(\sigma)\leq f_i(\tau)$ and $f_j(\sigma)\geq f_j(\tau)$ for some $i,j=1,...,\maxdim$. Since, from Corollary \ref{coro:ConnectionCriticalImpliesGreaterImage}, $\sigma\rightconnects{f}\tau$ implies that $f(\sigma)\succneqq f(\tau)$, it follows that $f_j(\sigma)\geq f_j(\tau)$ is trivial and $f_i(\sigma)\leq f_i(\tau)$ implies $f_i(\sigma) = f_i(\tau)$. Similarly, $\sigma\leftconnects{f}\tau$ implies that $f(\sigma)\precneqq f(\tau)$, so in this case $f_i(\sigma)\leq f_i(\tau)$ is trivially true and $f_j(\sigma) = f_j(\tau)$. This leads to the following relation.

	\begin{prop}\label{prop:CriticalComponents}
		Let $f:K\rightarrow\R^\maxdim$ be \mdm and consider $\cC$, the set of critical points of $f$. Consider the relation $R$ defined on $\cC$ as follows:
		\begin{gather*}
			\sigma R \tau\Leftrightarrow f_i(\sigma) = f_i(\tau)\text{ for some }i = 1,...,\maxdim\text{ and either }\sigma\rightconnects{f}\tau \text{ or } \sigma\leftconnects{f}\tau.
		\end{gather*}
		The transitive closure of $R$ is an equivalence relation.
	\end{prop}

	\begin{proof}
		We easily see that $R$ is both reflexive and symmetric, which makes it straightforward to verify that its transitive closure is an equivalence relation.
	\end{proof}
	\pagebreak
	\begin{defn}[Critical components]\label{def:CriticalComponents}
		A \emph{critical component} of a \mdm function $f:K\rightarrow\R^\maxdim$ is a class of the equivalence relation on $\cC$ defined in Proposition \ref{prop:CriticalComponents}, which we note $\sim$. We use the standards notations regarding equivalence relations, meaning that the partition of $\cC$ in critical components is $\ccomponents$ and the critical component in which some $\sigma\in\cC$ belongs is $[\sigma]\in\ccomponents$.
	\end{defn}

	Here are a few interesting consequences of Definition \ref{def:CriticalComponents}. First, for $\maxdim=1$, so when $f:K\rightarrow\R$ is discrete Morse, $\sigma \sim\tau$ if and only if $\sigma = \tau$. Indeed, when $\sigma,\tau\in\cC$ are such that $\sigma\neq\tau$ and $\sigma\rightconnects{f}\tau$, we know from Corollary \ref{coro:ConnectionCriticalImpliesGreaterImage} that $f(\sigma)>f(\tau)$, and we deduce that $\sigma\nsim\tau$. Hence, the critical components defined by $\sim$ are just the isolated critical points. This agrees with the original smooth and discrete Morse theories.

	Furthermore, using $\sim$, it is possible for $\ccomponents$ to have multiple components even if $\cC$ is connected, as shown in Example \ref{ex:CriticalComponents}\ref{ex:CriticalComponentsTriangle}. Moreover, it is possible to find critical components which are not necessarily connected, as in Example \ref{ex:CriticalComponents}\ref{ex:CriticalComponentsNotConnected}.

	\begin{ex}\label{ex:CriticalComponents}
		Each \mdm function below has two distinct critical components, which are colored in red and orange.
		\begin{enumerate}[label={(\alph{enumi})}]
			\item\label{ex:CriticalComponentsTriangle} In Figure \ref{fig:CriticalComponentsTriangle}, we see the components of the \mdm function considered in Example \ref{ex:ParetoCircle}.
			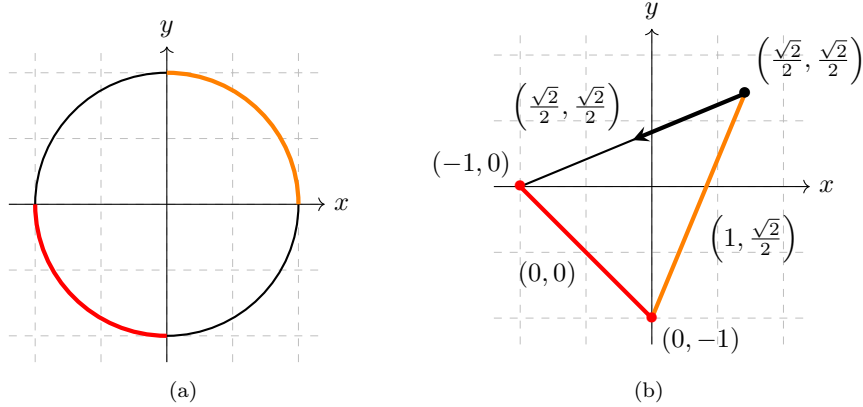
\begin{figure}[h]
				\centering
				\begin{subfigure}{0.48\textwidth}
					\centering
					\begin{tikzpicture}[scale=1.75]
						\draw[help lines, color=gray!50, dashed, step=0.5] (-1.2,-1.2) grid (1.2,1.2);
						\draw[->] (-1.2,0)--(1.2,0) node[right]{$x$};
						\draw[->] (0,-1.2)--(0,1.2) node[above]{$y$};
						\draw[thick] (0,0) circle (1);
						\draw[ultra thick, red] (180:1) arc (180:270:1);
						\draw[ultra thick, orange] (0:1) arc (0:90:1);
					\end{tikzpicture}
					\caption{}
					\label{fig:CriticalComponentsTriangleSmooth}
				\end{subfigure}
				\begin{subfigure}{0.48\textwidth}
					\centering
					\begin{tikzpicture}[scale=1.75]
						\draw[help lines, color=gray!50, dashed, step=0.5] (-1.2,-1.2) grid (1.2,1.2);
						\draw[->] (-1.2,0)--(1.2,0) node[right]{$x$};
						\draw[->] (0,-1.2)--(0,1.2) node[above]{$y$};
						\draw[ultra thick,orange] (0,-1) -- ({cos(45)},{sin(45)});
						\draw[thick] ({cos(45)},{sin(45)})node{$\bullet$} -- (-1,0);
						\draw[ultra thick,red] (-1,0)node{$\bullet$} -- (0,-1)node{$\bullet$};
						\draw[ultra thick, -stealth] (45:1)node{$\bullet$} -- ({(cos(45)+cos(180))/2},{(sin(45)+sin(180))/2});
						\draw (-1,0)node[above left]{$(-1,0)$}
						(0,-1)node[below right]{$(0,-1)$}
						({cos(45)},{sin(45)})node[above right]{$\left(\frac{\sqrt{2}}{2},\frac{\sqrt{2}}{2}\right)$}
						(-0.5,-0.5)node[below left]{$(0,0)$}
						({(cos(45)+cos(270))/2},{(sin(45)+sin(270))/2})node[below right]{$\left(1,\frac{\sqrt{2}}{2}\right)$}
						({(cos(45)+cos(180))/2},{(sin(45)+sin(180))/2})node[above left]{$\left(\frac{\sqrt{2}}{2},\frac{\sqrt{2}}{2}\right)$};
					\end{tikzpicture}
					\caption{}
					\label{fig:CriticalComponentsTriangleDiscrete}
				\end{subfigure}
				\caption{In \subref{fig:CriticalComponentsTriangleSmooth}, the two critical components of $\iota:S^1\rightarrow\R^2$ are shown in red and orange. In \subref{fig:CriticalComponentsTriangleDiscrete}, the corresponding critical components of the \mdm function discretizing $\iota$ from Figure \ref{fig:ParetoCircleDiscrete}}\label{fig:CriticalComponentsTriangle}
			\end{figure}

			\item\label{ex:CriticalComponentsEnum2} In Figure \ref{fig:CriticalComponentsCircle}, we have a \mdm function $f:K\rightarrow\R^2$ such that for each vertex $v\in K$, $f(v)=\iota(v)$ where $\iota:S^1\hookrightarrow\R^2$ is the inclusion on the circle and, for each edge $\sigma=\{v_1,v_2\}$, $f(\sigma) = \left(\max\{f_1(v_1),f_1(v_2)\},\max\{ f_2(v_1),f_2(v_2)\}\right)$.
			\begin{figure}[h]
				\centering
				\begin{tikzpicture}[scale=1.75]
					\draw[help lines, color=gray!50, dashed, step=0.5] (-1.2,-1.2) grid (1.2,1.2);
					\draw[->] (-1.2,0)--(1.2,0) node[right]{$x$};
					\draw[->] (0,-1.2)--(0,1.2) node[above]{$y$};
					\draw[ultra thick, orange] (0:1) -- (30:1)node{$\bullet$} -- (60:1)node{$\bullet$} -- (90:1);
					\draw[thick] (90:1)node{$\bullet$} -- (120:1)node{$\bullet$} -- (150:1)node{$\bullet$} -- (180:1);
					\draw[thick] (270:1) -- (300:1)node{$\bullet$} -- (330:1)node{$\bullet$} -- (360:1)node{$\bullet$};
					\draw[ultra thick, red] (180:1)node{$\bullet$} -- (210:1)node{$\bullet$} -- (240:1)node{$\bullet$} -- (270:1)node{$\bullet$};
					\draw[ultra thick, -stealth] (90:1)node{$\bullet$} -- ({(cos(90)+cos(120))/2},{(sin(90)+sin(120))/2});
					\draw[ultra thick, -stealth] (120:1)node{$\bullet$} -- ({(cos(150)+cos(120))/2},{(sin(150)+sin(120))/2});
					\draw[ultra thick, -stealth] (150:1)node{$\bullet$} -- ({(cos(150)+cos(180))/2},{(sin(150)+sin(180))/2});
					\draw[ultra thick, -stealth] (0:1)node{$\bullet$} -- ({(cos(0)+cos(330))/2},{(sin(0)+sin(330))/2});
					\draw[ultra thick, -stealth] (330:1)node{$\bullet$} -- ({(cos(300)+cos(330))/2},{(sin(300)+sin(330))/2});
					\draw[ultra thick, -stealth] (300:1)node{$\bullet$} -- ({(cos(270)+cos(300))/2},{(sin(270)+sin(300))/2});
				\end{tikzpicture}
				\caption{The gradient field and critical components, in red and orange the \mdm function from Example \ref{ex:CriticalComponents}\ref{ex:CriticalComponentsEnum2}}\label{fig:CriticalComponentsCircle}
			\end{figure}
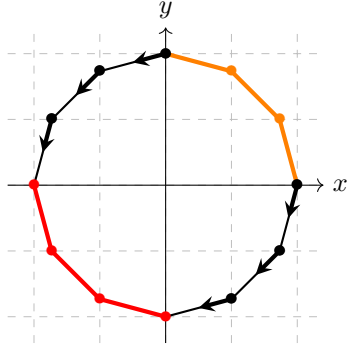

			\item\label{ex:CriticalComponentsNotConnected} We have in Figure \ref{fig:CriticalComponentsMorseSet} the same \mdm function that was considered in Example \ref{ex:MorseSet1}. Notice that the orange critical component is disconnected.

			\begin{figure}[h]
				\centering
				\begin{tikzpicture}[scale=5]
					\fill[orange!30] (0,0) -- (1,0) -- (0.5,0.866) -- cycle;
					\fill[black!10] (1,0) -- (2,0) -- (1.5,0.866) -- (0.5,0.866) -- cycle;
					\draw[orange,ultra thick] (1,0) -- (1.5,0.866);
					\draw (0.5,0.866) -- (0,0) -- (1,0)node{$\bullet$} -- (2,0)node{$\bullet$} -- (1.5,0.866)node{$\bullet$} -- (0.5,0.866)node{$\bullet$} -- (1,0);
					\draw[red] (0,0)node{$\bullet$};
					\draw[ultra thick, -latex] (1,0) -- (0.5,0);
					\draw[ultra thick, -latex] (0.5,0.866) -- (0.25,0.433);
					\draw[ultra thick, -latex] (0.75,0.433) -- (1,0.577);
					\draw[ultra thick, -latex] (1.5,0.866) -- (1,0.866);
					\draw[ultra thick, -latex] (1.5,0) -- (1.5,0.289);
					\draw[ultra thick, -latex] (2,0) -- (1.75,0.433);
					\draw (0,0)node[left]{$(0,0)$}
					(0.25,0.433)node[above left]{$(0,1)$}
					(0.5,0.866)node[above left]{$(0,1)$}
					(0.5,0)node[below]{$(1,1)$}
					(0.5,0.289)node{$(1,5)$}
					(0.75,0.433)node[xshift=3pt, yshift=12pt]{$(1,4)$}
					(1,0)node[below]{$(1,1)$}
					(1,0.577)node[above]{$(1,4)$}
					(1,0.866)node[above]{$(1,2)$}
					(1.25,0.433)node[xshift=-10pt, yshift=5pt]{$(1,3)$}
					(1.5,0)node[below]{$(1,4)$}
					(1.5,0.289)node[above]{$(1,4)$}
					(1.5,0.866)node[above]{$(1,2)$}
					(1.75,0.433)node[right]{$(1,3)$}
					(2,0)node[right]{$(1,3)$};
				\end{tikzpicture}
				\caption{Critical components, in red and orange, of the \mdm function from Example \ref{ex:MorseSet1}. Note that the red component contains only the critical vertex in the lower left part of the complex}\label{fig:CriticalComponentsMorseSet}
			\end{figure}
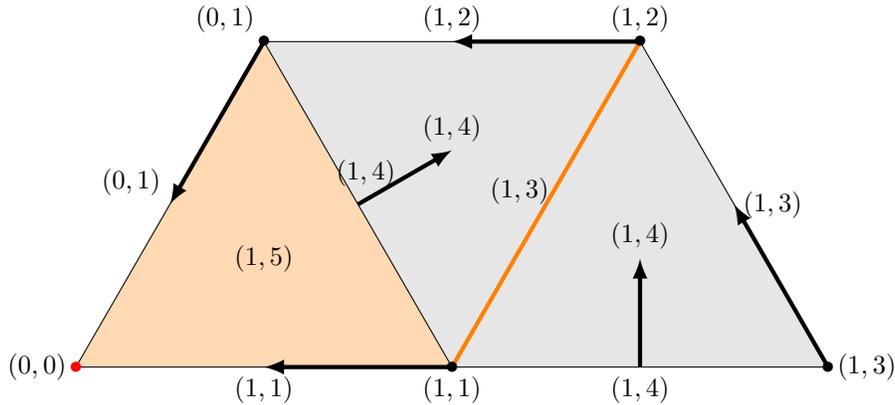

			\item The function in Figure \ref{fig:ComponentsSquare} is \mdm. Notice that there is a connection going from the red component to the orange one and vice versa.
			\begin{figure}[h]
				\centering
				\begin{tikzpicture}[scale=1.75]
					\draw[ultra thick, orange] (2,2) -- (0,2);
					\draw[ultra thick, red] (0,0) -- (2,0);
					\draw[thick] (2,0)node{$\bullet$} -- (2,2) (0,2)node{$\bullet$} -- (0,0);
					\node[orange] at (2,2) {$\bullet$};
					\node[red] at (0,0) {$\bullet$};
					\draw[ultra thick, -stealth] (0,2) -- (0,1);
					\draw[ultra thick, -stealth] (2,0) -- (2,1);
					\node[below left] at (0,0){$(0,2)$};
					\node[below] at (1,0){$(4,2)$};
					\node[below right] at (2,0){$(3,1)$};
					\node[right] at (2,1){$(3,1)$};
					\node[above right] at (2,2){$(2,0)$};
					\node[above] at (1,2){$(2,4)$};
					\node[above left] at (0,2){$(1,3)$};
					\node[left] at (0,1){$(1,3)$};
				\end{tikzpicture}
				\caption{A \mdm function with its critical components represented in red and orange}\label{fig:ComponentsSquare}
			\end{figure}
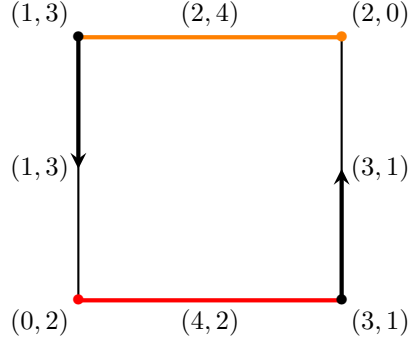
		\end{enumerate}
	\end{ex}

	\pagebreak
	\subsection{Acyclicity of a \mdm function}

	This last example shows that the suggested definition of critical components is not sufficient to induce a Morse decomposition for any \mdm function $f:K\rightarrow\R^\maxdim$. Indeed, if $\cC$ is the set of critical points of $f$, we know that $\cM=\{\{\sigma\}\in\cC\}$ is a Morse decomposition of $\Pi_f$ by Proposition \ref{prop:MinimalMorseDecompositionMDM} but, from Theorem~\ref{theo:ConditionsForMorseDecomposition}, $\cM'=\left\lbrace M([\sigma])\ |\ [\sigma]\in\ccomponents\right\rbrace$ is only a Morse decomposition when there exists no $\cM'$-cycle.

	Using Theorem~\ref{theo:ConditionsForMorseDecomposition}, we establish here some necessary and sufficient conditions for the critical components to induce a Morse decomposition. To do so, we first characterize a $\cM'$-cycle in terms of a \mdm function $f$. This characterization follows directly from Proposition \ref{prop:MpathsCharacterization}.
	\pagebreak
	\begin{prop}\label{prop:fPathsCharacterization}
		Let $f:K\rightarrow\R^\maxdim$ be \mdm, $\cC$ be the set of critical points of $f$. Consider the collection
		\begin{gather*}
			\cM=\left\lbrace M([\sigma])\ |\ [\sigma]\in\ccomponents\right\rbrace.
		\end{gather*}
		A sequence $[\sigma_0],[\sigma_1],...,[\sigma_n]\in\ccomponents$ is a $\cM$-path if and only if there exists a sequence $\tau_0',\tau_1,\tau_1',\tau_2,\tau_2',...,\tau_{n-1}',\tau_n\in\cC$ such that
		\begin{itemize}[label=$\bullet$]
			\item $\tau_0'\in[\sigma_0]$, $\tau_i,\tau_i'\in[\sigma_i]$ for each $i=1,...,n-1$ and $\tau_n\in[\sigma_n]$;
			\item $\tau_{i-1}'\rightconnects{f}\tau_i$ for each $i=1,...,n$.
		\end{itemize}
		Moreover, $[\sigma_0],[\sigma_1],...,[\sigma_n]\in\ccomponents$ is a $\cM$-cycle iff there exists such a sequence $\tau_0',\tau_1,\tau_1',\tau_2,\tau_2',...,\tau_{n-1}',\tau_n\in\cC$ and $[\sigma_0]=[\sigma_n]$.
	\end{prop}

	This leads to the following definition.

	\begin{defn}[$f$-cycle]
		Let $f:K\rightarrow\R^\maxdim$ be \mdm. A \emph{$f$-cycle} is a sequence of critical components $[\sigma_0],[\sigma_1],...,[\sigma_n]=[\sigma_0]\in\ccomponents$ such that, for each $i=1,...,n$, there exists some $\tau'\sim\sigma_{i-1}$ and $\tau\sim\sigma_i$ such that $\tau'\rightconnects{f}\tau$. We say that $f$ is \emph{acyclic} if there exists no $f$-cycle.
	\end{defn}

	The next theorem is a direct consequence of Theorem~\ref{theo:ConditionsForMorseDecomposition} and Proposition~\ref{prop:fPathsCharacterization}.
	\begin{thm}\label{theo:MorseDecompositionfCycle}
		Let $f:K\rightarrow\R^\maxdim$ be \mdm. The collection
		\begin{gather*}
			\cM=\left\lbrace M([\sigma])\ |\ [\sigma]\in\ccomponents\right\rbrace.
		\end{gather*}
		is a Morse decomposition if and only if $f$ is acyclic.
	\end{thm}

	Thus, when a \mdm function is acyclic, its critical components induce another set of Morse equation and inequalities. This follows from Proposition \ref{prop:MorseEquationMorseDecomposition} and Corollary \ref{coro:MorseInequalitiesMorseDecomposition}.

	\begin{thm}\label{theo:MorseInequalitiesCritComponents}
		Let $f:K\rightarrow\R^\maxdim$ be \mdm with $\dim K=n$ and
		\begin{gather*}
			m_p := \sum_{[\sigma]\in\ccomponents}\beta_p(M([\sigma]))
		\end{gather*}
		where $\beta_p(M([\sigma]))$ is the $p^\text{th}$ Conley coefficient of $M([\sigma])$. If $f$ is acyclic, then
		\begin{gather*}
			\sum_{p=0}^nm_pt^p = \sum_{p=0}^n\beta_p(K)t^p + (1+t)Q(t)
		\end{gather*}
		for some polynomial $Q(t)$ with non-negative coefficients. Thus, for all $p=0,1,...,n$, we have
		\begin{enumerate}
			\item strong Morse inequalities:
			\begin{gather*}
				m_p - m_{p-1} + \cdots + (-1)^p m_0 \geq \beta_p(K) - \beta_{p-1}(K) + \cdots (-1)^p \beta_0(K);
			\end{gather*}
			\item weak Morse inequalities:
			\begin{gather*}
				m_p \geq \beta_p(K);
			\end{gather*}
			\item an alternative expression for the Euler characteristic $\chi(K)$ of $K$:
			\begin{gather*}
				m_0 - m_1 + \cdots (-1)^n m_n = \beta_0(K) - \beta_1(K) + \cdots (-1)^n \beta_n(K) = \chi(K).
			\end{gather*}
		\end{enumerate}
	\end{thm}

	Theorem \ref{theo:MorseInequalitiesCritComponents} is analogous to Theorem 6.2 in \citep{Wan1975}. In this work, the author studies the singularities of smooth functions $f:M\rightarrow\R^2$ on a manifold $M$ and, in order to prove the existence of Morse inequalities relating the homology of $M$ to the critical components of $f$, a \emph{no cycle property} needs to be introduced.

	This suggests the acyclicity of a vector-valued Morse function, whether it is discrete or smooth, is an essential property to establish Morse inequalities using the set of critical components of the function. Nonetheless, similar results could seemingly be found using the concept of persistence paths studied by \citet{Budney2021} without assuming acyclicity.

	\section{Conclusion and future work}

	The main achievement of this paper is the use of the framework of combinatorial vector fields and topological dynamics to provide a more complete definition of the multidimensional discrete Morse function (\mdm) and the study of its properties. Building on preliminary results in~\citep{Allili2019} and based on the new framework, a reformulation of the definition of the \mdm has allowed to establish key properties such as the handle decomposition and collapsing theorems and more importantly results on Morse inequalities and Morse decompositions. Moreover, a method of classification of critical cells of \mdm functions into critical components is proposed and conditions for obtaining Morse decompositions and Morse inequalities taking into account the critical components are specified.

	The results above suggest the possibility of undertaking future works in the following directions.

	First, even though we only defined the \mdm theory for simplicial complexes, it would be of interest to extend it to more general complexes or spaces. In particular, many of the results on combinatorial flows used in this article generalize to finite $T_0$-spaces \citep{Lipinski2022}, which are strongly related to simplicial complexes \citep{McCord1966, Stong1966}, thus suggesting the \mdm theory itself could be extended to those spaces.

	Furthermore, the notion of persistence paths, as studied by \citet{Budney2021}, could also be defined in our context and potentially lead to both theoretical and practical results. On the one hand, we could link the critical components of a \mdm function $f$ to the homology of its domain without presuming the acyclicity of $f$ as in Theorem \ref{theo:MorseInequalitiesCritComponents}. On the other hand, this could be used to find an alternative method to simplify the computation of multipersistent homology.

	Experimental results validating this approach are a work in progress.

	%\bibliographystyle{mystyle}
	%\bibliography{Refs.bib}

\end{document}